\documentclass[reqno,10pt]{amsart}
\usepackage{amsfonts}
\usepackage{graphicx}
\usepackage{amsfonts,amsmath, amssymb}
\usepackage{hyperref}
\usepackage{dutchcal}
\usepackage{marginnote}
\usepackage{color}
\usepackage{cite}
\usepackage{bm}
\usepackage{enumerate}
\numberwithin{equation}{section}

\theoremstyle{plain}
\newtheorem{theorem}{Theorem}[section]
\newtheorem{corollary}[theorem]{Corollary}
\newtheorem{lemma}[theorem]{Lemma}
\newtheorem{proposition}[theorem]{Proposition}

\theoremstyle{remark}
\newtheorem{remark}[theorem]{Remark}

\def\a{\alpha}
\def\g{\gamma}

\def\f{\frac}

\def\b{\bar}

\usepackage{tikz}

\newcommand{\dd}{{\rm d}}

\newcommand{\FU}{\mathbf{U}}

\newcommand{\CD}{\mathcal {D}}
\newcommand{\CE}{\mathcal{E}}
\newcommand{\CF}{\mathcal{F}}

\newcommand{\CI}{\mathcal{I}}

\newcommand{\CL}{\mathcal{L}}

\newcommand{\cm}{\mathcal{m}}

\newcommand{\cw}{\mathcal{w}}
\newcommand{\cp}{\mathfrak{p}}
\newcommand{\cu}{\mathfrak{u}}
\newcommand{\cv}{\mathfrak{v}}

\newcommand{\pa}{\partial}

\newcommand{\vep}{\varepsilon}

\newcommand{\eqdef}{\overset{\mbox{\tiny{def}}}{=}}
\newcommand{\dv}{\text{div}_\alpha}

\begin{document}
\date{\today}
\title[Instability of boundary layer in compressible flow]{Linear instability analysis on compressible Navier-Stokes equations with strong boundary layer}

\author[T. Yang]{Tong Yang}
\address[T. Yang]{Department of Mathematics, City University of Hong Kong, Hong Kong, China}
\email{matyang@cityu.edu.hk}

\author[Z. Zhang]{Zhu Zhang}
\address[Z. Zhang]{Department of Mathematics, City University of Hong Kong, Hong Kong, China}
\email{zhuzhangpde@gmail.com}
\begin{abstract}
It is a classical problem in fluid dynamics about the stability and instability of different hydrodynamic patterns in various physical settings, in particular in the high Reynolds number limit of laminar flow with boundary layer. However, 
there are very few mathematical results on the compressible fluid despite the extensive studies when
the fluid is governed by  the
 incompressible Navier-Stokes equations. This paper aims to introduce a new approach to study
 the compressible Navier-Stokes equations  in the subsonic and high Reynolds number regime where a subtle quasi-compressible and Stokes iteration is 
 {developed}. As a byproduct, we show the spectral instability of subsonic boundary layer.
\end{abstract}

\maketitle
\tableofcontents
\section{Introduction}
One of the most fundamental problems in fluid mechanics is to understand the  physical mechanisms that lead to stability or instability of hydrodynamic patterns. With high Reynolds number, most of the laminar flows are unstable  so the small perturbation will eventually cascade into turbulence. Under many circumstances, the early stage of such transition is the  instability induced by viscous disturbance wave, which is now named as  Tollmien-Schlichting or T-S wave. For incompressible flow, the physical description of T-S  waves can be found, for instance in the pioneer work by Heisenberg, C.C. Lin, Tollmien and Schlichting, { cf.} \cite{DR,H,L,SG}, and a formal construction was established by Wasow \cite{W}. Until recently, the rigorous mathematical justification was given by Grenier-Guo-Nguyen \cite{GGN1}.

From physical point of view, it is important to study the compressible flow with boundary layer that
 arises from for instance the flow near the airfoil.
The theoretical investigation can be traced back to Lee-Lin \cite{LL} in which the Rayleigh's criterion for inviscid flow was extended to the compressible subsonic flow. Later on, the asymptotic expansion used in \cite{LL} near the critical layer was rigorously justified by Morawitz \cite{MR}. 
For more investigation from the physical perspectives, we refer to  \cite{DR,L,LR,SG} and the references therein. It is worth noting that the instability mechanisms studied in these literatures are inviscid in nature while the viscous transition mechanisms have not yet been investigated. The aim of this paper is to fill in this gap by  justifying  the presence of T-S wave in compressible boundary layer rigorously.

\subsection{Problem and main result} Consider the 2D compressible Navier-Stokes equations for isentropic flow in  half-space $\{(x,y)\mid x\in \mathbb{T},y\in \mathbb{R}_+\}:$
\begin{equation}\label{1.1}
\left\{\begin{aligned}
&\pa_t\rho+\nabla\cdot(\rho \vec{U})=0,\\
&\rho\pa_t\vec{U}+\rho\vec{U}\cdot\nabla \vec{U}+\nabla P(\rho)=\mu\vep\Delta \vec{U}+\lambda\vep\nabla(\nabla\cdot\vec{U})+\rho\vec{F},\\
&\vec{U}|_{y=0}=\vec{0}.
\end{aligned}\right.
\end{equation}
In above equations $\rho$,  $\vec{U}=(u,v)$ and $P(\rho)$ stand for the density, velocity field and pressure of the fluid. The vector field $\vec{F}$ is a given external force. 
The constants $\mu>0, \lambda\geq 0$ are rescaled shear and bulk viscosity respectively, while $0<\vep\ll1$ is a small parameter  which is proportional to the reciprocal of the Reynolds number. For simplicity and without loss of generality, the constant $\mu$ is set  to 1 throughout the paper. 

A shear flow with boundary layer  is defined by
$$(\rho_s,\vec{\FU}_s)\eqdef(1,U_s(Y),0),~ Y:=\frac{y}{\sqrt{\vep}},~\text{with }U_s(0)=0, \lim_{Y\rightarrow +\infty}U_s(Y)=1.$$
It is a steady solution to \eqref{1.1} with external force $\vec{F}=(-\pa_Y^2U_s,0)$. 

In this paper, we study the compressible Navier-Stokes system linearized around above  shear flow. Denote the Mach number by $\cm:=\frac{1}{\sqrt{P'(1)}}$. The linearization gives
\begin{equation}\label{1.2}
\left\{
\begin{aligned}
&\pa_t\rho+U_s\pa_x\rho+\nabla\cdot \vec{u}=0,~t>0,~(x,y)\in \mathbb{T}\times\mathbb{R}_+,\\
&\pa_t\vec{u}+U_s\pa_x\vec{u}+\cm^{-2}\nabla \rho+v\pa_yU_s\vec{e}_1-{\vep}\Delta\vec{u}-\lambda\vep\nabla (\nabla\cdot \vec{u})-\rho\vec{F}=0,\\
&\vec{u}|_{y=0}=\vec{0}.
\end{aligned}
\right.
\end{equation}
To study \eqref{1.2}, we use the rescaled variable
$$\tau=\frac{t}{\sqrt{\vep}},~X=\frac{x}{\sqrt{\vep}},~Y=\frac{y}{\sqrt{\vep}},
$$
and then look for solution to the linearized compressible Navier-Stokes system in the following form
$$(\tilde{\rho},\tilde{u},\tilde{v})(Y)e^{i\a(X-c\tau)}.
$$
Plugging this ansatz into \eqref{1.2} yields to the following system (we replace $(\tilde{\rho},\tilde{u},\tilde{v})$ by $(\rho,u,v)$ for simplicity of notation)
\begin{equation}\label{1.4}
\left\{
\begin{aligned}
&i\a(U_s-c)\rho+\dv(u,v)=0,\\
&\sqrt{\vep}\Delta_\a u+\lambda i\a\sqrt{\vep}\dv(u,v)-i\a(U_s-c)u-(i\a\cm^{-2}+\sqrt{\vep}\pa_Y^2U_s)\rho-v\pa_YU_s=0,\\
&\sqrt{\vep}\Delta_\a v+\lambda \sqrt{\vep}\pa_Y\dv(u,v)-i\a(U_s-c)v-\cm^{-2}\pa_Y\rho=0,
\end{aligned}
\right.
\end{equation}
with no-slip boundary condition
\begin{align}\label{1.4-1}
u|_{Y=0}=v|_{Y=0}=0.
\end{align}
In \eqref{1.4}, 
$\Delta_\a =(\pa_Y^2-\a^2)$ and	
$\dv(u,v)=i\a u+\pa_Y v$  denote the Fourier transform of Laplacian and divergence operators respectively. For convenience, we denote by $\CL(\rho,u,v)$ the linear operator \eqref{1.4}. If for some $c$ with $\text{Im}c>0$ and wave number $\a\in \mathbb{R}$, the boundary value problem \eqref{1.4} with \eqref{1.4-1} has a non-trivial solution, then the boundary layer profile $(\rho_s,\vec{\FU}_s)$ is spectral unstable. Otherwise, it is spectral stable.

In the following analysis, we focus on  a class of shear flows that satisfy the following assumptions:
\begin{itemize}
	\item $U_s\in C^3(\overline{\mathbb{R}_+})$ and satisfies
	\begin{align}
	U_s(0)=0,~U_s(Y)>0,~\lim_{Y\rightarrow +\infty}U_s(Y)=1,~\text{and } U_s'(0)=1.\label{A0}
	\end{align}
	\item There exist positive constants $s_0$, $s_1$ and $s_2$, such that
	\begin{align}
	\label{A2}
	s_1e^{-s_0Y}\leq \pa_YU_s(Y)\leq s_2e^{-s_0Y}, ~\forall Y\geq 0.	
	\end{align}
	\item  The boundary layer flow is assumed to be {\it  uniformly subsonic}, i.e. $\cm\in (0,1)$. Moreover, there exists a constant $\sigma_1=\sigma_1(\cm)>0$ such that for all $Y>0$, it holds
	\begin{align}\label{A1}
	H(Y)\eqdef\frac{-\pa_Y^2U_s(1-\cm^2U_s^2)-2\cm^2U_s|\pa_YU_s|^2}{|\pa_Y^2U_s|+|\pa_YU_s|^2}\geq \sigma_1.
	\end{align}
	\item There exists a constant $\sigma_2>0$  such that for any $Y>0,$ it holds
	\begin{align}\label{A3}
	\left|\frac{\pa_Y^3U_s}{\pa_Y^2U_s}\right|+\f{|\pa_Y^2U_s|}{\pa_YU_s}+\frac{1-U_s}{\pa_YU_s}\leq \sigma_2.
	\end{align}
\end{itemize}
Note that this class of profiles include the exponential profile $U_s(Y)=1-e^{-Y}$ with $\sigma_1(\cm)=\frac{1-\cm^2}{2}$ and $s_0=s_1=s_2=\sigma_2=1.$ Moreover, from \eqref{A1}, we have
\begin{align}\label{A4}
-\pa_Y^2U_s(Y)\geq \frac{\sigma_1}{1-\cm^2}|\pa_YU_s(Y)|^2,~\forall Y\geq 0.
\end{align}
Hence, the profiles in this class are strictly concave.

The  main result in the paper can be stated  as follows.

\begin{theorem}\label{thm1.1}
	Let the Mach number $\cm\in (0,\f{1}{\sqrt{3}})$. There is $K_0>1$ such that for any sufficiently small $0<\vep\ll 1$ and $\a=K\vep^{\f18}$ with $K\geq K_0$, there exists $c_\vep\in \mathbb{C}$ with $\a\text{Im}c_\vep\approx \vep^{\f14}$, such that the linearized compressible Navier-Stokes system \eqref{1.2} admits a solution $(\rho,u,v)$ in the form of
	\begin{align}\label{sol}
	(\rho,u,v)(t,x,y)=e^{\frac{i\a}{\sqrt{\vep}}(x-c_\vep t)}(\tilde{\rho},\tilde{u},\tilde{v})(Y),~Y:=\frac{y}{\sqrt{\vep}},
	\end{align}
	where the profile $(\tilde{\rho},\tilde{u},\tilde{v}) \in H^1(\mathbb{R}_+)\times \left( H^2(\mathbb{R}_+)\cap H^1_0(\mathbb{R}_+)  \right)^2$ and satisfies \eqref{1.4}.
\end{theorem}
	\begin{remark} In the following we give several remarks on Theorem \ref{thm1.1}.
\begin{itemize}			
	\item[(a)] As shown in the proof, the solution bounds depend on some negative power of $1-\cm$ that
	 are uniform when $\cm\in (0,\f{1}{\sqrt{3}})$. Therefore, by taking the vanishing Mach number limit,
	 we have the Tollmien-Schlichting wave solution for the  incompressible flow that was analyzed 
	 in 
	 \cite{GGN1}  by Grenier-Guo-Nguyen.
 \item[(b)] The growing modes obtained are supported in the frequency regime $n=\frac{\a}{\sqrt{\vep}}=K\vep^{-\f38}$ and grow like $\exp( K^{-\f23}n^{\f23}t)$ with $K\gg1$.  
	These parameters  indicate  a spectral instability in Gevrey
space with index equal to  $\f32$. How to justify the stability in  Gevrey $\f32$ class for the linear problem \eqref{1.2}  that shares the same index as in  the incompressible case studied in \cite{GMM1} will be discussed in our future work.
	
\item[(c)] By constructing a suitable approximate growing mode, we can also obtain the same results as in Theorem \ref{thm1.1} for the wave number $\a=C\vep^{\beta}$, with $C>0$ and $\beta\in (1/12,1/8)$. Here we will not give the detailed analysis in this regime because we focus on the most unstable mode when $\beta=1/8$. We refer to \cite{LYZ}  on the incompressible MHD system for the related discussions and analysis.
\end{itemize}
\end{remark}
\subsection{Relevant literature}

Since this paper is motivated by the study of inviscid limit and  Prandtl boundary layer expansion for incompressible Navier-Stokes equations, we briefly summarize some related works in this direction. Indeed, there are two main destabilizing mechanisms in the boundary layer that makes inviscid limit problem very challenging. The first one is induced by the inflexion points inside the boundary layer profile. In this case, the linearized Navier-Stokes system exhibits a strong ill-posedness below the analyticity regularity, { cf.}\cite{G1,GN}. Therefore, the results of inviscid limit can be obtained only when the data is analytic at least near the boundary, cf. \cite{KVW,M,NN,SC2,WWZ}. The second one is induced by the small disturbance around a monotone and concave boundary layer profile, called Tollmien-Schlichting instability that has been extensively studied in physical literatures and
 was justified rigorously in Grenier-Guo-Nguyen \cite{GGN1} by constructing a growing mode in Gevrey $3/2$ space to linearized incompressible Navier-Stokes equations. The main idea in \cite{GGN1} is to use the stream function and vorticity, that is, the Orr-Sommerfeld (OS) formulation, and then to solve it via an  iteration based on Rayleigh and Airy equation that  can be viewed as the inviscid and viscous approximation  to the original OS equation.
We also refer to Grenier-Nguyen \cite{GN2} for a result of nonlinear instability for small data that depend on viscosity coefficient. On the other hand, this instability result was complemented by the  work of G\'erard-Varet-Maekawa-Masmoudi \cite{GMM1,GMM2} that establishes the Gevrey stability of Euler plus Prandtl expansion with critical Gevrey index $3/2$; see also \cite{CWZ} for a result in $L^\infty$-setting.
Most recently, the formation of boundary layer is studied by Maekawa \cite{M2} using the Rayleigh profile.  For the Sobolev data, the boundary layer expansion is only valid under certain symmetry assumptions or for steady flows, cf.  \cite{GM,GI,IM,LMN,MT} and the references therein. Finally, we refer to \cite{DDGM,DG,GD} for the instability analysis of  boundary layer profile in different settings.


For compressible flow, even though there are many interesting results on
the  Navier-Stokes equations at high Reynolds number in different settings, cf.  \cite{ADM,LW,P,Wang,WW,WXY,XY,ZZZ} and the references therein,  to our knowledge, the stability/instability properties of strong boundary layer for compressible Navier-Stokes equations have not yet been investigated. Compared to the incompressible Navier-Stokes equations, the major difficulty comes from the fact that the Orr-Sommerfeld formulation is no longer available for the compressible case. As a result, Rayleigh-Airy iteration approach used in \cite{GGN1}  can not be applied, either the approaches used in \cite{GMM1,GMM2}. Therefore, the novelty  of this paper is to introduce a new iteration approach  to study compressible flow in the subsonic and high Reynolds number regime.
We believe that the analytic techniques developed in this work can be used in  other related problems for subsonic flows. 

In the next subsection, we present the strategy of the proof for better understanding of the
detailed analysis in the following sections.

\subsection{Strategy of proof}
The instability analysis is based on the following steps.

\underline{Step 1. Construction of the approximate growing mode.} As in the incompressible case \cite{GGN1}, the T-S instability is driven by the interaction of inviscid and viscous perturbations. Set the  approximate growing mode $\vec{\Xi}_{\text{app}}=(\rho_{\text{app}},u_{\text{app}},v_{\text{app}})$ as
\begin{align}\label{t0}
(\rho_{\text{app}},u_{\text{app}},v_{\text{app}})=(\rho_{\text{app}}^s,u^s_{\text{app}},v_{\text{app}}^s)-\frac{v^s_{\text{app}}(0;c)}{v_{\text{app}}^f(0;c)}(\rho_{\text{app}}^f,u^f_{\text{app}},v_{\text{app}}^f).
\end{align}
Here the slow mode $(\rho_{\text{app}}^s,u^s_{\text{app}},v_{\text{app}}^s)$ is an approximate solution to the inviscid system \eqref{2.1.1}, while the fast mode $(\rho_{\text{app}}^f,u^f_{\text{app}},v_{\text{app}}^f)$ is an approximate solution to the full system \eqref{1.4} which exhibits viscous boundary layer structure near $Y=0$. These solutions will
be given in Sec. 2.1 and  2.2 respectively. Note that the approximate solutions defined in \eqref{t0} have zero normal velocity at the boundary, i.e. $v_{\text{app}}^s(0;c)\equiv0.$ Then to recover the  no-slip boundary condition, inspired by \cite{DDGM,DG}, we analyze the zero point of $\CF_{\text{app}}(c)\eqdef u_{\text{app}}^s(0;c)-\frac{v^s_{\text{app}}(0;c)}{v_{\text{app}}^f(0;c)}u_{\text{app}}^f(0;c)$ by applying
 Rouch\'e's Theorem. Precisely, we study the equation $\CF_{\text{app}}(c)=0$ in a family of  $\vep$-dependent domains $D_0$ (see \eqref{2.3.4}). Then by Rouch\'e's Theorem, we can show that $\CF_{\text{app}}(c)$ has the same number of zero points as a linear function $\CF_{\text{ref}}(c)$ defined by \eqref{2.3.3-1}. In addition,  we prove that $|\CF_{\text{app}}(c)|$ has a strictly positive lower bound on $\pa D_0$.

\underline{Step 2. Stability of the approximate growing mode.} Since the approximate
solution $\vec{\Xi}_{\text{app}}$ exhibits the instability already, this step is to show the existence of an exact solution near $\vec{\Xi}_{\text{app}}$.
This is the most difficult and key  step. For this, we study  the following resolvent problem
\begin{equation}\label{t1}
\left\{
\begin{aligned}
&i\a(U_s-c){\rho}+\dv(u,v)=0,\\
&\sqrt{\vep}\Delta_\a u+\lambda i\a\sqrt{\vep}\dv(u,v)- i\a(U_s-c) u-v\pa_YU_s -(i\a\cm^{-2}+\sqrt{\vep}\pa_Y^2U_s)\rho =f_u,\\
&\sqrt{\vep}\Delta_\a v+\lambda\sqrt{\vep}\pa_Y\dv(u,v)-i\a(U_s-c)v-\cm^{-2}\pa_Y\rho=f_v,\\
&v|_{Y=0}=0,
\end{aligned}
\right.
\end{equation}
with a given inhomogeneous source term $(f_u,f_v)$. Here, we emphasize that in \eqref{t1} only normal velocity field $v$ is prescribed on the boundary. Even though we relax the boundary constraint on  $u$ in \eqref{t1}, it is still difficult for existence because the presence of stretching term $v\pa_YU_s$.
This difficulty is overcome by  the following three ingredients.
\begin{itemize}
	\item \underline{Quasi-compressible approximation.} 
When the inhomogeneous source
$(f_u,f_v)\in H^1(\mathbb{R}_+)^2$, we  introduce the following {\it quasi-compressible} system:
\begin{equation}\label{t2}
  \left\{ 
   \begin{aligned}
&i\a(U_s-c)\varrho +i\a \mathfrak {u}+\pa_Y \mathfrak {v}=0,\\
&\sqrt{\vep}\Delta_\a\left[\cu+(U_s-c)\varrho\right]-i\a(U_s-c)\cu-\cv\pa_YU_s-i\a\cm^{-2}\varrho=f_u,\\
&\sqrt{\vep}\Delta_\a \cv-i\a(U_s-c)\cv-\cm^{-2}\pa_Y\varrho=f_v,\\
&\cv|_{Y=0}=0,
   \end{aligned}
  \right. 
\end{equation}
that will be denoted by $L_Q(\varrho,\cu,\cv)=(0,f_u,f_v)$. Note that the  inviscid part of the original linear operator $\CL$ is kept in $L_Q$, while the diffusion terms are modified to be  divergence free. It turns out that
for Mach number $\cm\in (0,1)$, the system \eqref{t2} exhibits a similar stream function-vorticity structure as the incompressible Navier-Stokes equations.  In fact, if we introduce the stream function $\Psi$ associated to the modified velocity variable $(\cu+(U_s-c)\varrho,\cv)$ satisfying 
$$\pa_Y\Psi=\cu+(U_s-c)\varrho,~ -i\a\Psi=\cv,~\Psi|_{Y=0}=0,
$$ 
then \eqref{t2} can be reformulated in terms of $\Psi$ as
\begin{equation}\label{t3}
\begin{aligned}
\text{OS}_{\text{CNS}}(\Psi)&:=\f{i}{n}\Lambda(\Delta_\a \Psi)+(U_s-c)\Lambda(\Psi)-\pa_Y(A^{-1}\pa_YU_s)\Psi\\
&=f_v-\f{1}{i\a}\pa_Y(A^{-1}f_u),
\end{aligned}
\end{equation}
where $n=\a/\sqrt{\vep}$, $A(Y)=1-\cm^2(U_s-c)^2$ and $\Lambda(\Psi)=\pa_Y(A^{-1}\pa_Y\Psi)-\a^2\Psi$. Note that $A(Y)$  is invertible at least for $\cm\in (0,1)$ and $c$ near the origin.  When the Mach number $\cm=0,$ we have $\Lambda= \Delta_\a$ and  $A(Y)\equiv 1$. Thus $\text{OS}_{\text{CNS}}$ is exactly the classical Orr-Sommerfeld operator for incompressible Navier-Stokes system. Therefore, 
 $\text{OS}_{\text{CNS}}$ can be viewed as the compressible counterpart of the  Orr-Sommerfeld equation, which to our best knowledge is for the first time derived in the literatures. This formulation motivates the notion ``quasi-compressible'' approximation. 
 
 We solve \eqref{t3} with   artificial boundary conditions
$\Psi|_{Y=0}=\Lambda(\Psi)|_{Y=0}=0$ that  allows us to obtain the weighted estimates on $\Lambda(\Psi)$. One can see that when $\cm=0$, these boundary conditions are simply the perfect-slip 
boundary conditions used in \cite{CWZ,GM,GMM2} for the study of incompressible Navier-Stokes equations.  However, for the problem considered in this paper, the multiplier  $\cw(Y)=-\pa_Y(A^{-1}\pa_YU_s)$ is not real. Therefore both its leading and first order terms  $\cw_0$, $\cw_1$ (see Lemma \ref{lm3.1} for the precise definitions) play a role in the energy estimates. 
For the bound estimations, we essentially use $\cm\in (0,\f1{\sqrt{3}})$ and the new structural condition \eqref{A1} of the profile in order to show that  the function $\cw_0-U_s\cw_1$ has a positive lower bound,  cf. \eqref{3.1.15-2} in the proof of Lemma \ref{prop3.1}.

After we obtain $\Psi$ that solves \eqref{t3}, the solution $(\varrho,\cu,\cv)$ to \eqref{t2} can be recovered in terms of $\Psi$, cf.  \eqref{3.1.2} and \eqref{3.1.3}. Here we would like to mention that \eqref{t2} has a regularizing effect on density. That is, 
formally by applying $\dv$ to the momentum equation in \eqref{t2} and by noting that the diffusion term lies in the kernel of $\dv$, we have $\Delta_\a\varrho\in L^2(\mathbb{R}_+)$. 
This  reveals an elliptic structure for the linearized compressible Navier-Stokes equations in the subsonic regime.\\

\item \underline{Stokes approximation.}
Note that $(\varrho,\cu,\cv)$ is not an exact solution to \eqref{t1} and its error is 
 \begin{align}\label{t4}
&E_Q(\varrho,\cu,\cv)\eqdef \CL(\varrho,\cu,\cv)-L_Q(\varrho,\cu,\cv)\nonumber\\
&\qquad\quad=\left(0, -\sqrt{\vep}\Delta_\a\left[(U_s-c)\varrho\right]+\lambda \sqrt{\vep}i\a\dv(\cu,\cv)-\sqrt{\vep}\pa_Y^2U_s\varrho,\lambda\sqrt{\vep}\pa_Y\dv(\cu,\cv)\right).
\end{align}
This error term involves a small factor of $\sqrt{\vep}$ but  lies only in $L^2(\mathbb{R}_+)$. This fact prevents us from using the standard fixed point argument to solve \eqref{t1}. To recover the regularity, we introduce another operator $L_S$ that  we call {\it Stokes approximation}. It is obtained from $\CL$ by removing the stretching term, that is, 
$$ L_S(\xi,\phi,\psi)\eqdef \CL(\xi,\phi,\psi)+(0,\psi\pa_YU_s,0).
$$
To eliminate the error $E_Q(\cp,\cu,\cv)$, we then take $(\xi,\phi,\psi)$ as the solution to
$$
L_S(\xi,\phi,\psi)=-E_Q(\varrho,\cu,\cv),~ \pa_Y\phi|_{Y=0}=\psi|_{Y=0}=0.
$$
By using the energy approach in the same spirit as Matsumura-Nishida \cite{MN} and Kawashima \cite{Ka}, we are able to show $(\xi,\phi,\psi)$ is in $H^1(\mathbb{R}_+)\times H^2(\mathbb{R}_+)^2$. Thus, the error term $E_S(\xi,\phi,\psi):=(0,\psi\pa_YU_s,0)$ is in the weighted space $H^2_w(\mathbb{R}_+)$ so that we can treat it as source term of $L_Q$. Therefore, we can  iterate the above two approximations.\\

\item \underline{Quasi-compressible-Stokes iteration.}
Recall that we have the following two  decompositions of the linear operator $\CL:$
$$\CL=L_Q+E_Q=L_S+E_S.
$$
The solvability to \eqref{t1} can be justified via an iteration scheme that is illustrated as follows. Assume that at the $N$-th step we have an approximate solution in the form of $\sum_{j=0}^N\vec{\Xi}_j$ which satisfies
$$
\CL\left(\sum_{j=0}^N\vec{\Xi}_j\right)=(0,f_u,f_v)+\vec{\CE}_N.
$$
Here $\vec{\CE}_N$ is an error term at this step. Provided that $\vec{\CE}_N$ is smooth enough and has zero value at its first component, we can introduce a corrector
$$\vec{\Xi}_{N+1}=-L_Q^{-1}( \vec{\CE}_N )+L_S^{-1}\circ E_Q\circ L_Q^{-1}(\vec{\CE}_N ),
$$
where $L_Q^{-1}$ and $L_S^{-1}$  denote respectively the solution operators to quasi-compressible and Stokes approximate systems. The approximate solution at the $N+1$-step is therefore defined by $\sum_{j=0}^{N+1}\vec{\Xi}_j$. Then we have
$$
\begin{aligned}
\CL\left(\sum_{j=1}^{N+1} \vec{\Xi}_j\right)&=(0,f_u,f_v)+\vec{\CE}_{N+1}\\
&:=(0,f_u,f_v)+E_S\circ L_S^{-1}\circ E_Q \circ L_Q^{-1}(\vec{\CE}_N).
\end{aligned}
$$
A combination of the smallness of $E_Q$, the regularizing effect of $L_S^{-1}$ and the strong decay property of $E_S$ yields  the  contraction in $H^1_w(\mathbb{R}_+)$ of truncated error operator $E_S\circ L_S^{-1}\circ E_Q \circ L_Q^{-1}$ so that the convergence of series $\sum_{j=1}^\infty \vec{\Xi}_j$ in $H^1(\mathbb{R}_+)\times H^2(\mathbb{R}_+)^2$ follows, cf. the proof of Proposition \ref{prop3.0}.
\end{itemize}
\vspace{0.3cm}

\underline{Step. 3. Recovery of the no-slip boundary condition.} We look for  solutions to the original system \eqref{1.4} with $v|_{Y=0}=0$ in the form of $(\rho,u,v)=(\rho_{\text{app}},u_{\text{app}},v_{\text{app}})-(\rho_{R},u_{R},v_{R})$. Here the remainder $(\rho_R,u_R,v_R)$ solves $\CL(\rho_R,u_R,v_R)=E, v_R|_{Y=0}=0$, where $E$ is the error due to the approximation $(\rho_{\text{app}},u_{\text{app}},v_{\text{app}})$. In this step, we need to decompose the error into  regular and smallness parts as in  \eqref{2.4.1}. The remainder is divided accordingly  into $\vec{\Xi}_R=\vec{\Xi}_{\text{re}}+\vec{\Xi}_{\text{sm}}.$ The reason for such  decomposition is that the regular part $E_{\text{re}}$ coming from the rough approximation to Rayleigh equation has a worse bound than the smallness part so that we can
compensate some extra order of $\vep$ by the favorable bounds of $\vec{\Xi}_{\text{re}}$ due to strong decay and $H^1$-regularity of $E_{\text{re}}$, cf.  Proposition \ref{prop3.0}.
Eventually, we can prove that $|u_R(0;c)|\leq C\vep^{\f1{16}}$ on $\pa D_0$, which is smaller than $|u_{\text{app}}(0;c)|$. Then we conclude Theorem \ref{thm1.1}
by Rouch\'e's Theorem.\\

The rest of the paper is organized as follows. In  the next section, we will construct the approximate growing mode. In Section 3, we will show 
the solvability of the linearized system \eqref{1.4} with zero normal velocity condition in order to resolve the remainder due to the approximation. The proof is divided into  several steps. Firstly, two approximate systems, that is, Quasi-compressible and Stokes approximations will be  introduced in Sections 3.1 and 3.2 respectively. Based on these two systems, the iteration scheme will be analyzed in Section 3.3.
The proof of Theorem  \ref{thm1.1} will be given in the final section. In the Appendix, we will
give the proof of the  invertibility of operator $\Lambda$ that  is used in the construction of solution to the equation \eqref{3.1.4-1}.\\

In the paper, for any $z\in \mathbb{C}\setminus \mathbb{R}_-$,  we take the principle analytic branch of $\log z$ and $z^k, k\in (0,1)$, i.e.
$$\log z\triangleq Log|z|+i\text{Arg}z,~ z^k\triangleq|z|^ke^{ik \text{Arg} z},~\text{Arg}z\in (-\pi,\pi].$$

{\bf Notations:} Throughout the paper,  $C$ denotes a generic positive constant and  $C_a$ means that the generic  constant depending on $a$. These constants may vary from line to line.  $A\lesssim B$ and $A=O(1)B$ mean that  there exists a generic constant $C$ such that $A\leq CB$. And $A\lesssim_a B$ implies that the constant $C$ depends on $a$. Similar definitions hold for   $A\gtrsim B$ 
 and $A\gtrsim_a B$.   
  Moreover, we use notation $A\approx B$ if $A\lesssim B$ and $A\gtrsim B$. $\|\cdot\|_{L^2}$  and $\|\cdot\|_{L^\infty}$ 
  denote the standard $L^2(\mathbb{R}_+)$ and    $L^\infty(\mathbb{R}_+)$ norms respectively. 
For any $\eta>0$,  $L^\infty_\eta(\mathbb{R}_+)$ denotes the weighted Lebesgue space 
with the  norm  $\|f\|_{L^\infty_\eta}\triangleq \sup_{Y\in \mathbb{R}_+}\left|e^{\eta Y}f(Y)\right|$. And the weighted Sobolev space $W^{k,\infty}_\eta(\mathbb{R}_+)$ $(k\in \mathbb{N})$ has the norm 
$\|f\|_{W^{k,\infty}_\eta}=\sum_{j\leq k}\|\pa_Y^j f\|_{L^\infty_\eta}$.

\section{Approximate growing mode}

In the following three subsections, we will construct the approximate growing mode that satisfies the no-slip boundary condition.
As for the incompressible Navier-Stokes equations, it is based on the superposition of the slow mode
and the fast mode that represent the interaction of the inviscid and viscous effects near the boundary.

\subsection{Slow mode} 
In this subsection, we will construct the slow mode to capture the inviscid behavior. For this, we consider  the following system denoted by $\CI(\rho,u,v)=\vec{0}$: 
\begin{equation}
\left\{
\begin{aligned}\label{2.1.1}
&i\a(U_s-c)\rho+\dv(u,v)=0,\\
&i\a(U_s-c)u+\cm^{-2}i\a\rho+v\pa_YU_s=0,\\
&i\a(U_s-c)v+\cm^{-2}\pa_Y\rho=0.
\end{aligned}\right.
\end{equation}
By introducing  a new function $\Phi=\f{i}{\a}v$, from $\eqref{2.1.1}_1$, we have
\begin{align}\label{2.1.2}
u=\pa_Y\Phi-(U_s-c)\rho.
\end{align}
Substituting this into $\eqref{2.1.1}_2$ yields
\begin{align}
-\cm^{-2}A(Y)\rho=(U_s-c)\pa_Y\Phi-\Phi\pa_YU_s,\nonumber
\end{align}
where 
\begin{align}\label{2.1.4}
A(Y)\eqdef 1-\cm^2(U_s-c)^2.
\end{align}
Note that for  the {\it uniformly subsonic} boundary layer, i.e. $\cm\in (0,1)$,  when 
$|c|\ll1$, $A(Y)$ is invertible so that  we can represent $\rho$ in terms of $\Phi$ by 
\begin{align}\label{2.1.3}
\rho=-\cm^2A^{-1}(Y)\left[ (U_s-c)\pa_Y\Phi-\Phi\pa_YU_s\right].
\end{align}
Plugging \eqref{2.1.3} into $\eqref{2.1.1}_3$, we derive the following equation for $\Phi$, which can be viewed as an analogy of the classical Rayleigh equation  in the  compressible setting:
\begin{align}\label{2.1.5}
\text{Ray}_{\text{CNS}}\eqdef\pa_Y\left\{ A^{-1}\left[ (U_s-c)\pa_Y\Phi-\Phi\pa_YU_s \right]\right\}-\a^2(U_s-c)\Phi=0.
\end{align}
We remark that the equation \eqref{2.1.5} was firstly derived  by Lee and Lin in \cite{LL} for the study of stability of shear flow in  inviscid fluid. Thus \eqref{2.1.5} is sometimes referred to as Lee-Lin equation. 

The slow mode will be constructed based on an approximate solution to \eqref{2.1.5}. 
Since  \eqref{2.1.5} has similar structure of the Rayleigh equation,  the construction is similar as \cite{GGN1,LYZ} for incompressible flow. In what follows we sketch 
the key steps to make the paper to be self-contained.

Starting from $\a=0$, the equation \eqref{2.1.5} admits following two independent solutions
\begin{align}
\varphi_{+}(Y)=(U_s-c),~\varphi_-(Y)=(U_s-c)\int_1^Y \frac{1}{(U_s(X)-c)^2}\dd X-\cm^2(U_s-c)Y,~\text{for  }\text{Im}c>0.\nonumber
\end{align}
For $\a>0$, to capture the decay property of the solution, we set
\begin{align}
\beta\eqdef \a A_\infty^{\f12},~\text{where }A_\infty=\lim_{Y\rightarrow +\infty}A(Y)= 1-\cm^2(1-c)^2= 1-\cm^2+O(1)|c|,~
\text{for }|c|\ll 1. \label{beta} 
\end{align}
Then we define
\begin{align}
\varphi_{+,\a}(Y)=e^{-\beta Y}\varphi_+(Y),~\varphi_{-,\a}(Y)=e^{-\beta Y}\varphi_-(Y).\label{2.1.6}
\end{align}
Direct computation yields the following error terms:
\begin{align}
\text{Ray}_{\text{CNS}}\left(\varphi_{+,\a}\right)&= -2\beta A^{-2} \pa_YU_s\varphi_{+,\a} +A^{-1}(U_s-c)(\beta^2-\a^2A)\varphi_{+,\a},\label{2.1.7}\\
\text{Ray}_{\text{CNS}}\left(\varphi_{-,\a}\right)&= - 2\beta A^{-2}\pa_YU_s\varphi_{-,\a} +A^{-1}(U_s-c)(\beta^2-\a^2A)\varphi_{-,\a}-2\beta e^{-\beta Y}.\label{2.1.8}
\end{align}
To have  a better approximate solution for \eqref{2.1.5} up to $O(\a^2)$,  the following approximate Green's function is needed:
\begin{equation}
G_\a(X,Y)\eqdef -(U_s(X)-c)^{-1}
\left\{
\begin{aligned}
&e^{-\beta(Y-X)}\varphi_+(Y)\varphi_{-}(X),~X<Y,\\
&e^{-\beta(Y-X)}\varphi_{+}(X)\varphi_{-}(Y),~X>Y.\nonumber
\end{aligned}\right.
\end{equation}
Then we define  a corrector 
\begin{align}\label{2.1.9}
\varphi_{1,\a}(Y)\eqdef 2\int_0^\infty G(X,Y)A^{-2}(X)\pa_YU_s(X)\varphi_{+,\a}(X)\dd X,
\end{align}
and set
\begin{align}\label{2.1.10}
\Phi_{\text{app}}^s(Y;c)\eqdef \varphi_{+,\a}+\beta \varphi_{1,\a}.
\end{align}
Hence, by \eqref{2.1.7} and \eqref{2.1.8}, we have
\begin{align}\label{2.1.11}
\text{Ray}_{\text{CNS}}\left(\Phi_{\text{app}}^s \right)=&-2\beta^2A^{-2}U_s'\varphi_{1,\a}+4\beta^2e^{-\beta Y}\int_Y^\infty A^{-2}(X)U_s'(X)\varphi_{+}(X)\dd X\nonumber\\
&+A^{-1}(U_s-c)(\beta^2-\a^2 A)\Phi_{\text{app}}^s\nonumber\\
=& O(1)\a^2|\pa_YU_s|.
\end{align}
In summary, $\Phi_{\text{app}}^s$ is the slow mode with properties given in the following lemma.

\begin{lemma}\label{lm2.1}
	Let the Mach number $\cm\in (0,1).$ Then  for each $Y\geq0$, $\Phi_{\text{app}}^s(Y;c)$ is holomophic in the upper-half  plane $\{c\in \mathbb{C}\mid\text{Im}c>0\}$. Moreover, there exists $\g_1\in (0,1)$, such that if $\text{Im}c>0$ and $|c|<\g_1$, we have the expansions:
	\begin{align}
	\Phi_{\text{app}}^s(0;c)&=-c+\frac{\a}{(1-\cm^2)^{\f12}}+O(1)\a|c\log\text{Im}c|,\label{2.1.12}\\
	\pa_Y\Phi_{\text{app}}^s(0;c)&=1+O(1)\a|\log\text{Im}c|.\label{2.1.13}
	\end{align} 
\end{lemma}
\begin{proof}
Since $\text{Im}c>0$, $U_s(Y)-c\neq0,  ~\forall Y\geq0$. The analyticity of $\Phi_{\text{app}}^s$ follows from the explicit formula \eqref{2.1.6}, \eqref{2.1.9} and \eqref{2.1.10}. Now we derive the boundary values $\Phi_{\text{app}}^s(0;c) $ and $\pa_Y\Phi_{\text{app}}^s(0;c).$ 

Firstly, note  that
\begin{align}\label{2.1.14}
\varphi_{1,\a}(Y;c)=&-2e^{-\beta Y}\varphi_+(Y)\int_0^Y\varphi_-(X)A^{-2}(X)\pa_YU_s(X)\dd X\nonumber\\
&-2e^{-\beta Y}\varphi_-(Y)\int_Y^\infty\varphi_+(X)A^{-2}(X)\pa_YU_s(X)\dd X.
\end{align}
Then
\begin{align}\label{2.1.15}
\varphi_{1,\a}(0;c)&=-2\varphi_-(0)\int_0^\infty A^{-2}(X)(U_s-c)\pa_YU_s(X)\dd X=-\f{\varphi_-(0)}{\cm^2}\int_0^\infty\f{\dd }{\dd X}(A^{-1})\dd X\nonumber\\
&=-\frac{\varphi_-(0)}{\cm^2}\left( A^{-1}(+\infty)-A^{-1}(0)\right)=\frac{-\varphi_-(0)(1-2c)}{[1-\cm^2(1-c)^2][1-\cm^2c^2]}\nonumber\\
&=-{\varphi_-(0)}\left(\f{1}{1-\cm^2}+O(1)|c|\right),~\text{for }|c|\ll 1.
\end{align}
Then by using \eqref{2.1.10} and the fact that $\beta=\a[(1-\cm^2)^{\f12}+O(1)|c|]$, 
one has
\begin{align}\label{2.1.16}
\Phi_{\text{app}}^s(0;c)=\varphi_{+,\a}(0,c)+\beta\varphi_{1,\a}(0,c)=-c-\a\varphi_-(0)\left(\frac{1}{(1-\cm^2)^{\f12}}+O(1)|c|\right).
\end{align}

To estimate  the boundary value $\pa_Y\Phi_{\text{app}}^s(0;c)$, differentiating \eqref{2.1.14} yields that
\begin{align}
\pa_Y\varphi_{1,\a}(Y;c)=&-\beta\varphi_{1,\a}(Y;c)+U_s'(Y)(U_s-c)^{-1}\varphi_{1,\a}(Y;c)\nonumber\\
&-2e^{-\beta Y}A(Y)(U_s-c)^{-1}\int_Y^\infty \varphi_+(X)A^{-2}(X)\pa_YU_s(X)\dd X.\nonumber
\end{align}
Similar to \eqref{2.1.15}, by using $U_s'(0)=1$, we obtain
\begin{align}
\nonumber
\pa_Y\varphi_{1,\a}(0;c)&=-\beta\varphi_{1,\a}(0;c)-c^{-1}\varphi_{1,\a}(0;c)+2c^{-1}A(0)\int_0^\infty\varphi_+(X)A^{-2}(X)\pa_YU_s(X)\dd X\\
\nonumber&=-\beta\varphi_{1,\a}(0;c)-c^{-1}\varphi_{1,\a}(0;c)+c^{-1}A(0)\left(\frac{1}{1-\cm^2}+O(1)|c|\right)\\
\nonumber&=\frac{1}{c(1-\cm^2)}\left( 1+\varphi_{-}(0)\right)+O(1)(1+|\varphi_-(0)|).
\end{align}
Here we have  used \eqref{2.1.15} in the last identity. Consequently, it holds that
\begin{align}\label{2.1.17}
\pa_Y\Phi_{\text{app}}^s(0;c)&=\pa_Y\varphi_{+,\a}(0;c)+\beta\pa_Y\varphi_{1,\a}(0;c)\nonumber\\
&=1+\frac{\a}{c(1-\cm^2)^{\f12}}(1+\varphi_{-}(0))+O(1)|\a|(1+|\varphi_-(0)|).
\end{align}
Finally, we have $\varphi_-(0)=-1+O(1)|c\log\text{Im}c|,$ cf. Lemma 3.1 in \cite{LYZ}. Then by substituting 
 this into \eqref{2.1.16} and \eqref{2.1.17}, we obtain \eqref{2.1.12} and \eqref{2.1.13}. The proof of Lemma \ref{lm2.1} is completed.
\end{proof}

With $\Phi^s_{\text{app}}$, we can  define the slow mode of fluid quantities $\vec{\Xi}_{\text{app}}^s\eqdef(\rho_{\text{app}}^s,u_{\text{app}}^s,v_{\text{app}}^s)$ 
by using  \eqref{2.1.2} and \eqref{2.1.3} as follows:
\begin{align}
v_{\text{app}}^s=-i\a\Phi_{\text{app}}^s,~\rho_{\text{app}}^s=-\cm^2A^{-1}\left[ (U_s-c)\pa_Y\Phi_{\text{app}}^s-\Phi_{\text{app}}^s\pa_YU_s\right],~u_{\text{app}}^s=\pa_Y\Phi_{\text{app}}^s-(U_s-c)\rho_{\text{app}}^s.\label{2.1.18}
\end{align}
One can directly check that $\vec{\Xi}_{\text{app}}^s$ satisfies
\begin{align}\label{2.1.19}
\CI(\vec{\Xi}_{\text{app}}^s)=(0,0,\text{Ray}_{\text{CNS}}(\Phi_{\text{app}}^s)),
\end{align}
where the error $\text{Ray}_{\text{CNS}}(\Phi_{\text{app}}^s)$ is given in \eqref{2.1.11}. Therefore,  $\vec{\Xi}_{\text{app}}^s$ is an approximate solution to
the  inviscid equation \eqref{2.1.1} up to $O(\a^2)$.

\subsection{Fast mode} To capture the viscous effect of \eqref{1.4} in the boundary layer, 
we need to construct a boundary sublayer corresponding to the fast mode in the approximate
solution. Let $z\eqdef\delta^{-1}Y$. Here $0<\delta\ll 1$ is the scale of boundary sublayer which will be determined later. Now we scale the density and velocity fields in the sublayer by setting
\begin{align}\label{2.2.1-7}
\cp(z)=\rho(Y),~~\mathcal{U}(z)=u(Y),~~\mathcal{V}(z)=(i\a\delta)^{-1}v(Y).
\end{align}
This leads to the following rescaled system associated to \eqref{1.4}:
\begin{align}
&(U_s-c)\cp+\mathcal{U}+\pa_z\mathcal{V}=0,\label{2.2.1-1}\\
&\pa_z^2\mathcal{U}-in\delta^2(U_s-c)\mathcal{U}-in\delta^3U_s'\mathcal{V}-(\cm^{-2}in\delta^2+\delta^2\pa_Y^2U_s)\cp-\a^2\delta^2\left[(1+\lambda)\mathcal{U}+\lambda\pa_z\mathcal{V} \right]=0,\label{2.2.1-2}\\
&\pa_z^2\mathcal{V}-in\delta^2(U_s-c)\mathcal{V}-\a^2\delta^2\mathcal{V}+\lambda\pa_z(\mathcal{U}+\pa_z\mathcal{V})+i\cm^{-2}\vep^{-\f12}\a^{-1}\pa_z\cp=0.\label{2.2.1-3}
\end{align}
Here the constant $n\eqdef\frac{\a}{\sqrt{\vep}}$ which is   the rescaled frequency.
Recalling $U_s'(0)=1,$ we can rewrite $U_s(Y)-c$ as
\begin{align}\label{2.2.1-4}
U_s(Y)-c=U_s'(0)Y-c+[U_s(Y)-U_s'(0)Y]=\delta(z+z_0)+O(1)|\delta|^2|z|^2,
\end{align}
where  $z_0\eqdef -\delta^{-1}c$, and 
\begin{align}
\label{2.2.1-5}
U_s'(Y)=U_s'(0)+U_s'(Y)-U_s'(0)=1+O(1)|\delta||z|.
\end{align}
In view of \eqref{2.2.1-1}-\eqref{2.2.1-5}, it is natural to set 
$$\delta=e^{-\f16\pi i}n^{-\f13},
$$
so that $in\delta^3=1.$  Formally, we have the expansion
\begin{align}\nonumber
\cp=\cp_0+\delta\cp_1+\cdots,~~\mathcal{U}=\mathcal{U}_0+\delta\mathcal{U}_1\cdots,~~\mathcal{V}=\mathcal{V}_0+\delta\mathcal{V}_1\cdots.
\end{align}
Inserting this expansion into \eqref{2.2.1-1}-\eqref{2.2.1-3} and taking the leading order, we can derive the following system for $(\cp_0,\mathcal{U}_0,\mathcal{V}_0)$
\begin{align}
\cp_0(z)=\mathcal{U}_0(z)+\pa_z\mathcal{V}_0(z)&=0, \label{2.2.2}\\
\pa_z^2\mathcal{U}_0(z)-(z+z_0)\mathcal{U}_0(z)-\mathcal{V}_0(z)&=0,\label{2.2.4}
\end{align}
where the variable $z$ lies in the segment $e^{\f{1}{6}\pi i}\mathbb{R}_+$. From \eqref{2.2.2}, we
 observe that the leading order terms of the
 density and divergence of velocity field vanish in the sublayer.
We also require $(\mathcal{U}_0,\mathcal{V}_0)$
to concentrate near the boundary, that is,
\begin{align}
\lim_{z\rightarrow \infty,z\in e^{\f16\pi i}\mathbb{R}_+}(\mathcal{U}_0,\mathcal{V}_0)=\vec{0}.\nonumber
\end{align}
Differentiating \eqref{2.2.4}, by  \eqref{2.2.2}, we obtain
\begin{align}
\pa_z^4\mathcal{V}_0-(z+z_0)\pa_z^2\mathcal{V}_0=0.\label{2.2.6}
\end{align}
Therefore, from \eqref{2.2.2} and \eqref{2.2.6}  we have 
\begin{align}\nonumber
\mathcal{U}_0(z)= -\frac{\text{Ai}(1,z+z_0)}{\text{Ai}(2,z_0)},~\mathcal{V}_0(z)=\frac{\text{Ai}(2,z+z_0)}{\text{Ai}(2,z_0)}.
\end{align}
Here $\text{Ai}(1,z)$ and $\text{Ai}(2,z)$ are respectively the first and second order primitives of the classical Airy function $\text{Ai}(z)$ which is the solution to Airy equation $$\pa_z^2\text{Ai}-z\text{Ai}=0.$$
$\text{Ai}(2,z),$ $\text{Ai}(1,z)$ and $\text{Ai}(z)$ all vanish at infinity along $e^{\f{1}{6}\pi i}\mathbb{R}_+$. They satisfy the relations $\pa_z\text{Ai}(k,z)=\text{Ai}(k-1,z)$, $k=1,2,$ where $\text{Ai}(0,z)\equiv\text{Ai}(z)$. For the detailed construction of these profiles, we refer to \cite{GMM1}. 

Finally, by rescaling the leading order  profile $(\cp_0,\mathcal{U}_0,\mathcal{V}_0)$ via \eqref{2.2.1-7}, we  define the fast mode as 
\begin{align}\label{2.2.1}
\vec{\Xi}_{\text{app}}^f=(\rho_{\text{app}}^f,u_{\text{app}}^f,v_{\text{app}}^f)(Y)\eqdef (0,\mathcal{U}_0,i\a \delta\mathcal{V}_0)(\delta^{-1}Y).
\end{align}
Obviously, 
\begin{align}\label{2.2.8}
u_{\text{app}}^f(0;c)=-\frac{\text{Ai}(1,z_0)}{\text{Ai}(2,z_0)},~v_{\text{app}}^f(0;c)=i\a\delta.
\end{align}

\subsection{Approximate growing mode}
Based on slow and fast modes constructed in the above two subsections, we
are now ready to  construct  an approximate growing mode to \eqref{1.4} with boundary condition \eqref{1.4-1}. Set 
\begin{align}\label{2.3.1}
\vec{\Xi}_{\text{app}}(Y;c)&=(\rho_{\text{app}},u_{\text{app}},v_{\text{app}})(Y;c)\eqdef \vec{\Xi}_{\text{app}}^s(Y;c)-\frac{v^{s}_{\text{app}}(0;c)}{v^{f}_{\text{app}}(0;c)}\vec{\Xi}_{\text{app}}^f(Y;c)\nonumber\\
&= \vec{\Xi}_{\text{app}}^s(Y;c)+\delta^{-1}\Phi_{\text{app}}^s(0;c)\vec{\Xi}_{\text{app}}^f(Y;c),
\end{align}
where  $\vec{\Xi}_{\text{app}}^s$ , $\vec{\Xi}_{\text{app}}^f$  are defined in \eqref{2.1.18}, \eqref{2.2.1} respectively, and the function $\Phi_{\text{app}}^s(Y;c)$ is defined in \eqref{2.1.10} with boundary data satisfying \eqref{2.1.12} and \eqref{2.1.13}. Thanks to \eqref{2.2.8}, the normal velocity $v_{\text{app}}$ satisfies the zero boundary condition, that is,  $v_{\text{app}}(0;c)\equiv0.$ Therefore, the approximate solution \eqref{2.3.1} satisfies the full no-slip boundary condition \eqref{1.4-1} if and only if
 the following function vanishes at some point $c$:
\begin{align}\nonumber
\mathcal{F}_{\text{app}}(c)\eqdef u_{\text{app}}(0;c)=\pa_Y\Phi_{\text{app}}^s(0;c)+c\rho_{\text{app}}^s(0;c)-\delta^{-1}\Phi_{\text{app}}^s(0;c)\frac{\text{Ai}(1,z_0(c))}{\text{Ai}(2,z_0(c))}.
\end{align}

To show  the existence of such $c$ for $\CF_{\text{app}}(c)=0$, we consider the Mach number $\cm\in (0,1)$
and the wave number $\a=K\vep^{\f18}$ with $K\geq 1$ being a large but fixed real number. Set 
\begin{align}\label{2.3.3}
c_0\eqdef \left(\frac{K}{(1-\cm^2)^{\f12}}+K^{-1}(1-\cm^2)^{\f14}e^{\f{1}{4}\pi i}\right)\vep^{\f18}
\end{align}
and define a disk centered at $c_0$ by 
\begin{align}\label{2.3.4}
D_{0}\eqdef\left\{c\in 
\mathbb{C}\mid|c-c_0|\leq K^{-1-\theta}(1-\cm^2)^{\f14}\vep^{\f18}\right\},
\end{align}
with some constant $\theta\in (0,1)$. Clearly, for any $ \cm\in (0,1)$, there exists a positive constant $\tau_0>0$ ($\tau_0\rightarrow 0$ as $\cm\rightarrow 1$), such that for sufficiently large $K$,  the following estimates hold for any $c\in D_0$: \begin{align}\label{c}
\text{Im}c\geq \tau_0K^{-1}\vep^{\f18},~0<\text{arg}c<\tau_0K^{-2},~\text{and }\frac{K(1-\tau_0K^{-2})}{(1-\cm^2)^{\f12}}\vep^{\f18}\leq |c|\leq \frac{K(1+\tau_0K^{-2})}{(1-\cm^2)^{
\f12}}\vep^{\f18}.
\end{align}

With the above preparation, we will prove the following proposition about the existence
of approximate growing mode.

\begin{proposition}\label{prop2.1}
Let $\cm\in (0,1)$.	There exists a positive constant $K_0>1$, such that if $K\geq K_0$, then there exists $\vep_1\in (0,1)$, such that for $\a=K\vep^{\f18}$ and $c\in D_0$ with $\vep\in (0,\vep_1)$, the function $\CF_{\text{app}}(c)$ has a unique zero point in $D_0$. Moreover, on the circle $\pa D_0,$ it holds that
	\begin{align}\label{2.3.5}
	|\mathcal{F}_\text{app}(c)|\geq \frac{1}{2}K^{-\theta}.
	\end{align} 
\end{proposition}
\begin{proof}
	The proof follows the approach used in  Proposition 3.2 in the authors' paper \cite{LYZ} with Liu on the  incompressible MHD system. For  completeness, we sketch the main steps as follows. Firstly,  we take $K_0$ sufficiently large
	so that \eqref{c} holds in the disk $D_0$. Then for $\a=K\vep^{\f18}$ and any $ c\in D_0,$ by  \eqref{2.1.12}, \eqref{2.1.13} and \eqref{c}, we have
	\begin{align}\label{2.3.6}
	\Phi_{\text{app}}^s(0;c)=-c+\f{\a}{(1-\cm^2)^{\f12}}+O(1)\vep^{\f14}|\log\vep|,~\pa_Y\Phi_{\text{app}}^s(0;c)=1+O(1)\vep^{\f18}|\log \vep|.
	\end{align}
	Thus from \eqref{2.1.3} the expression for $\rho_{\text{app}}^s$, one obtains
	\begin{align}\label{2.3.7}
	\rho_{\text{app}}^s(0;c)&=\cm^2A^{-1}(0)\left(c\pa_Y\Phi_{\text{app}}^s(0;c)+U_s'(0)   \Phi_{\text{app}}^s(0;c) \right)	=\frac{\cm^2}{1-\cm^2c^2}\left(\frac{\a}{(1-\cm^2)^{\f12}}  +O(1)\vep^{\f14}|\log\vep|\right )\nonumber\\
	&=O(1)\vep^{\f18}.
	\end{align}
Next, we consider the ratio $\frac{\text{Ai}(1,z_0)}{\text{Ai}(2,z_0)}$. Recall $z_0(c)=-\delta^{-1}c=-e^{\f16\pi i}K^{\f13}\vep^{-\f18}c$.  \eqref{c} implies that
\begin{align}
|z_0|=\frac{K^{\f43}}{(1-\cm^2)^{\f12}}(1+\tau_0K^{-2}),\text{ and }-\f{5}{6}\pi<\text{arg}z_0<-\f{5}{6}\pi+\tau_0K^{-2},~~ \forall c\in D_0.\label{2.3.9}
\end{align}	
Then using the asymptotic behavior of Airy profile (e.g. cf. \cite{GMM1}) and by  \eqref{2.3.9}, we  obtain 
\begin{align}\label{2.3.11}
\frac{\text{Ai}(1,z_0)}{\text{Ai}(2,z_0)}=-z_0^{\f12}+O(1)|z_0|^{-1}=-\frac{K^{\f23}}{(1-\cm^2)^{\f14}}e^{-\frac{5}{12}\pi i}+O(1)K^{-\f43},~~K\gg 1.
\end{align}

Now we set
\begin{align}
\mathcal{F}_{\text{ref}}(c)\eqdef 1+e^{-\f{1}{4}\pi i}K(1-\cm^2)^{-\f14}\vep^{-\f18}\left( -c+\frac{\a}{(1-\cm^2)^{\f{1}{2}}}  \right)\label{2.3.3-1}.
\end{align}
On one hand, there exists a unique zero point $c_0$  in \eqref{2.3.3} to the mapping $\mathcal{F}_{\text{ref}}(c)$. And on the boundary $\pa D_0$, it holds
\begin{align}\label{2.3.12}
|\mathcal{F}_{\text{ref}}(c)|=K^{-\theta}.
\end{align}
On the other hand, we can show that $\mathcal{F}_{\text{ref}}(c)$ is the leading order of $\mathcal{F}_{\text{app}}(c)$. In fact, by  \eqref{2.3.6}, \eqref{2.3.7} and \eqref{2.3.11}, 
we have the following estimate on the difference:
\begin{align}\nonumber
\left| \mathcal{F}_{\text{app}}(c)-\mathcal{F}_{\text{ref}}(c)\right|&\leq\left| 1+e^{\f{1}{6}\pi i}K^{\f13}\vep^{-\f18}\left(c-\f{\a}{(1-\cm^2)^{\f12}} \right)\f{\text{Ai}(1,z_0)}{\text{Ai}(2,z_0)}-\mathcal{F}_{\text{ref}}(c)\right|+C_{K,\cm}\vep^{\f18}|\log \vep|\\
&\leq \left| 1+e^{\f{1}{6}\pi i}K^{\f13}\vep^{-\f18}\left(c-\f{\a}{(1-\cm^2)^{\f12}} \right)z_0^{\f12}-\mathcal{F}_{\text{ref}}(c)\right|+C_{\cm}K^{-2}+C_{K,\cm}\vep^{\f18}|\log \vep|\nonumber\\
&\leq C_{\cm}K^{-2}+C_{K,\cm}\vep^{\f18}|\log\vep|,\nonumber
\end{align}
where the positive constants $C_\cm$ is independent of $K$ and $\vep$ and $C_{K,\cm}$  depends on $K$ and Mach number $\cm$, but not on $\vep$. Now we take $K_0$  larger if needed and then take $\vep_1\in (0,1)$ suitably small such that for $\vep\in (0,\vep_1)$ and $K\geq K_0$, it holds that
$$C_\cm K^{-2}+C_{K,\cm}\vep^{\f18}|\log\vep|<\f12K^{-\theta}.
$$
Consequently, one obtains
$$|\mathcal{F}_{\text{app}}(c)-\mathcal{F}_{\text{ref}}(c)|\leq \f12|\mathcal{F}_{\text{ref}}(c)|,~\forall c\in \pa D_0.
$$
Combining this with \eqref{2.3.12} yields \eqref{2.3.5}. Moreover,
since $\text{Ai}(1,z)$ and $\text{Ai}(2,z)$ are analytic functions and  $\text{Ai}(2,z_0)\neq0$ due to \eqref{2.3.11}, both $\mathcal{F}_{\text{app}}(c)$ and $\mathcal{F}_{\text{ref}}(c)$ are analytic in $D_0$.  Therefore, by Rouch\'e's Theorem, $\mathcal{F}_{\text{app}}(c)$ and $\mathcal{F}_{\text{ref}}(c)$ have the same number of zeros in $D_0$. The proof of Proposition \ref{prop2.1} is then completed.
\end{proof}	

We now conclude this subsection by summarizing the relations between the parameters $n$ (rescaled frequency), $\delta$ (scale of sublayer), $\a$ (wave number) and $\vep$ that will be used frequently later:
$$n=\frac{\a}{\sqrt{\vep}};
	\text{ and } \delta=e^{-\f16\pi i}n^{-\f13}.
$$
If in particular $\a=K\vep^{\f18}$ and $c\in D_0$, then 
\begin{align}
\a\approx |c|\approx \text{Im}c\approx |\delta|\approx n^{-\f13}\approx \vep^{\f18},\label{para}
\end{align}
where the relations may depend on $K$ but not on $\vep$.

\subsection{Estimates on error terms}

In this subsection, we will give the detailed estimate on  the error of the approximate solution  \eqref{2.3.1} by using a decomposition that takes the decay and regularity in $Y$  into consideration.

{\it \underline{``Regular+Smallness'' decomposition:}} Precisely, the approximate solution $\vec{\Xi}_{\text{app}}$ to  \eqref{1.4} has the following error representation:
\begin{align}\label{2.4.1}
\CL(\vec{\Xi}_{\text{app}})=(0,0,E_{v,\text{re}})+(0,E_{u,\text{sm}},E_{v,\text{sm}}).
\end{align}
Here the regular part
\begin{align}\label{2.4.2}
E_{v,\text{re}}=\text{Ray}_{\text{CNS}}(\Phi_{\text{app}}^s)
\end{align}
with  $\text{Ray}_{\text{CNS}}(\Phi_{\text{app}}^s)$  defined in \eqref{2.1.11}. Observe that $E_{v,\text{re}}$ has strong decay in $Y$ due to the background boundary layer profile. And
the smallness part reads
\begin{equation}
\begin{aligned}\label{2.4.3}
E_{u,\text{sm}}=&\sqrt{\vep}\Delta_{\a}u_{\text{app}}^s+\lambda i\a\sqrt{\vep}\dv(u_{\text{app}}^s,v_{\text{app}}^s)-\sqrt{\vep}\pa_Y^2U_s\rho_{\text{app}}^s+\eta\sqrt{\vep}\a^2u_{\text{app}}^f\\
&-i\a\eta(U_s(Y)-U_s'(0)Y){u}_{\text{app}}^f-\eta v_{\text{app}}^f(\pa_YU_s(Y)-\pa_YU_s(0)),\\
E_{v,\text{sm}}=&\sqrt{\vep}\Delta_{\a}v_{\text{app}}^s+\lambda\sqrt{\vep}\pa_Y\dv(u_{\text{app}}^s,v_{\text{app}}^s)+\eta\sqrt{\vep}\Delta_\a v_{\text{app}}^f-i\a\eta(U_s-c)v_{\text{app}}^f,
\end{aligned}
\end{equation}
where  $\eta\eqdef\delta^{-1}\Phi_{\text{app}}^s(0;c)$. In the following, 
we will show that  $E_{u,\text{sm}}$ and $E_{v,\text{sm}}$ are of higher order in $\vep$ than $E_{v,\text{re}}$.\\

The estimates on these error terms are summarized in the following proposition. Let us first define some weighted Sobolev spaces for later use. 
\begin{align}
L^2_w(\mathbb{R}_+)&\eqdef\left\{f\in L^2(\mathbb{R}_+)\bigg|~\|f\|_{L^2_w}\eqdef\||\pa_Y^2U_s|^{-\f12}f\|_{L^2}<\infty\right\},\nonumber\\
H^N_w(\mathbb{R}_+)&\eqdef\left\{f\in H^N(\mathbb{R}_+)\bigg|~\|f\|_{H^N_w}\eqdef\sum_{j=0}^N\|\pa_Y^jf\|_{L^2_w}<\infty\},~N\text{ is a positive integer}\right\}.\label{space}
\end{align}
Recall $K_0\geq 1$ and $\vep_1\in (0,1)$ are constants given in Proposition \ref{prop2.1}. For $K\geq K_0$ and $\vep\in (0,\vep_1),$ the disk $D_0$ is defined in \eqref{2.3.4}. The following proposition gives
the precise error bound estimates.

\begin{proposition}\label{prop2.2}
Let the Mach number $\cm\in (0,1)$. There exists $\vep_2\in (0,\vep_1)$, such that for $\vep\in (0,\vep_2)$, $\a=K\vep^{\f18}$ and $c\in D_0$, the error terms $E_{v,\text{re}}$, $E_{u,\text{sm}}$ and $E_{v,\text{sm}}$ satisfy the estimates
	\begin{align}\label{2.4.4}
\|E_{v,\text{re}}(\cdot~;c)\|_{H^1_w}&\lesssim_{K} \vep^{\f{3}{16}},\\
	\|E_{u,\text{sm}}(\cdot~;c)\|_{L^2}+\|E_{v,\text{sm}}(\cdot~;c)\|_{L^2}&\lesssim_K \vep^{\f{7}{16}},\label{2.4.5}
	\end{align}
where the constant is uniform in $\vep$.
\end{proposition}
The proof of Proposition \ref{prop2.2} follows from a series of estimations on the approximate solution. First of all, we show some properties of corrector $\varphi_{1,\a}$ and approximate solutions $\Phi_{\text{app}}^s$ to Rayleigh operator that are defined in \eqref{2.1.9} and \eqref{2.1.10} respectively. Fix $\cm\in (0,1)$ and  set ${\beta}_1\eqdef\frac{1}{2}(1-\cm^2)^{\f12}\a$. 

\begin{lemma}\label{lm2.3}
Let $\gamma_1$ be the constant given in Lemma \ref{lm2.1}. There exists $\gamma_2\in (0,\gamma_1)$, such that for $c\in \{c\in \mathbb{C} \mid \text{Im}c>0,~|c|< \gamma_2 \}$ and $\a\in (0,1)$, $\varphi_{1,\a}$ satisfies 
\begin{align}
&\|\varphi_{1,\a}\|_{L^\infty_{\beta_1}}+\a^{\f12}\|\varphi_{1,\a}\|_{L^2}\lesssim 1,\label{2.4.8}\\
&\|\pa_Y\varphi_{1,\a}\|_{L^\infty_{\beta_1}}\lesssim 1+|\log\text{Im}c|,~\|\pa_Y\varphi_{1,\a}\|_{L^2}\lesssim 1,\label{2.4.8-2}\\
&\|\pa_Y^2\varphi_{1,\a}\|_{L^\infty_{\beta_1}}+(\text{Im}c)^{-\f12}\|\pa_Y^2\varphi_{1,\a}\|_{L^2}\lesssim (\text{Im}c)^{-1},\label{2.4.8-3}\\
&\|\pa_Y^3\varphi_{1,\a}\|_{L^\infty_{\beta_1}}+(\text{Im}c)^{-\f12}\|\pa_Y^3\varphi_{1,\a}\|_{L^2}\lesssim (\text{Im}c)^{-2}.\label{2.4.8-4}
\end{align}
Moreover, if in addition $\a=K\vep^{\f18}$ and $c\in D_0$, we have 
\begin{align}\label{2.4.8-1}
\a^{\f12}\|\Phi_{\text{app}}^s\|_{L^2}+\|\pa_Y\Phi_{\text{app}}^s\|_{H^1}+\|\Phi_{\text{app}}^s\|_{W^{2,\infty}_{\beta_1}}+(\text{Im}c)^{\f12}\|\pa_Y^3\Phi^s_{\text{app}}\|_{L^2}+\text{Im}c\|\pa_Y^3\Phi_{\text{app}}^s\|_{L^\infty_{\beta_1}}\lesssim 1.
\end{align}
\end{lemma}
\begin{proof}
Recall that $\beta=\a\left[(1-\cm^2)^{\f12}+O(1)|c|\right]$ from \eqref{beta} and $A(Y)=1-\cm^2U_s^2+O(1)|c|$ from \eqref{2.1.4}. Taking $\gamma_2\in(0,\gamma_1)$ suitably small, we have
$\text{Re}\beta>\beta_1$ and $|A^{-1}|\leq \frac{1}{2(1-\cm^2)}\lesssim1$ for $|c|<\gamma_2$.
Then the proof of \eqref{2.4.8}-\eqref{2.4.8-4} follows from an argument
exactly as in Lemma 3.6 in \cite{LYZ} by using the explicit expression  \eqref{2.1.14}. Hence, we omit it for brevity.

For \eqref{2.4.8-1}, we recall \eqref{2.1.10} and observe that
\begin{align}\label{2.4.8-5}
\a^{\f12}\|\varphi_{+,\a}\|_{L^2}+\|\pa_Y\varphi_{+,\a}\|_{H^2}+\|\varphi_{+,\a}\|_{W^{3,\infty}_{\beta_1}}\lesssim 1.
\end{align} 
By \eqref{para}, we have $\a/\text{Im}c\lesssim1$. Thus putting \eqref{2.4.8}-\eqref{2.4.8-4} and \eqref{2.4.8-5} together yields the desired estimate \eqref{2.4.8-1}. The proof of the lemma is 
then completed.
\end{proof}

By  Lemma \ref{lm2.3}, we can immediately obtain the following estimates on the slow mode $\vec{\Xi}_{\text{app}}^s$  given in \eqref{2.1.18}.
\begin{corollary}\label{cor2.3}
If $\a=K\vep^{\f18}$ and $c\in D_0\cap\{c\in \mathbb{C}\mid|c|< \gamma_2\}$, the slow mode $\vec{\Xi}_{\text{app}}^s$ satisfies the following estimates:
\begin{align}\label{2.4.11}
\cm^{-2}\|\rho_{\text{app}}^s\|_{H^1}+\|u_{\text{app}}^s\|_{H^1}+\|v_{\text{app}}^s\|_{H^2}&\lesssim 1,\\
\cm^{-2}\|\pa_Y^2\rho_{\text{app}}^s\|_{L^2}+\|\pa_Y^2u_{\text{app}}^s\|_{L^2}&\lesssim (\text{Im}c)^{-\f12}.\label{2.4.12}
\end{align}
\end{corollary}
\begin{proof}
The estimate on $v_{\text{app}}^s$ follows from \eqref{2.4.8-1} directly. For $\rho_{\text{app}}^s$, using \eqref{2.4.8-1} gives
$$\cm^{-2}\|\rho_{\text{app}}^s\|_{L^2}\lesssim \|\pa_Y\Phi_{\text{app}}^s\|_{L^2}+\|\pa_YU_s\|_{L^2}\|\Phi_{\text{app}}^s\|_{L^\infty}\lesssim1.
$$
Differentiating \eqref{2.1.3} with respect to $Y$ yields the formulas 
$$\cm^{-2}\pa_Y\rho_{\text{app}}^s=-\cm^{-2}\rho_{\text{app}}^sA^{-1}A'-A^{-1}(U_s-c)\pa_Y^2\Phi_{\text{app}}^s+\Phi_{\text{app}}^sA^{-1}\pa_Y^2U_s,
$$
and 
$$\cm^{-2}\pa_Y^2\rho_{\text{app}}^s=-2\cm^{-2}A^{-1} \pa_YA\pa_Y\rho_{\text{app}}^s-\cm^{-2}A^{-1}\pa_Y^2A\rho_{\text{app}}^s-A^{-1}\pa_Y\left[(U_s-c)\pa_Y^2\Phi_{\text{app}}^s-\Phi_{\text{app}}^s\pa_Y^2U_s\right].
$$
Taking $L^2$-norm and by  \eqref{2.4.8-1}, we obtain
$$\cm^{-2}\|\pa_Y\rho_{\text{app}}^s\|_{L^2}\lesssim \cm^{-2}\|\rho_{\text{app}}^s\|_{L^2}+\|\pa_Y^2\Phi_{\text{app}}^s\|_{L^2}+\|\Phi_{\text{app}}^s\|_{L^\infty}\|\pa_Y^2U_s\|_{L^2}\lesssim 1,
$$
and
$$\cm^{-2}\|\pa_Y^2\rho_{\text{app}}^s\|_{L^2}\lesssim \cm^{-2}\|(\rho_{\text{app}}^s,\pa_Y\rho_{\text{app}}^s)\|_{L^2}+\|\pa_Y\Phi_{\text{app}}^s\|_{H^2}+\|\pa_Y^3U_s\|_{L^2}\|\Phi_{\text{app}}^s\|_{L^\infty}\lesssim (\text{Im}c)^{-\f12}.
$$
The velocity field $u_{\text{app}}^s$ can be estimated in the same way so that  we omit the details. And
this completes the proof of the corollary.
\end{proof}
The following lemma gives some pointwise estimates on the  fast mode $(u_{\text{app}}^f,v_{\text{app}}^f)$ defined in \eqref{2.2.1}. The proof follows from Lemma 3.9 in \cite{LYZ} by using the pointwise estimate of Airy profiles. Thus,  we omit the details for brevity.
\begin{lemma}\label{lm2.4}
The fast mode $(u_{\text{app}}^s,v_{\text{app}}^s)$ has the following pointwise estimates:
\begin{align}
\left| \pa^k_Y u_{\text{app}}^f(Y;c)\right|&\lesssim n^{\f{k}{3}}e^{-\tau_1n^{\f13}Y},~k=0,1,2\label{2.4.9}\\
\left| \pa^k_Y v_{\text{app}}^f(Y;c)\right|&\lesssim n^{\f{k-2}{3}}e^{-\tau_1n^{\f13}Y},~k=0,1,2,\label{2.4.10}
\end{align} 
for some constant $\tau_1>0$ which does not depend on $n$.
\end{lemma}
With the above estimates, we are now ready to prove Proposition \ref{prop2.2} as follows. \\

{\bf Proof of Proposition \ref{prop2.2}:} We start with proving  \eqref{2.4.4} for  $E_{v,\text{re}}$. Recall the definition \eqref{2.4.2} and explicit formula \eqref{2.1.11}. By taking $\vep_2>0$ suitably small, such that $D_0\subset\{\text{Im}c>0, |c|\leq \gamma_2\}$, we have $\text{Re}\beta>\beta_1$ and $|A|\gtrsim1$. Then thanks to \eqref{A3}, \eqref{beta}, the bounds \eqref{2.4.8}-\eqref{2.4.8-1} and the fact that 
$$|\beta^2-\a^2 A|\lesssim \a^2|A(+\infty)-A(Y)|\lesssim \a^2|1-U_s(Y)|,
$$
we have
\begin{align}
|E_{v,\text{re}}(Y)|&\lesssim \a^2 \pa_YU_s(Y)e^{-\beta_1 Y}\|\varphi_{1,\a}\|_{L^\infty_{\beta_1}}+\a^2e^{-\beta_1 Y}(1-U_s(Y))+|\beta^2-\a^2A|e^{-\beta_1Y}\|\Phi_{\text{app}}^s\|_{L^\infty_{\beta_1}}\nonumber\\
&\lesssim \a^2\pa_YU_s(Y)e^{-\beta_{1}Y}\|\varphi_{1,\a}\|_{L^\infty_{\beta_1}}+\a^2e^{-\beta_1 Y}(1-U_s(Y))(1+\|\Phi_{\text{app}}^s\|_{L^\infty_{\beta_1}})\nonumber\\
&\lesssim \a^2e^{-\beta_1 Y}\pa_YU_s(Y).\nonumber
\end{align} 
This gives the following weighted estimate
\begin{align}\label{2.4.13}
\|E_{v,\text{re}}\|_{L^2_w}\lesssim \a^2\||\pa_Y^2U_s|^{-\f12}\pa_YU_s\|_{L^\infty}\|e^{-\beta_1Y}\|_{L^2}\lesssim \a^{\f32}\lesssim \vep^{\f3{16}},
\end{align}
by using the concavity \eqref{A4}. Moreover, differentiating \eqref{2.1.11} 
yields 
$$
|\pa_YE_{v,\text{re}}|\lesssim \a^2\left( |\pa_YU_s|^2+|\pa_Y^2U_s|   \right)|\varphi_{1,\a}|+\a^2|\pa_YU_s|\left( e^{-\beta_1 Y}+|\pa_Y\varphi_{1,\a}|+|\Phi_{\text{app}}^s|+|\pa_Y\Phi_{\text{app}}^s|    \right).
$$
With this, 
by the  bounds \eqref{2.4.8}-\eqref{2.4.8-1} and the concavity \eqref{A4}, we obtain
\begin{align}
\|\pa_YE_{v,\text{re}}\|_{L^2_w}&\lesssim \a^2\|e^{-\beta_1Y}\|_{L^2}\left( 1+\||\pa_Y^2U_s|^{-\f12}\pa_YU_s\|_{L^\infty}\right)\left(1+ \|\varphi_{1,\a}\|_{W^{1,\infty}_{\beta_1}}+\|\Phi_{\text{app}}^s\|_{W^{1,\infty}_{\beta_1}} \right)\nonumber\\
&\lesssim \a^{\f{3}{2}}\lesssim \vep^{\f{3}{16}}.\label{2.4.14}
\end{align}
Putting \eqref{2.4.13} and \eqref{2.4.14} together yields the  estimate \eqref{2.4.4} on part of the error
with decay.

Now we turn to estimate the part of error with smallness $(E_{u,\text{sm}},E_{v,\text{sm}})$ which is defined in \eqref{2.4.3}. Keeping in mind the bounds of parameters \eqref{para} and 
$$|\eta|=|\delta|^{-1}|\Phi_{\text{app}}^s(0;c)|\lesssim_K 1, \forall c\in D_0$$
because of  \eqref{2.1.12}. Note that also $|U_s(Y)-U_s'(0)Y|\lesssim Y^2, |\pa_YU_s(Y)-\pa_YU_s(0)|\lesssim Y$ and $$\|\dv(u_{\text{app}}^s,v_{\text{app}}^s)\|_{H^1}\leq C\|u_{\text{app}}^s\|_{H^1}+C\|v_{\text{app}}^s\|_{H^2}.$$ 
By the bounds on $(u_{\text{app}}^s,v_{\text{app}}^s)$ and $(u_{\text{app}}^f,v_{\text{app}}^f)$ given in  Corollary \ref{cor2.3} and Lemma \ref{lm2.4}, we have
\begin{align}
\|E_{u,\text{sm}}\|_{L^2}&\lesssim \sqrt{\vep}(\|u_{\text{app}}^s\|_{H^2}+\|v_{\text{app}}^s\|_{H^2}+\|\rho_{\text{app}}^s\|_{L^2})+\sqrt{\vep}\a^2\|u_{\text{app}}^f\|_{L^2}+\a\|Y^2u_{\text{app}}^f\|_{L^2}+\|Yv_{\text{app}}^f\|_{L^2}\nonumber\\
&\lesssim \sqrt{\vep}(1+|\text{Im}c|^{-\f12})+\sqrt{\vep}\a^2\|e^{-\tau_1n^{\f13} Y}\|_{L^2}+\a\|Y^2e^{-\tau_1n^{\f13}Y}\|_{L^2}+n^{-\f23}\|Ye^{-\tau_1n^{\f13}Y}\|_{L^2}\nonumber\\
&\lesssim \sqrt{\vep}(1+|\text{Im}c|^{-\f12})+\sqrt{\vep}\a^2n^{-\f16}+\a n^{-\f56}+n^{-\f76}\lesssim_K \vep^{\f7{16}},\label{2.4.15}
\end{align}
and
\begin{align}
\|E_{v,\text{sm}}\|_{L^2}&\lesssim \sqrt{\vep}(\|v_{\text{app}}^s\|_{H^2}+\|u_{\text{app}}^s\|_{H^1})+\sqrt{\vep}\|v_{\text{app}}^f\|_{H^2}+\a\left( |c|\|v_{\text{app}}^f\|_{L^2}+\|Yv_{\text{app}}^f\|_{L^2}\|Y^{-1}U_s\|_{L^\infty}  \right)\nonumber\\
&\lesssim \sqrt{\vep}+\sqrt{\vep}\|e^{-\tau_1n^{\f13}Y}\|_{L^2}+\a n^{-\f23}\left( |c|\|e^{-\tau_1n^{\f13}Y}\|_{L^2}+\|Ye^{-\tau_1n^{\f13}Y}\|_{L^2} \right)\nonumber\\
&\lesssim \sqrt{\vep}(1+n^{-\f16})+\a|c| n^{-\f56}+\a n^{-\f76}\lesssim_K \vep^{\f12}.\label{2.4.16}
\end{align}
Estimates \eqref{2.4.15} and \eqref{2.4.16} give the bound \eqref{2.4.5} for  $(E_{u,\text{sm}},E_{v,\text{sm}})$. Then the proof of the proposition is completed.  \qed

\section{Solvability of remainder system}
In this section, we will construct a solution to the resolvent problem 
\begin{equation}\label{3.0.1}
\left\{
\begin{aligned}
&i\a(U_s-c)\rho+\dv(u,v)=0,\\
&\sqrt{\vep}\Delta_\alpha u
+\lambda i\a\sqrt{\vep}\dv(u,v)-i\a(U_s-c)u-v\pa_YU_s- (\cm^{-2} i\a+\sqrt{\vep}\pa_Y^2U_s) \rho=f_u,\\
&\sqrt{\vep}\Delta_\a v+\lambda\sqrt{\vep} \pa_Y\dv(u,v)-i\a(U_s-c)v-\cm^{-2} \pa_Y\rho=f_v,\\
&v|_{Y=0}=0,
\end{aligned}
\right.
\end{equation}
where $(f_u,f_v)$ is a given inhomogeneous source term. If $(f_u,f_v)\in H^1(\mathbb{R}_+)^2$, we  define the operator
\begin{align}
\label{add3.1.5}
\Omega(f_u,f_v)\eqdef f_v-\frac{1}{i\a}\pa_Y(A^{-1}f_u).
\end{align}
Recall \eqref{space} the weighted function space $L^2_w(\mathbb{R}_+)$. 
The following is the main result in this section.

\begin{proposition}[Solvability of resolvent problem]\label{prop3.0}
	Let the Mach number $\cm\in (0,\f1{\sqrt{3}}).$ There exists $\vep_{3}\in (0,1)$, such that 
	$\forall \vep\in (0,\vep_3)$, $\a=K\vep^{\f18}$ and $c\in D_0$, then the following two statements hold.
	
	(1) If $(f_u,f_v)\in L^2(\mathbb{R}_+)^2$, then there exists a solution $\vec{\Xi}=(\rho,u,v)\in H^{1}(\mathbb{R}_+)\times H^2(\mathbb{R}_+)\times H^2_0(\mathbb{R}_+)$ to \eqref{3.0.1} which satisfies the following estimates:
	\begin{align}
	\|(\cm^{-1}\rho,u,v)\|_{L^2}&\lesssim \frac{1}{\a(\text{Im}c)^2}\|(f_u,f_v)\|_{L^2},	\label{3.0.2}\\
	\|\cm^{-2}\pa_Y\rho\|_{L^2}+\|\dv(u,v)\|_{H^1}&\lesssim \frac{1}{\a(\text{Im}c)^2}\|(f_u,f_v)\|_{L^2},\label{3.0.2-1}\\
	\|(\pa_Yu,\pa_Yv)\|_{L^2}&\lesssim \frac{n^{\f12}}{\a(\text{Im}c)^{\f12}}\|(f_u,f_v)\|_{L^2},\label{3.0.3}\\
	\|(\pa_Y^2u,\pa_Y^2v)\|_{L^2}&\lesssim \frac{n}{\a\text{Im}c}\|(f_u,f_v)\|_{L^2}.\label{3.0.4}
	\end{align}
	
	(2) If in addition we have $(f_u,f_v)\in H^1(\mathbb{R}_+)^2$ with $\|\Omega(f_u,f_v)\|_{L^2_w}<\infty,$ then there exists a solution $\vec{\Xi}=(\rho,u,v) \in H^{1}(\mathbb{R}_+)\times H^2(\mathbb{R}_+)\times H^2_0(\mathbb{R}_+)$ to \eqref{3.0.1} which satisfies the following improved estimates
	\begin{align}
	\|(\cm^{-1}\rho, u,v)\|_{H^1}&\lesssim \frac{1}{\text{Im}c}\|\Omega(f_u,f_v)\|_{L^2_w}+\f1{\a}\|f_u\|+\|f_v\|_{L^2}+\|\dv(f_u,f_v)\|_{L^2},\label{3.0.5}\\
	\|(\pa_Y^2u,\pa_Y^2v)\|_{L^2}
	&\lesssim\f{n^{\f12}}{(\text{Im}c)^{\f12}} \|\Omega(f_u,f_v)\|_{L^2_w}+\frac{1}{\text{Im}c}\left(\f1{\a}\|f_u\|+\|f_v\|_{L^2}+\|\dv(f_u,f_v)\|_{L^2} \right).\label{3.0.6}
	\end{align} 
	
	Moreover, if $(f_u,f_v)(\cdot~;c)$ is analytic in $c$ with values in $L^2$, then the solution $\vec{\Xi}(\cdot~;c)$ is analytic with values in $H^1(\mathbb{R}_+)\times H^2(\mathbb{R}_+)\times H^2_0(\mathbb{R}_+)$.
	
\end{proposition}
\begin{remark}\label{rmk3.1}\quad
	
\begin{itemize}
	\item[(a)] 	By Sobolev embedding $H^1(\mathbb{R}_+)\hookrightarrow L^\infty(\mathbb{R}_+)$, the mapping $u(0;c): D_0\mapsto \mathbb{C}$ is analytic.
	\item[(b)]	The solutions to \eqref{3.0.1} are in general not unique because we do not prescribe the boundary data at $Y=0$ for $u$.
\item[(c)] The constants in estimates \eqref{3.0.2}-\eqref{3.0.6} are uniform for $\cm\in (0,\cm_0]$ with any $\cm_0\in (0,\f1{\sqrt{3}})$.
\item[(d)] As one can see from the proof, the argument also works for a wider regime of parameters: 
\begin{align}\label{re}
|\a|\lesssim 1, ~ |c|\ll 1,~\frac{|c|^2}{\text{Im}c}\ll 1,~ \frac{1}{n(\text{Im}c)^2}\ll 1.
\end{align}
In fact, the boundedness of wave number $\a$ is essentially used in the proof. In addition,
we require $c$ to satisfy \eqref{bc} so that $c\in \Sigma_{Q}\cap\Sigma_{S}$ where $\Sigma_{Q}$ and $\Sigma_{S}$ are resolvent sets of $L_Q$ and $L_S$ respectively. Moreover, in view of \eqref{3.3N.10-1}, we require smallness of $\f{1}{n(\text{Imc})^2}$ in order to establish the convergence of iteration. These requirements can be fulfilled by
the smallness in \eqref{re}.
\end{itemize}
\end{remark}

As mentioned in the Introduction, the proof of Proposition \ref{prop3.0} is based on the following two  newly introduced decompositions, i.e., quasi-compressibile approximation and the Stokes approximation.

\subsection{Quasi-compressible approximation} 
Following the strategy described in the Introduction, we first consider the following approximate problem:
\begin{equation}\label{3.1.1}
\left\{
\begin{aligned}
&i\a(U_s-c)\varrho+\dv(\cu,\cv)=0,\\
&\sqrt{\vep}\Delta_\a\left(\cu+(U_s-c)\varrho\right)-i\a (U_s-c)\cu-\cv\pa_YU_s-i\a\cm^{-2}\varrho=s_1,\\
&\sqrt{\vep}\Delta_\a \cv-i\a(U_s-c)\cv-\cm^{-2}\pa_Y\varrho=s_2,\\
&\cv|_{Y=0}=0,	
\end{aligned}
\right.
\end{equation}
with a given inhomogeneous source term $(s_1,s_2)$. 

By the continuity equation $\eqref{3.1.1}_1$, we can define a stream function $\Psi$ such that
\begin{align}\label{3.1.2}
\pa_Y\Psi=\cu+(U_s-c)\varrho,~-i\a\Psi=\cv,~\Psi|_{Y=0}=0.
\end{align}
Then by  $\eqref{3.1.1}_2$ we can express the density $\rho$ in terms of $\Psi$ as
\begin{align}\label{3.1.3}
\cm^{-2}\varrho(Y)=-A^{-1}(Y)\left[\f{i}{n}\Delta_\a\pa_Y\Psi+(U_s-c)\pa_Y\Psi-\Psi\pa_YU_s+(i\a)^{-1}s_1\right].
\end{align}
Substituting \eqref{3.1.3} into  $\eqref{3.1.1}_3$, we derive the following  equation for $\Psi$
which can be viewed as the Orr-Sommerfeld equation in the compressible setting:
\begin{align}
\label{3.1.4}
\text{OS}_{\text{CNS}}(\Psi)\eqdef\f{i}{n}\Lambda(\Delta_\a \Psi)+(U_s-c)\Lambda(\Psi)-\pa_Y(A^{-1}\pa_YU_s)\Psi=\Omega(s_1,s_2),~Y>0,
\end{align}
where the definition of operator $\Lambda$  is given by
\begin{equation}
\begin{aligned}\label{3.1.5}
&\Lambda: H^2(\mathbb{R}_+)\cap H^1_0(\mathbb{R}_+)\rightarrow L^2(\mathbb{R}_+),\\
&\Lambda(\Psi)\eqdef \pa_Y(A^{-1}\pa_Y\Psi)-\a^2\Psi,
\end{aligned}
\end{equation}
and $\Omega$ is given in \eqref{add3.1.5}.  

 In order to solve \eqref{3.1.4}, we consider the following boundary condition
\begin{align}\label{3.1.6}
\Psi|_{Y=0}=\Lambda (\Psi)|_{Y=0}=0.
\end{align} 
If $\Psi$ solves problem \eqref{3.1.4}  with boundary conditions \eqref{3.1.6}, it is straightforward to check that $(\rho,u,v)$ defined by \eqref{3.1.2} and \eqref{3.1.3} is a solution to  \eqref{3.1.1}. 

Thus, in the following, we consider
 the boundary value problem \begin{align}\label{3.1.4-1}
\text{OS}_{\text{CNS}}(\Psi)=h,~Y>0,~~\Psi|_{Y=0}=\Lambda(\Psi)|_{Y=0}=0,
\end{align}
with a given inhomogeneous source $h\in L^2_w(\mathbb{R}_+)$.
Let us first introduce the multiplier \begin{align}\label{w}
\mathcal{w}(Y)\eqdef -\left(\pa_Y\left( A^{-1}\pa_YU_s\right)\right)^{-1}.
\end{align}
A straightforward computation yields the  properties of $w$ stated in the following lemma.
\begin{lemma}\label{lm3.1}
	Let $\cm\in (0,1)$ and $U_s$ satisfy \eqref{A0}-\eqref{A3}. There exists $\gamma_3>0$, such that if $|c|< \gamma_3$,  $\cw(Y)$ has the expansion
	\begin{align}\label{ew}
	\cw=\cw_0+c\cw_1+O(1)|c|^2|\pa_Y^2U_s|^{-1}.
	\end{align} 
	Here $\cw_0$ and $\cw_1$ are given by
	\begin{align}\label{3.1.7}
	\cw_0=\f{(1-m^2U_s^2)^2}{H\left(|\pa_YU_s|^2+|\pa_Y^2U_s|\right)},~\cw_1=\f{4\cm^2U_s(1-m^2U_s^2)}{H\left(|\pa_YU_s|^2+|\pa_Y^2U_s|\right)}
	-\f{2\cm^2(1-\cm^2U_s^2)^2\left(|\pa_YU_s|^2-U_s\pa_Y^2U_s \right)}{H^2\left(|\pa_YU_s|^2+|\pa_Y^2U_s|\right)^2},
	\end{align}
	where the function $H(Y)$ is defined in \eqref{A1}.
	Moreover, it holds that
	\begin{align}\label{3.1.8}
	\cw_0(Y)\approx |\cw(Y)|\approx |\pa_Y^2U_s(Y)|^{-1}.
	\end{align}
\end{lemma}
Set the function space
\begin{align}
\mathbb{X}\eqdef \left\{\Psi\in H^3(\mathbb{R}_+)\cap H^1_0(\mathbb{R}_+),~\Lambda(\Psi)|_{Y=0}=0~\bigg|~\|\pa_Y\Psi,\a\Psi\|_{L^2}+\|\Lambda(\Psi)\|_{L^2_w}+\|\pa_Y\Lambda(\Psi)\|_{L^2_w}<\infty\right\}.\nonumber
\end{align}

For the problem \eqref{3.1.4-1}, we have the following lemma.

\begin{lemma}[A priori estimates] \label{prop3.1}
	Let $\cm\in (0,\f{1}{\sqrt{3}})$ and $\Psi\in \mathbb{X}$ be a solution to \eqref{3.1.4-1}. There exists $\gamma_4\in (0,\gamma_3)$, such that for $\a\in (0,1)$ and $c$ lies in \begin{align}\label{Q}
	 \Sigma_{Q}\eqdef\{c\in \mathbb{C}\mid Imc>\min\{\gamma^{-1}_4|c|^2,~\gamma_4^{-1}n^{-1}\},~|c|<\gamma_4\},
	\end{align}
	$\Psi$ has the following estimates
	\begin{align}\label{3.1.9}
	\|(\pa_Y\Psi,\a\Psi)\|_{L^2}+\|\Lambda(\Psi)\|_{L^2_w}&\leq\f{C}{\text{Im} c}\|h\|_{L^2_w},\\
	\|\left(\pa_Y\Lambda(\Psi),\a\Lambda(\Psi)\right)\|_{L^2_w}&\leq\f{Cn^{\f12}}{(\text{Im} c)^{\f12}}\|h\|_{L^2_w}.\label{3.1.10}
	\end{align}
\end{lemma}
\begin{proof} 
	Taking inner product of \eqref{3.1.4-1} with the multiplier $-\cw\overline{\Lambda(\Psi)}$ leads to
	\begin{align}\label{3.1.11}
	\underbrace{-\frac{i}{n}\int_0^{\infty}\cw\Lambda(\Delta_\a\Psi)\overline{\Lambda(\Psi)}\dd Y}_{J_1}+\underbrace{\int_0^{\infty}-(U_s-c)\cw|\Lambda(\Psi)|^2\dd Y}_{J_2}+\underbrace{\int_0^{\infty}-\Psi\overline{\Lambda(\Psi)}\dd Y}_{J_3}+\underbrace{\int_0^\infty h\cw\overline{\Lambda(\Psi)}\dd Y}_{J_4}=0.
	\end{align}
	Now we estimate $J_1-J_4$ separately. Let us consider $J_3$ first. By integrating by parts and using the boundary condition $\Psi|_{Y=0}=0,$ we  obtain
	\begin{align}\label{3.1.12}
	J_3=\int_0^\infty\bar{A}^{-1}|\pa_Y\Psi|^2+\a^2|\Psi|^2\dd Y.
	\end{align}
	Recalling \eqref{2.1.4} about the definition of $A(Y)$, we have 
	\begin{align}
	\b{A}^{-1}=(1-\cm^2U_s^2)^{-1}-2\cm^2U_s(1-\cm^2U_s^2)^{-2}\b{c}+O(1)|c|^2.\nonumber
	\end{align}
	With this identity, the assumption $\cm\in (0,1)$ and \eqref{A0} for positivity of $U_s$, we can 
	deduce from \eqref{3.1.12} that
	\begin{align}\label{3.1.14-1}
	\text{Re}J_3\gtrsim \a^2\|\Psi\|_{L^2}^2+(1-O(1)|c|)\|\pa_Y\Psi\|_{L^2}^2,
\end{align}
and
\begin{align}\label{3.1.14}
\text{Im}J_3\gtrsim \text{Im}c\|\cm|U_s|^{\f12}\pa_Y\Psi\|_{L^2}^2-O(1)|c|^2\|\pa_Y\Psi\|_{L^2}^2,
\end{align}	where the constants may depend on $\cm$ but not on either $\vep$ or $c$.
	
	For $J_2$, we obtain from the expansion \eqref{ew} and bound \eqref{3.1.8} that
	\begin{align}\label{3.1.15}
	-(U_s-c)\cw&=-U_s\cw_0+\left(\cw_0-U_s\cw_1\right)c+O(1)|c|^2|\pa_Y^2U_s|^{-1}.	\end{align}
	Using the explicit formula \eqref{3.1.7} of $\cw_0$ and $\cw_1$ gives 
	\begin{equation}\label{3.1.15-2}
	\begin{aligned}
	\cw_0-U_s\cw_1&=\frac{(1-\cm^2U_s^2)}{H^{2}\left(|\pa_YU_s|^2+|\pa_Y^2U_s| \right)}\left\{(1-5\cm^2U_s^2)H+\frac{2\cm^2(1-\cm^2U_s^2)(U_s|\pa_YU_s|^2-U_s^2\pa_Y^2U_s)}{|\pa_YU_s|^2+|\pa_Y^2U_s|}\right\}\\
	&=\frac{(1-\cm^2U_s^2)}{H^{2}\left(|\pa_YU_s|^2+|\pa_Y^2U_s| \right)^2}\bigg\{(1-5\cm^2U_s^2)\left[(1-\cm^2U_s^2)|\pa_Y^2U_s|-2\cm^2U_s|\pa_YU_s|^2\right]\\
	&\qquad\qquad\qquad+2\cm^2(1-\cm^2U_s^2)(U_s|\pa_YU_s|^2+U_s^2|\pa_Y^2U_s|)\bigg\}\\
	&=\frac{(1-\cm^2U_s^2)}{H^{2}\left(|\pa_YU_s|^2+|\pa_Y^2U_s| \right)^2}\bigg\{(1-3\cm^2U_s^2)(1-\cm^2U_s^2)|\pa_Y^2U_s|+8\cm^4U_s^3|\pa_YU_s|^2\bigg\}.
	\end{aligned}
	\end{equation}
	Since $\cm\in (0,\f1{\sqrt{3}})$, we have $(1-3\cm^2U_s^2)\geq (1-3\cm^2)>0.$ Thus by  \eqref{A4} and \eqref{3.1.15-2}, it holds that 
	\begin{align}\label{3.1.15-1}
	C_1|\pa_Y^2U_s|^{-1}\leq \cw_0-U_s\cw_1\leq C_2 |\pa_Y^2U_s|^{-1},
	\end{align}
	where the positive constants $C_1$ and $C_2$ are uniform in $\vep$ and $c$. Therefore, taking real and imaginary part of $J_2$ respectively and using the bounds \eqref{3.1.8}, \eqref{3.1.15} and \eqref{3.1.15-1} yield
	\begin{align}\label{3.1.16-1}
	\left|\text{Re}J_2\right|\lesssim C\|\Lambda(\Psi)\|_{L^2_w}^2,
	\end{align}
and
	\begin{align}\label{3.1.16}
	\text{Im}J_2\gtrsim \left(\text{Im}c-O(1)|c|^2\right)\|\Lambda(\Psi)\|_{L^2_w}^2.
	\end{align}
	
	For $J_1$, we rewrite
	\begin{align}\label{3.1.17}
	J_1=\f{-i}{n}\int_0^\infty \cw\Delta_\a\Lambda(\Psi)\overline{\Lambda(\Psi)}\dd Y+\f{i}{n}\int_0^\infty \cw[\Delta_\a,\Lambda](\Psi) \overline{\Lambda(\Psi)}\dd Y:=J_{11}+J_{12},
	\end{align}
	where  $[\Delta_\a,\Lambda](\Psi)$ is  the commutator $\Delta_\a\left[\Lambda(\Psi)\right]-\Lambda\left[\Delta_\a(\Psi)\right]$. By integrating by parts and using the boundary condition $\Lambda(\Psi)|_{Y=0}=0,$ we obtain
	\begin{align}\label{3.1.18}
	J_{11}=\frac{i}{n}\int_0^\infty \cw\left(|\pa_Y\Lambda(\Psi)|^2+\a^2|\Lambda(\Psi)|^2\right)\dd Y+\frac{i}{n}\int_0^\infty\pa_Y\cw\pa_Y\Lambda(\Psi)\overline{\Lambda(\Psi)}\dd Y.
	\end{align}
	Then by  \eqref{A3} and \eqref{3.1.8}, we have 
	$$
	\begin{aligned}
	\left\|\pa_Y\cw\pa_Y^2U_s\right\|_{L^\infty}&=\|\cw^{2}\pa_Y^2(A^{-1}\pa_YU_s)\pa_Y^2U_s\|_{L^\infty}\\
	&\leq \|\cw\pa_Y^2U_s\|^2_{L^\infty}\left( \|\pa_YU_s\|_{L^\infty}+\||\pa_Y^2U_s|^{-1}|\pa_YU_s|^2\|_{L^\infty}+\||\pa_Y^2U_s|^{-1}|\pa_Y^3U_s\|_{L^\infty}\right)\leq C.
	\end{aligned}
	$$
	Thus the last integral on the right hand side of \eqref{3.1.18} is bounded by
	\begin{align}\label{3.1.19}
	\left|\frac{i}{n}\int_0^\infty\pa_Y\cw\pa_Y\Lambda(\Psi)\overline{\Lambda(\Psi)}\dd Y\right|&\lesssim \f1n\|\pa_Y\cw\pa_Y^2U_s\|_{L^\infty}\|\pa_Y\Lambda(\Psi)\|_{L^2_w}\|\Lambda(\Psi)\|_{L^2_w}\nonumber\\
	&\lesssim \frac{1}{n}\|\pa_Y\Lambda(\Psi)\|_{L^2_w}\|\Lambda(\Psi)\|_{L^2_w}.
	\end{align}
	By taking real and imaginary parts of $J_{11}$ respectively and using \eqref{3.1.19}, we deduce that
	\begin{align}
	\left|\text{Re}J_{11}\right|&\lesssim \frac{1}{n}\int_0^\infty |\text{Im}\cw|\left(|\pa_Y\Lambda(\Psi)|^2+\a^2|\Lambda(\Psi)|^2 \right)\dd Y +\f1n\|\pa_Y\Lambda(\Psi)\|_{L^2_w}\|\Lambda(\Psi)\|_{L^2_w},\nonumber\\
	&\lesssim \frac{|c|}{n}\|(\pa_Y\Lambda(\Psi),\a\Lambda(\Psi))\|_{L^2_w}^2+\f1n\|\pa_Y\Lambda(\Psi)\|_{L^2_w}\|\Lambda(\Psi)\|_{L^2_w},
	\label{3.1.20}
	\end{align}
	and
	\begin{align}
	\text{Im}J_{11}&\gtrsim \f1n\int_0^\infty \text{Re}\cw\left(|\pa_Y\Lambda(\Psi)|^2+\a^2|\Lambda(\Psi)|^2 \right)\dd Y-\f1n\|\pa_Y\Lambda(\Psi)\|_{L^2_w}\|\Lambda(\Psi)\|_{L^2_w}\nonumber\\
	&\gtrsim \frac{1}{n}\|(\pa_Y\Lambda(\Psi),\a\Lambda(\Psi))\|_{L^2_w}^2-\f1n\|\pa_Y\Lambda(\Psi)\|_{L^2_w}\|\Lambda(\Psi)\|_{L^2_w},\label{3.1.21}
	\end{align}
	where we have used the fact that
	\begin{align}|
	\text{Im}\cw|\lesssim |\text{Im}c|\cw_1+O(1)|c|^2|\pa_Y^2U_s|^{-1}\lesssim |c||\pa_Y^2U_s|^{-1} \nonumber
	\end{align}
and
\begin{align}
	\text{Re}\cw\gtrsim \cw_0-|c||\cw_1|-O(1)|c|^2||\pa_Y^2U_s|^{-1}\gtrsim |\pa_Y^2U_s|^{-1}\nonumber
	\end{align}
	by the expansion  \eqref{ew}.
	
	Next we estimate $J_{12}$. Recall \eqref{3.1.5} the definition of $\Lambda$. We have
	\begin{align}
	[\Delta_\a,\Lambda](\Psi)&=\Delta_\a[\Lambda(\Psi)]-\Lambda[\Delta_\a(\Psi)]=\pa_Y^3(A^{-1}\pa_Y\Psi)-\pa_Y(A^{-1}\pa_Y^3\Psi)\nonumber\\
	&=2\pa_Y(A^{-1})\pa_Y^3\Psi+3\pa_Y^2(A^{-1})\pa_Y^2\Psi+\pa_Y^3(A^{-1})\pa_Y\Psi.\label{3.1.22-2}
	\end{align}
	We rewrite $\pa_Y^2\Psi$ and $\pa_Y^3\Psi$ as
	\begin{equation}
	\begin{aligned}\nonumber
	\pa_Y^2\Psi&=A\Lambda(\Psi)+A^{-1}\pa_YA\pa_Y\Psi+\a^2A\Psi,\\
	\pa_Y^3\Psi&=A\pa_Y\Lambda(\Psi)+2\Lambda(\Psi)\pa_YA+\pa_Y\Psi\left( A^{-1}\pa_Y^2A+\a^2A\right)+2\a^2\Psi\pa_YA.
	\end{aligned}
	\end{equation}
By $\a\in (0,1)$, it holds that
	\begin{align}\label{3.1.22}
	|\pa_Y^2\Psi|\lesssim |\Lambda(\Psi)|+|\pa_Y\Psi|+\a|\Psi|,~|\pa_Y^3\Psi|\lesssim |\pa_Y\Lambda(\Psi)|+ |\Lambda(\Psi)|+|\pa_Y\Psi|+\a|\Psi|.
	\end{align}
By \eqref{A3}, we have
\begin{equation}
\begin{aligned}
&\left|\pa_Y(A^{-1})\right|\leq C|\pa_YU_s|,~\left|\pa_Y^2(A^{-1})\right|\leq C|\pa_Y^2U_s|+C|\pa_YU_s|^2\leq C|\pa_YU_s|,\\
&\left|\pa_Y^3(A^{-1})\right|\leq C|\pa_Y^3U_s|+C|\pa_Y^2U_s\pa_YU_s|+C|\pa_YU_s|^3\leq C|\pa_YU_s|.\label{3.1.22-3}
\end{aligned}
\end{equation}
	Then applying the bounds \eqref{3.1.22} and \eqref{3.1.22-3} to \eqref{3.1.22-2} gives
	\begin{align}
	\left|[\Delta_\a,\Lambda](\Psi)\right|\lesssim |\pa_YU_s|\left( |\pa_Y\Lambda(\Psi)|+ |\Lambda(\Psi)|+|\pa_Y\Psi|+\a|\Psi| \right),\nonumber
	\end{align}
which by \eqref{A4} implies
\begin{align}
 	\left\|[\Delta_\a,\Lambda](\Psi)\right\|_{L^2_w}&\lesssim \||\pa_Y^2U_s|^{-\f12}\pa_YU_s\|\bigg( \|\pa_Y\Lambda(\Psi)\|_{L^2}+ \|\Lambda(\Psi)\|_{L^2}+\|\pa_Y\Psi\|_{L^2}+\a\|\Psi\|_{L^2} \bigg)\nonumber\\
 	& \lesssim \|\pa_Y\Lambda(\Psi)\|_{L^2}+ \|\Lambda(\Psi)\|_{L^2}+\|\pa_Y\Psi\|_{L^2}+\a\|\Psi\|_{L^2}.     \label{3.1.22-1}
\end{align}
	Substituting \eqref{3.1.22-1} into $J_{12}$ and using Cauchy-Schwarz inequality yield
	\begin{align}\label{3.1.23}
	|J_{12}|&\lesssim \f{1}{n}\|\cw\pa_Y^2U_s\|_{L^\infty}\|\Lambda(\Psi)\|_{L^2_w}\|[\Delta_\a,\Lambda](\Psi)\|_{L^2_w}\nonumber\\
	&\lesssim \f{1}{n}\|\Lambda(\Psi)\|_{L^2_w}\bigg(\|\pa_Y\Lambda(\Psi)\|_{L^2_w}+\|\Lambda(\Psi)\|_{L^2_w}+\|(\pa_Y\Psi,\a\Psi)\|_{L^2}\bigg),
	\end{align}
	where  we have used  \eqref{3.1.8}. By \eqref{3.1.20}, \eqref{3.1.21} and \eqref{3.1.23},
	we can deduce from real and imaginary parts of \eqref{3.1.17}:
	\begin{align}
	\left|\text{Re}J_1\right|\lesssim& \frac{|c|}{n}\|(\pa_Y\Lambda(\Psi),\a\Lambda(\Psi))\|_{L^2_w}^2\nonumber\\
	&+\f1n\|\Lambda(\Psi)\|_{L^2_w}\bigg(\|\pa_Y\Lambda(\Psi)\|_{L^2_w}+\|\Lambda(\Psi)\|_{L^2_w}+\|(\pa_Y\Psi,\a\Psi)\|_{L^2}\bigg),\label{3.1.24}
	\end{align}
and
	\begin{align}
	\text{Im}J_1\gtrsim& \frac{1}{n}\|(\pa_Y\Lambda(\Psi),\a\Lambda(\Psi))\|_{L^2_w}^2\nonumber\\
	&-\f1n\|\Lambda(\Psi)\|_{L^2_w}\bigg(\|\pa_Y\Lambda(\Psi)\|_{L^2_w}+\|\Lambda(\Psi)\|_{L^2_w}+\|(\pa_Y\Psi,\a\Psi)\|_{L^2}\bigg).\label{3.1.25}
	\end{align}
	
	Finally, for $J_4$, we have by Cauchy-Schwarz inequality that
	\begin{align}\label{3.1.26}
	J_4\lesssim \|\cw \pa_Y^2U_s\|_{L^\infty} \|h\|_{L^2_w}\|\Lambda(\Psi)\|_{L^2_w}\lesssim \|h\|_{L^2_w}\|\Lambda(\Psi)\|_{L^2_w}.
	\end{align}
	Thus, we have completed the estimation on $J_1-J_4$.
	
	By taking imaginary part of \eqref{3.1.11} and using previous bounds \eqref{3.1.14}, \eqref{3.1.16},  \eqref{3.1.25} and \eqref{3.1.26} for $J_1-J_4$, we have
	\begin{align}\label{3.1.27}
	&\frac{1}{n}\|(\pa_Y\Lambda(\Psi),\a\Lambda(\Psi))\|_{L^2_w}^2+\text{Im}c\left(\|\Lambda(\Psi)\|_{L^2_w}^2+\|\cm|U_s|^{\f12}\pa_Y\Psi\|_{L^2}^2\right)\nonumber\\
	&\qquad\lesssim \f1n\|\Lambda(\Psi)\|_{L^2_w}\left(\|\pa_Y\Lambda(\Psi)\|_{L^2_w}+\|\Lambda(\Psi)\|_{L^2_w}+\|(\pa_Y\Psi,\a\Psi)\|_{L^2}\right)\nonumber\\
	&\qquad\qquad+|c|^2\left( \|\Lambda(\Psi)\|_{L^2_w}^2+\|(\pa_Y\Psi,\a\Psi)\|_{L^2}^2\right)+\|h\|_{L^2_w}\|\Lambda(\Psi)\|_{L^2_w}.
	\end{align}
	Similarly, taking real part of \eqref{3.1.11} and using \eqref{3.1.14-1}, \eqref{3.1.16-1}, \eqref{3.1.24} and \eqref{3.1.26} give
	\begin{align}\label{3.1.28}
	\|(\pa_Y\Psi,\a\Psi)\|_{L^2}^2\lesssim& \f1n\|\Lambda(\Psi)\|_{L^2_w}\left(\|\pa_Y\Lambda(\Psi)\|_{L^2_w}+\|\Lambda(\Psi)\|_{L^2_w}+\|(\pa_Y\Psi,\a\Psi)\|_{L^2}\right)\nonumber
	\\
	&+\|\Lambda(\Psi)\|_{L^2_w}^2+\frac{|c|}{n}\|(\pa_Y\Lambda(\Psi),\a\Lambda(\Psi))\|_{L^2_w}^2 +\|h\|_{L^2_w}\|\Lambda(\Psi)\|_{L^2_w}.
	\end{align}
	Multiplying estimate \eqref{3.1.28} by $\text{Im}c$ and suitably combining it with \eqref{3.1.27},  we can obtain by Young's inequality that
	\begin{align}
	&\frac{1}{n}\|(\pa_Y\Lambda(\Psi),\a\Lambda(\Psi))\|_{L^2_w}^2+\text{Im}c\left(\|\Lambda(\Psi)\|_{L^2_w}^2+\|(\pa_Y\Psi,\a\Psi)\|_{L^2}^2\right)\nonumber\\
	&\qquad\lesssim \f1n\|\Lambda(\Psi)\|_{L^2_w}\left(\|\pa_Y\Lambda(\Psi)\|_{L^2_w}+\|\Lambda(\Psi)\|_{L^2_w}+\|(\pa_Y\Psi,\a\Psi)\|_{L^2}\right)\nonumber\\
	&\qquad\qquad+|c|^2\left( \|\Lambda(\Psi)\|_{L^2_w}^2+\|(\pa_Y\Psi,\a\Psi)\|_{L^2}^2\right)+\frac{|c|}{n}\|(\pa_Y\Lambda(\Psi),\a\Lambda(\Psi))\|_{L^2_w}^2+\|h\|_{L^2_w}\|\Lambda(\Psi)\|_{L^2_w}\nonumber\\
	&\qquad\leq \left(\frac{1}{2n}+\f{C|c|}{n}\right)\|(\pa_Y\Lambda(\Psi),\a\Lambda(\Psi)\|_{L^2_w}^2+\text{Im}c\left(\f12+\f{C}{n\text{Im}c}+\frac{C|c|^2}{\text{Im}c}\right)\nonumber\\
	&\qquad\qquad\times\left(\|\Lambda(\Psi)\|_{L^2_w}^2+\|(\pa_Y\Psi,\a\Psi)\|_{L^2}^2\right)+\frac{C}{\text{Im}c}\|h\|_{L^2_w}^2.\label{3.1.28-1}
	\end{align}
	By taking $\gamma_4\in (0,\gamma_3)$ suitably small such that $$C|c|\leq C\gamma_4\leq \f14,~\text{and }
	\frac{C}{n\text{Im}c}+\frac{C|c|^2}{\text{Im}c}\leq2C\gamma_4\leq \f14,~ \forall c\in \Sigma_Q,$$ we can absorb the first and second terms on the right hand side of \eqref{3.1.28-1} 
	by 
	 the left hand side. Thus, 
	$$\|\Lambda(\Psi)\|_{L^2_w}+\|(\pa_Y\Psi,\a\Psi)\|_{L^2}\leq \frac{C}{\text{Im}c}\|h\|_{L^2_w},~\text{and }\|(\pa_Y\Lambda(\Psi),\a\Lambda(\Psi))\|_{L^2_w}\leq\frac{Cn^{\f12}}{(\text{Im}c)^{\f12}}\|h\|_{L^2_w}.
	$$
	The above two inequalities immediately imply the estimates \eqref{3.1.9} and \eqref{3.1.10}. The proof of the lemma is completed.
\end{proof}
With the a priori estimates in Lemma \ref{prop3.1}, we can prove the existence, uniqueness and analytic dependence on $c$ of the solution to  the compressible
Orr-Sommerfeld equation \eqref{3.1.4-1} in the following lemma. 

\begin{lemma}[Construction of the solution]\label{lm2.5}
	Let $\cm\in (0,\frac{1}{\sqrt{3}})$, $\a\in (0,1)$ and $c\in \Sigma_Q$. If $\|h\|_{L^2_w}<\infty$, there exists a unique solution $\Psi\in \mathbb{X}$ to \eqref{3.1.4-1} which satisfies estimates  \eqref{3.1.9} and \eqref{3.1.10}. Moreover, if $h(\cdot~;c)$ is analytic in $c$  in $L^2_w(\mathbb{R}_+)$, then $\Psi(\cdot~;c)$ is analytic  in $\mathbb{X}$.
\end{lemma}
\begin{remark}
	By elliptic regularity, the solution $\Psi$ is in $H^4(\mathbb{R}_+)\cap H^1_0(\mathbb{R}_+)$.
\end{remark}
\begin{proof}
	The proof is based on a cascade of approximate process and a continuity argument. First of all, we set $W\eqdef \Lambda(\Psi)$ and reformulate \eqref{3.1.4-1} as
	\begin{equation}
	\left\{
	\begin{aligned}\label{3.1.40}
	&\f{i}{n}\Lambda(\Delta_\a \Lambda^{-1}W)+(U_s-c)W+\cw^{-1}\Lambda^{-1}W=h,~Y>0,\\
	&W|_{Y=0}=0.
	\end{aligned}\right.
	\end{equation}	
	Here the inverse operator $\Lambda^{-1}: L^2(\mathbb{R}_+)\rightarrow H^2(\mathbb{R}_+)\cap H^1_0(\mathbb{R}_+)$ is constructed in Lemma \ref{lm4.1}.  If one can show the solvability of \eqref{3.1.40} in $H^1_w(\mathbb{R}_+)$, then by  Lemma \ref{lm4.1}, $\Psi\eqdef\Lambda^{-1}(W)\in \mathbb{X}$ and it solves the equation \eqref{3.1.4-1}. Now we elaborate the construction of solution to \eqref{3.1.40} in the following three steps.
	
	{\it Step 1.} Fix any parameter $l>0$. We start from an auxiliary problem
	\begin{align}\label{3.1.41}
	T_l(W)\eqdef \frac{i}{n}\Delta_\a W+(U_s-c-il)W+\cw^{-1}\Lambda^{-1}W=h,~W|_{Y=0}=0.
	\end{align}
	We claim that there exists $l_0>0$, such that if $c\in \Sigma_{Q}$ and $\|h\|_{L^2_w}<\infty$, then \eqref{3.1.41} admits a unique solution $W\in H^2_w(\mathbb{R}_+)\cap H^1_0(\mathbb{R}_+)$ and the solution operator $T_{l_0}^{-1}: L^2_w(\mathbb{R}_+)\rightarrow H^2_w(\mathbb{R}_+)$ is analytic in $c$. To prove this claim, we define a sequence of approximate solutions $\{W_k\}_{k=0}^\infty$  by the hierarchy of equations
	\begin{align}
	\left[\text{Airy}-(c+il)\right](W_{k+1})=h-\cw^{-1}\Lambda^{-1}W_k,~ W_{k+1}\big|_{Y=0}=0,~W_0\equiv 0,\label{3.1.42}
	\end{align}
	where  $\text{Airy}\eqdef\frac{i}{n}\Delta_{\a}+U_s: H^2_w(\mathbb{R}_+)\cap H^1_0(\mathbb{R}_+)\rightarrow L^2_w(\mathbb{R}_+)$ is the Airy operator. For any $c\in \Sigma_Q$ and $l>0$, by direct energy method, it is straightforward to check that
	  $c+il$ lies in the resolvent set of Airy operator. Thus by an inductive argument, we can  solve $W_k$ and establish its analytic dependence on $c$ from \eqref{3.1.42}. In order to take the limit $k\rightarrow \infty,$ we need  some uniform estimates. Applying the multiplier $(\pa_Y^2U_s)^{-1}\b{W}_{k+1}$ to \eqref{3.1.42}, we have
	\begin{align}
	&\frac{i}{n}\int_0^\infty (\pa_Y^2U_s)^{-1}\b{W}_{k+1}\Delta_\a W_{k+1} \dd Y+\int_0^\infty (\pa_Y^2U_s)^{-1}(U_s-c-il)|W_{k+1}|^2\dd Y\nonumber\\
	&\qquad=-\int_0^\infty (\pa_Y^2U_s)^{-1}\cw^{-1}\Lambda^{-1}(W_k)\b{W}_{k+1}\dd Y+\int_0^\infty (\pa_Y^2U_s)^{-1}h\b{W}_{k+1}\dd Y.
	\label{3.1.43}
	\end{align}
	By Cauchy-Schwarz inequality, we deduce that
	\begin{align}\label{3.1.43-4}
	\left|\int_0^\infty (\pa_Y^2U_s)^{-1}h\b{W}_{k+1}\dd Y \right|\leq C\|h\|_{L^2_w}\|W_{k+1}\|_{L^2_w}.
	\end{align}
	By using \eqref{3.1.8} and the bound in \eqref{ap1} for $\Lambda^{-1}$, we have
	\begin{align}
	\left|\int_0^\infty (\pa_Y^2U_s)^{-1}\cw^{-1}\Lambda^{-1}(W_k)\b{W}_{k+1}\dd Y\right|&\leq C\||\cw|^{-1}|\pa_Y^2U_s|^{-\f12}\|_{L^\infty}\|W_{k+1}\|_{L^2_w}\|\Lambda^{-1}(W_k)\|_{L^2}\nonumber\\
	&\leq C\|W_{k+1}\|_{L^2_w}\|(1+Y)W_k\|_{L^2}\leq C\|W_{k+1}\|_{L^2_w}\|W_{k}\|_{L^2_w}.\label{3.1.43-1}
	\end{align}
	Integration by parts yields 
	\begin{align}
	\frac{i}{n}\int_0^\infty (\pa_Y^2U_s)^{-1}\b{W}_{k+1}\Delta_\a W_{k+1} \dd Y=&\frac{i}{n}\left(\|\pa_YW_{k+1}\|_{L^2_w}^2+\a^2\|W_{k+1}\|_{L^2_w}^2\right)\nonumber\\
	&+\f{i}{n}\int_0^\infty \f{\pa_Y^3U_s}{(\pa_Y^2U_s)^2}\pa_YW_{k+1}\b{W}_{k+1}\dd Y.\label{3.1.43-2}
	\end{align}
	By \eqref{A3}, the last integral in the above equality is bounded by
	\begin{align}\label{3.1.43-3}
	\left|\f{i}{n}\int_0^\infty \f{\pa_Y^3U_s}{(\pa_Y^2U_s)^2}\pa_YW_{k+1}\b{W}_{k+1}\dd Y\right|&\leq \f{1}{n}\left\|\f{\pa_Y^3U_s}{\pa_Y^2U_s}\right\|_{L^\infty}\|\pa_YW_{k+1}\|_{L^2_w}\|W_{k+1}\|_{L^2_w}\nonumber\\
	&\leq \f{C}{n}\|\pa_YW_{k+1}\|_{L^2_w}\|W_{k+1}\|_{L^2_w}.
	\end{align}
	By taking the  imaginary part of \eqref{3.1.43}, and  using the bounds obtained in \eqref{3.1.43-4}-\eqref{3.1.43-3} with Young's inequality, we have
	\begin{align}
	&\frac{1}{n}\left(\|\pa_YW_{k+1}\|_{L^2_w}^2+\a^2\|W_{k+1}\|_{L^2_w}^2\right)+(\text{Im}c+l)\|W_{k+1}\|_{L^2_w}^2\nonumber\\
	&\qquad\leq C\|W_{k+1}\|_{L^2_w}\left(\|h\|_{L^2_w}+\|W_{k}\|_{L^2_w}+\frac{1}{n}\|\pa_YW_{k+1}\|_{L^2_w}\right)\nonumber\\
	&\qquad\leq\f1{2n}\|\pa_YW_{k+1}\|_{L^2_w}^2+\frac{\text{Im}c+l}{2}\left(1+\f{C}{n(\text{Im}c+l)}\right)\|W_{k+1}\|_{L^2_w}^2+\frac{C}{\text{Im}c+l}\left(\|h\|_{L^2_w}^2+\|W_{k}\|_{L^2_w}^2 \right)\label{3.1.44}.
	\end{align}
	We choose $\gamma_4>0$  smaller if needed so that $\frac{C}{n\text{Im}c}<\frac{1}{2}$, for any $c\in \Sigma_{Q}$. Then  \eqref{3.1.44} gives
	\begin{align}
	\|W_{k+1}\|_{L^2_w}\leq \frac{C}{\text{Im}c+l}\left(\|W_{k}\|_{L^2_w}+\|h\|_{L^2_w}\right),~\|(\pa_YW_{k+1},\a W_{k+1})\|_{L^2_w}\leq \frac{Cn^{\f12}}{(\text{Im}c+l)^{\f12}}\left(\|W_{k}\|_{L^2_w}+\|h\|_{L^2_w}\right).\nonumber
	\end{align}
	
	Now we take the difference $W_{k+1}-W_k$. Similar argument gives
	$$\|W_{k+1}-W_{k}\|_{L^2_w}\leq \frac{C}{\text{Im}c+l}\|W_{k}-W_{k-1}\|_{L^2_w},~\|\pa_YW_{k+1}-\pa_YW_{k}\|_{L^2_w}\leq \frac{Cn^{\f12}}{(\text{Im}c+l)^{\f12}}\|W_{k}-W_{k-1}\|_{L^2_w}.
	$$
	By taking $l$ suitably large, such that $\frac{C}{\text{Im}c+l}\leq \frac{C}{l}\leq \f12$,  $\{W_{k}\}_{k=1}^\infty$ is a Cauchy sequence in $H^1_w(\mathbb{R}_+)$. This implies the existence of a limit function $W=\lim_{k\rightarrow \infty} W_k$ in $ H^1_w(\mathbb{R}_+)$
	that is the solution to \eqref{3.1.41}. By the elliptic regularity,  $W_k$ converges to $W$ in $H^2_w(\mathbb{R}_+).$ Moreover, by induction, each $W_k$ is analytic in $c$, so is $W$ by uniform convergence. This justifies the claim and step 1 is completed.
	
	{\it Step 2.} (Bootstrap from $T_{l_0}^{-1}$ to $T_{0}^{-1}$). Consider the equation \eqref{3.1.41} for any fix $l\in [0,l_0]$ with $W^l$ as its solution. Applying the multiplier $-\cw\b{W}^l$  and using the same argument as in Lemma \ref{prop3.1}, we can show that $W^l$ satisfies
	\begin{align}\label{3.1.45}
	\|W^l\|_{L^2_w}\leq \frac{C}{\text{Im}c}\|h\|_{L^2_w},~ \|(\pa_YW^l,\a W^l)\|_{L^2_w}\leq \frac{Cn^{\f12}}{(\text{Im}c)^{\f12}}\|h\|_{L^2_w},~\forall c\in \Sigma_Q,
	\end{align}
	where the constant $C$ is uniform in $l\in [0,l_0]$. Now we take $l_1=l_0-\lambda$ for some fixed constant $0<\lambda<2C^{-1}\gamma_4^{-1}n^{-1}$ and construct the solution $W^{l_1}=T_{l_1}^{-1}(h)$ through the following iteration
	$$W^{l_1}_{k+1}=T_{l_0}^{-1}\left(- i\lambda W_k^{l_1}+h\right),~ W_0^{l_1}\equiv0.
	$$
	Applying the a priori estimate \eqref{3.1.45} to $W_{k+1}^{l_1}-W_{k}^{l_1}$ yields that
	\begin{align}
	\|W_{k+1}^{l_1}-W_{k}^{l_1}\|_{L^2_w}&\leq \f{C\lambda}{\text{Im}c}\|W_k^{l_1}-W_{k-1}^{l_1}\|_{L^2_w}\leq \frac{1}{2\gamma_4n\text{Im}c}\|W_k^{l_1}-W_{k-1}^{l_1}\|_{L^2_w}\leq \frac{1}{2}\|W_k^{l_1}-W_{k-1}^{l_1}\|_{L^2_w},\nonumber\\
	\|\pa_YW_{k+1}^{l_1}-\pa_YW_k^{l_1}\|_{L^2_w}&\leq \frac{C\lambda n^{\f12}}{(\text{Im}c)^{\f12}}\|W_{k}^{l_1}-W_{k-1}^{l_1}\|_{L^2_w}\leq C\|W_{k}^{l_1}-W_{k-1}^{l_1}\|_{L^2_w},~\forall c\in \Sigma_Q.\nonumber
	\end{align}
Hence, $\{W_{k}^{l_1}\}_{k=0}^\infty$ is a Cauchy sequence in $H^1_w$ and it has a limit $W^{l_1}=\lim_{k\rightarrow\infty}W^{l_1}_k$. It is straightforward to check that $W^{l_1}$ is in $H^2_{w}(\mathbb{R}_+)\cap H^1_0(\mathbb{R}_+)$ and  satisfies \eqref{3.1.41} with $l=l_1$.  Moreover, from the previous step we have already shown that each $W_{k}^{l_1}$ is analytic in $c$. Thus analyticity of $W^{l_1}$ follows from the uniform convergence. Thus we have completed the construction solution operator $T^{-1}_{l_1}$. Noting that $W^{l_1}$ satisfies the a priori  estimate \eqref{3.1.45}, we can take $l_2=l_1-\lambda$ and construct the solution operator $T_{l_2}^{-1}$ in the same way. Repeating the same procedure, we can eventually establish the existence and analytic dependence on $c$ of the solution operator $T_0^{-1}$.
	
	{\it Step 3.} We now solve the original system
	\eqref{3.1.40} by using the following iteration
	\begin{align}\nonumber
	T_{0}(W_{k+1})=h+\frac{i}{n}\left[\Delta_\a,\Lambda \right](\Lambda^{-1}(W_{k})),~Y>0,~ W_{k+1}|_{Y=0}=0,~ W_0(Y)\equiv 0.
	\end{align}
	By using the bounds in  \eqref{3.1.22-1} and  \eqref{ap1}
	on the commutator $[\Delta_\a,\Lambda]$  and $\Lambda^{-1}$ respectively,  we have
	$$
	\begin{aligned}
	\left\|\left[\Delta_\a,\Lambda \right](\Lambda^{-1}(W_{k}))\right\|_{L^2_w}&\leq C\left(\|\pa_YW_{k}\|_{L^2}+\|W_k\|_{L^2}+\|(\pa_Y\Lambda^{-1}W_k,\a\Lambda^{-1}W_k)\|_{L^2}   \right)\\
	&\leq C\left(\|\pa_YW_k\|_{L^2}+\|W_k(1+Y)\|_{L^2}\right)\leq C\|W_k\|_{H^1_w}.
	\end{aligned}
	$$
	Then applying the a priori bound \eqref{3.1.45} to $W_{k+1}-W_{k}$ gives
	$$\|W_{k+1}-W_{k}\|_{L^2_w}\leq \frac{C}{n\text{Im}c}\|W_{k}-W_{k-1}\|_{H^1_w},~\|\pa_YW_{k+1}-\pa_YW_{k}\|_{L^2_w}\leq \frac{C}{n^{\f12}(\text{Im}c)^{\f12}}\|W_{k}-W_{k-1}\|_{H^1_w}.
	$$
	By taking $\gamma_4>0$  smaller if needed such that $\frac{C}{n\text{Im}c}+\frac{C}{n^{\f12}\text{Im}c^{\f12}}\leq C\gamma_4^{\f12}(1+\g_4^{\f12})<\f12$ for $c\in \Sigma_Q,$ we show that $\{W_k\}_{k=0}^\infty$ is a Cauchy sequence in $H^1_w(\mathbb{R}_+)$. Let $W:=\lim_{k\rightarrow\infty}W_k$. By the elliptic regularity and $H^1_w$-convergence, it is straightforward to check that $W_k$ converges to $W$ in $H^2_w(\mathbb{R}_+)$ and $W$ is a solution to \eqref{3.1.40}. Moreover, since each $W_k$ is analytic in $c$ and the convergence is uniform in $c\in \Sigma_Q$, we conclude that $W$ is analytic in $c$. The uniqueness of solution follows from the a priori estimates obtained in Lemma \ref{prop3.1}. Then the proof of the lemma  is completed.	
\end{proof}

Now let $\Psi$ be the solution to \eqref{3.1.4-1} with $h=\Omega(s_1,s_2)$ and $\Omega$ defined in \eqref{add3.1.5}. In terms of the fluid variables $(\varrho,\cu,\cv)$ given in  \eqref{3.1.2} and \eqref{3.1.3}, we have the following proposition for the solvability of the quasi-compressible approximation system \eqref{3.1.1}.

\begin{proposition}[Solvability of quasi-compressible system]\label{prop3.2}
 Under the same assumption on $\cm$, $\a$ and $c$ as in Lemma \ref{lm2.5}, if $\vec{s}=(s_1,s_2)\in H^1(\mathbb{R}_+)^2$ and $\|\Omega(s_1,s_2)\|_{L^2_w}<\infty$,  there exists a solution $(\varrho,\cu,\cv)\in H^2(\mathbb{R}_+)^3$ to the quasi-compressible approximation system \eqref{3.1.1}. Moreover, $(\varrho,\cu,\cv)$ satisfies the following estimates:
	\begin{align}
	&\|\cu\|_{H^1}+\|(\cm^{-2}\varrho,\cv)\|_{H^2}+\a^{-1}\|\dv(\cu,\cv)\|_{H^1}\nonumber\\
	&\qquad\qquad\lesssim \frac{1}{\text{Im}c}\|\Omega(s_1,s_2)\|_{L^2_w}+\f1\a\|s_1\|_{L^2}+\|s_2\|_{L^2}+\|\dv(s_1,s_2)\|_{L^2},\label{3.1.29}
	\end{align}
	and
	\begin{align}
	\|\pa_Y^2\cu\|_{L^2}\lesssim \frac{n^{\f12}}{(\text{Im}c)^{\f12}}\|\Omega(s_1,s_2)\|_{L^2_w}+\f1\a\|s_1\|_{L^2}+\|s_2\|_{L^2}+\|\dv(s_1,s_2)\|_{L^2}. \label{3.1.30}
	\end{align}
	Furthermore, if both $\vec{s}(\cdot~;c)$ and $\Omega(\cdot~;c)$ are analytic in $c$  in $H^1(\mathbb{R}_+)$ and $L^2_w(\mathbb{R}_+)$ respectively, then $(\varrho,\cu,\cv)(\cdot~;c)$ is analytic in $c$  in $H^2(\mathbb{R}_+)$. 
\end{proposition} 
\begin{remark}
	If $\dv(s_1,s_2)=0,$ then by \eqref{3.1.2}, \eqref{3.1.37-1} and regularity of $\Psi$ it is easy to deduce that $(\varrho,\cu,\cv)\in H^3(\mathbb{R}_+)^3$. This reveals the elliptic structure for linearized compressible Navier-Stokes equations around the subsonic boundary layer profile.
\end{remark}
\begin{proof}
	It is straightforward to check that $(\varrho,\cu,\cv)$ satisfies \eqref{3.1.1}. The analyticity directly follows from Lemma \ref{prop3.1}. It remains to show the estimates \eqref{3.1.29} and \eqref{3.1.30}. Firstly, by using bounds given in \eqref{3.1.9}, \eqref{3.1.10} with $h=\Omega(s_1,s_2)$ and \eqref{3.1.22}, we obtain that
	\begin{align}
	\|\pa_Y^2\Psi\|_{L^2}+\|(\pa_Y\Psi,\a\Psi)\|_{L^2}\lesssim \|\Lambda(\Psi)\|_{L^2}+\|(\pa_Y\Psi,\a\Psi)\|_{L^2}\lesssim \frac{1}{\text{Im}c}\|\Omega(s_1,s_2)\|_{L^2_w},
	\label{3.1.32}
	\end{align}
and
   \begin{align}
	\|\pa_Y^3\Psi\|_{L^2}&\lesssim \|\pa_Y\Lambda(\Psi)\|_{L^2}+ \|\Lambda(\Psi)\|_{L^2}+\|(\pa_Y\Psi,\a\Psi)\|_{L^2}\nonumber\\
	&\lesssim \frac{n^{\f12}}{(\text{Im}c)^{\f12}}\left(1+\frac{1}{n^{\f12}(\text{Im}c)^{\f12}}\right)\|\Omega(s_1,s_2)\|_{L^2}\lesssim \frac{n^{\f12}}{(\text{Im}c)^{\f12}}\|\Omega(s_1,s_2)\|_{L^2_w},\label{3.1.31}
	\end{align}
where we have used $n\text{Im}c\gtrsim 1$ for $c\in \Sigma_Q$. Then by $\cv=-i\a\Psi$, \eqref{3.1.32} and $\a\in (0,1)$, it holds that \begin{align}\label{3.1.33}
	\|\cv\|_{H^2}\lesssim \|\pa_Y^2\Psi\|_{L^2}+\|(\pa_Y\Psi,\a\Psi)\|_{L^2}\lesssim  \frac{1}{\text{Im}c}\|\Omega(s_1,s_2)\|_{L^2_w}.
	\end{align}
	
	Next we estimate $\rho$. Recall \eqref{3.1.3} for its representation. Since $\Psi|_{Y=0}=0,$ we can use Hardy inequality $\|Y^{-1}\Psi\|_{L^2}\leq 2\|\pa_Y\Psi\|_{L^2},$ and the bounds given in \eqref{3.1.32}, \eqref{3.1.31} to obtain
	\begin{align}\label{3.1.34}
	m^{-2}\|\varrho\|_{L^2}&\lesssim \f1n\|\pa_Y^3\Psi\|_{L^2}+(1+\f{\a^2}{n})\|\pa_Y\Psi\|_{L^2}+\|Y^{-1}\Psi\|_{L^2}\|Y\pa_YU_s\|_{L^\infty}+\f1\a\|s_1\|_{L^2}\nonumber\\
	&\lesssim \frac{1}{n}\|\pa_Y^3\Psi\|_{L^2}+\|\pa_Y\Psi\|_{L^2}+\f1\a\|s_1\|_{L^2}\nonumber\\
	&\lesssim\left( \f{1}{n^{\f12}(\text{Im}c)^{\f12}}+\frac{1}{\text{Im}c} \right)\|\Omega(s_1,s_2)\|_{L^2_w}+\f1\a\|s_1\|_{L^2}\nonumber\\
	&\lesssim \frac{1}{\text{Im}c}\|\Omega(s_1,s_2)\|_{L^2_w}+\f1\a\|s_1\|_{L^2}.
	\end{align}
 For $\pa_Y\varrho,$ differentiating \eqref{3.1.3} yields that
	\begin{align}\label{3.1.36}
	-\cm^{-2}\pa_Y\varrho&=\text{OS}_{\text{CNS}}(\Psi)+\a^2\left(\f{i}{n}\Delta_\a \Psi+(U_s-c)\Psi \right)-\f{i}{\a}\pa_Y(A^{-1}s_1)\nonumber\\
	&=\Omega(s_1,s_2)+\a^2\left(\f{i}{n}\Delta_\a \Psi+(U_s-c)\Psi \right)-\frac{i}{\a}\pa_Y(A^{-1}s_1)\nonumber\\
	&=s_2+\a^2\left(\f{i}{n}\Delta_\a \Psi+(U_s-c)\Psi \right),
	\end{align}
	where we have used the equation \eqref{3.1.4} in  second identity. Taking $L^2$-norm in \eqref{3.1.36} and using bound \eqref{3.1.32}, we can further deduce that
	\begin{align}\label{3.1.35}
	\cm^{-2}\|\pa_Y\varrho\|_{L^2}&\lesssim \|s_2\|_{L^2}+\f{\a^2}{n}\|\pa_Y^2\Psi\|_{L^2}+\a(1+\frac{\a^2}{n})\|\a\Psi\|_{L^2}\nonumber\\
	&\lesssim \|s_2\|_{L^2}+\|\pa_Y^2\Psi\|_{L^2}+\a\|\Psi\|_{L^2}\nonumber\\
	&\lesssim \frac{1}{\text{Im}c}\|\Omega(s_1,s_2)\|_{L^2_w}+\|s_2\|_{L^2}.
	\end{align} 
	Now we estimate $\pa_Y^2\varrho.$  By using \eqref{3.1.3} and \eqref{3.1.36}, we have
\begin{align}\label{3.1.37-1}
-\cm^{-2}\Delta_\a\varrho=\dv(s_1,s_2)+\a^2(U_s-c)^2\varrho+2\a^2\Psi\pa_YU_s.
\end{align}
	Then taking $L^2$ norm leads to
	\begin{align}\label{3.1.37}
	\cm^{-2}\|\pa_Y^2\varrho\|_{L^2}&\lesssim \|\dv(s_1,s_2)\|_{L^2}+\a^2(1+\cm^{-2})\|\varrho\|_{L^2}+\a^2\|\Psi\|_{L^2}\nonumber\\
	&\lesssim \|\dv(s_1,s_2)\|_{L^2}+\cm^{-2}\|\varrho\|_{L^2}+\a\|\Psi\|_{L^2}\nonumber\\
	&\lesssim \frac{1}{\text{Im}c}\|\Omega(s_1,s_2)\|_{L^2_w}+\f1\a\|s_1\|_{L^2}+\|\dv(s_1,s_2)\|_{L^2}.
	\end{align}
	Here we have used \eqref{3.1.32} and \eqref{3.1.34} in the last inequality. Therefore, $H^2$-estimate of $\varrho$ follows from \eqref{3.1.34}, \eqref{3.1.35} and \eqref{3.1.37}. Since $\dv(\cu,\cv)=-i\a(U_s-c)\varrho$, by using \eqref{3.1.34} and \eqref{3.1.35} we have
	\begin{align}
	\a^{-1}\|\dv(\cu,\cv)\|_{H^1}\lesssim
	\|\varrho\|_{H^1}\lesssim \frac{1}{\text{Im}c}\|\Omega(s_1,s_2)\|_{L^2_w}+\f1\a\|s_1\|_{L^2}+\|s_2\|_{L^2}.\label{3.1.38}
	\end{align}
	
	Finally, for $\cu$, by using \eqref{3.1.2}, \eqref{3.1.32}, \eqref{3.1.31}, \eqref{3.1.34}, \eqref{3.1.35} and \eqref{3.1.37}, we obtain that
	\begin{align}
	\|\cu\|_{H^1}&\lesssim \|\pa_Y\Psi\|_{H^1}+\|\varrho\|_{H^1}\lesssim \frac{1}{\text{Im}c}\|\Omega(s_1,s_2)\|_{L^2_w}+\f1\a\|s_1\|_{L^2}+\|s_2\|_{L^2},\label{3.1.38-1}\\
	\|\pa_Y^2\cu\|_{L^2}&\lesssim \|\pa_Y^3\Psi\|_{L^2}+\|\varrho\|_{H^2}\lesssim \frac{n^{\f12}}{(\text{Im}c)^{\f12}}\|\Omega(s_1,s_2)\|_{L^2_w}+\f1\a\|s_1\|_{L^2}+\|s_2\|_{L^2}+\|\dv(s_1,s_2)\|_{L^2}.\label{3.1.38-2}
	\end{align}
	Putting the estimates   in \eqref{3.1.33}, \eqref{3.1.34}, \eqref{3.1.35}, \eqref{3.1.37}-\eqref{3.1.38-1} together yields the  estimate \eqref{3.1.29}. Note that \eqref{3.1.30} directly follows from \eqref{3.1.38-2}. Then the  proof of proposition is  completed.
\end{proof}

\subsection{Stokes approximation} In this section, we study the following Stokes system with advection:
\begin{equation}\label{3.2.1}\left\{
\begin{aligned}
&i\a(U_s-c)\xi+\dv(\phi,\psi)=q_0,~\\
&\sqrt{\vep}\Delta_\a \phi+\lambda i\a\sqrt{\vep}\dv(\phi,\psi)-i\a(U_s-c)\phi-(i\a\cm^{-2}+\sqrt{\vep}\pa_Y^2U_s)\xi=q_1,\\
&\sqrt{\vep}\Delta_\a\psi+\lambda\sqrt{\vep}\pa_Y\dv(\phi,\psi)-i\a(U_s-c)\psi-m^{-2}\pa_Y\xi=q_2,\\
&\pa_Y\phi|_{Y=0}=\psi|_{Y=0}=0,
\end{aligned}\right.
\end{equation}
with a given inhomogeneous source term $\vec{q}=(q_0,q_1,q_2)\in H^1(\mathbb{R}_+)\times L^2(\mathbb{R}_+)^2$.  Compared with original system \eqref{3.0.1}, in \eqref{3.2.1} we remove the stretching term $-\psi\pa_YU_s$ in the momentum equation. We impose the  Neumann boundary condition $\pa_Y\phi|_{Y=0}=0$ on the  tangential velocity for obtaining estimates on the higher order derivatives. The following proposition gives the solvability of \eqref{3.2.1}.

\begin{proposition}\label{prop3.3}
	Let $\cm\in (0,1)$. Assume that $\a\in (0,1)$ and $\f{1}{n}=\frac{\sqrt{\vep}}{\a}\ll1$. There exists $\gamma_5\in (0,1)$, such that for any $c$ lies in
	\begin{align}\label{S}
	  \Sigma_{S}\eqdef\{c \in \mathbb{C}\mid \text{Im}c>\gamma_5^{-1}n^{-1},~|c|<\gamma_5\},
    \end{align}
	 the system \eqref{3.2.1} admits a unique solution $(\xi,\phi,\psi)\in
	H^1(\mathbb{R}_+)\times H^2(\mathbb{R}_+)^2$. Moreover, $(\xi,\phi,\psi)$ satisfies the following estimates:
	\begin{align}
	\|(\cm^{-1}\xi,\phi,\psi)\|_{L^2}&\leq \frac{C}{\a \text{Im}c}\|(\cm^{-1}q_0,q_1,q_2)\|_{L^2}\label{3.2.2},\\
	\|(\pa_Y\phi,\a\phi)\|_{L^2}+\|\pa_Y\psi,\a\psi\|_{L^2}&\leq \frac{Cn^{\f12}}{\a(\text{Im}c)^{\f12}}\|(\cm^{-1}q_0,q_1,q_2)\|_{L^2},\label{3.2.3}\\
	\|\dv(\phi,\psi)\|_{H^1}+\cm^{-2}\|\pa_Y\xi\|_{L^2}&\leq \frac{C}{\text{Im}c}\|(\cm^{-1}q_0,q_1,q_2)\|_{L^2}+C\|q_0\|_{H^1}\label{3.2.4},\\
	\|\Delta_\a\phi,\Delta_\a\psi)\|_{L^2}&\leq \frac{Cn}{\a\text{Im}c}\|(\cm^{-1}q_0,q_1,q_2)\|_{L^2}+C\|q_0\|_{H^1}.\label{3.2.4-1}
	\end{align}
	Here the positive constant $C$ does not depend on either $\a$ or $\vep$.
	Furthermore, if $\vec{q}(\cdot~;c)$ is analytic in $c$  in $H^1(\mathbb{R}_+)\times L^2(\mathbb{R}_+)^2$, then $(\xi,\phi,\psi)(\cdot~;c)$ is analytic in $c$ in $H^{1}(\mathbb{R}_+)\times H^2(\mathbb{R}_+)^2$.
\end{proposition}
\begin{remark}
We will use  the bounds given in \eqref{3.2.2}-\eqref{3.2.4-1} only when $q_0=0$ in the  proof of convergence of iteration.
\end{remark}
\begin{remark}
	In view of \eqref{3.2.3}-\eqref{3.2.4-1} with $q_0=0,$ the divergence part $\dv(\phi,\psi)$
	and the density $\xi$ of the solution have better estimates than other components because there is no
	strong sublayer related to these two fluid components. This stronger estimate is crucial in the proof of convergence of the iteration later.
\end{remark}
\begin{proof}
	We first focus on the a priori estimates \eqref{3.2.2}-\eqref{3.2.4-1}. By taking inner product of $\eqref{3.2.1}_2$ and $\eqref{3.2.1}_3$ with $-\b{\phi}$ and $-\b{\psi}$ respectively then integrating by parts, we obtain
	\begin{align}\label{3.2.5}
	&\sqrt{\vep}\left(\|(\pa_Y\phi,\a\phi)\|_{L^2}^2+\|(\pa_Y\psi,\a \psi)\|_{L^2}^2\right)+\lambda\sqrt{\vep}\|\dv(\phi,\psi)\|_{L^2}^2+i\a\int_0^\infty (U_s-c)\left(|\phi|^2+|\psi|^2\right)\dd Y\nonumber\\
	&\qquad-\cm^{-2}\int_0^\infty \xi\overline{\dv(\phi,\psi)}\dd Y=\int_0^\infty -(q_1+\sqrt{\vep}\pa_Y^2U_s\xi)\b{\phi}-q_2\b{\psi}\dd Y.
	\end{align}
	By Cauchy-Schwarz and Young's inequalities, it holds that
	\begin{align}\label{3.2.5-1}
	\left|\int_0^\infty (q_1+\sqrt{\vep}\pa_Y^2U_s\xi)\b{\phi}+q_2\b{\psi}\dd Y\right|&\lesssim \|(q_1,q_2)\|_{L^2}\|(\phi,\psi)\|_{L^2}+\sqrt{\vep}\|\xi\|_{L^2}\|\phi\|_{L^2}\nonumber\\
	&\lesssim \|(q_1,q_2)\|_{L^2}\|(\phi,\psi)\|_{L^2} +C\sqrt{\vep}(\|\xi\|_{L^2}^2+\|\phi\|_{L^2}^2).
	\end{align}
	By using the continuity equation, $\overline{\dv(\phi,\psi)}=i\a(U_s-\b{c})\b{\xi}+\b{q}_0$,  and the Cauchy-Schwarz inequality, we get
	\begin{align}\label{3.2.5-2}
	\text{Re}\left(-\cm^{-2}\int_0^\infty\xi\overline{\dv(\phi,\psi)}\dd Y\right)&=\text{Re}\left(-i\a\cm^{-2}\int_0^\infty(U_s-\b{c})|\xi|^2\dd Y\right)-\cm^{-2}\text{Re}\left(\int_0^\infty\xi\b{q}_0\dd Y\right)\nonumber\\
	&\geq\a\text{Im}c\|\cm^{-1}\xi\|_{L^2}^2-\cm^{-2}\|\xi\|_{L^2}\|q_0\|_{L^2}.
	\end{align}
	By \eqref{3.2.5-1} and \eqref{3.2.5-2}, the real part of \eqref{3.2.5} gives that
	\begin{align}
	&\sqrt{\vep}\left(\|(\pa_Y\phi,\a\phi)\|_{L^2}^2+\|(\pa_Y\psi,\a \psi)\|_{L^2}^2\right)+\lambda\sqrt{\vep}\|\dv(\phi,\psi)\|_{L^2}^2+\a\text{Im}c\|(\cm^{-1}\xi,\phi,\psi)\|_{L^2}^2
	\nonumber\\
	&\quad\qquad\leq C\sqrt{\vep}(\|\xi\|_{L^2}^2+\|\phi\|_{L^2}^2)+ C\|(\cm^{-1}\xi,\phi,\psi)\|_{L^2}\|(\cm^{-1}q_0,q_1,q_2)\|_{L^2}.\label{3.2.5-3}
	\end{align}
	By taking $\gamma_5\in (0,1)$ sufficiently small so that $\frac{C\sqrt{\vep}}{a\text{Im}c}\leq \frac{C}{n\text{Im}c}\leq C\gamma_5\leq \frac{1}{4},~\forall c\in \Sigma_S$, we can absorb the first term on the right hand side of \eqref{3.2.5-3} by the left hand side. Thus we get 
	$$\|(\cm^{-1}\xi,\phi,\psi)\|_{L^2}\leq \frac{C}{\a\text{Im}c}\|(\cm^{-1}q_0,q_1,q_2)\|_{L^2},$$
	and
	\begin{align}\label{3.2.5-4}
	\|(\pa_Y\phi,\a\phi)\|_{L^2}+\|(\pa_Y\psi,\a \psi)\|_{L^2}&\leq C\vep^{-\f14}\|(\cm^{-1}\xi,\phi,\psi)\|_{L^2}^{\f12}\|(\cm^{-1}q_0,q_1,q_2)\|_{L^2}^{\f12}.\nonumber\\
	&\leq \frac{C}{\vep^{\f14}\a^{\f12}(\text{Im}c)^{\f12}}\|(\cm^{-1}q_0,q_1,q_2)\|_{L^2}\nonumber\\
	&\leq \frac{Cn^{\f12}}{\a(\text{Im}c)^{\f12}}|(\cm^{-1}q_0,q_1,q_2)\|_{L^2}.
	\end{align}
	This completes the proof of \eqref{3.2.2} and \eqref{3.2.3}. 
	
	Next we estimate  $\|\pa_Y\xi\|_{L^2}$ and $\|\dv(\phi,\psi)\|_{H^1}$. Define $\omega\eqdef \pa_Y\phi-i\a\psi$ and denote $\mathcal{D}:=\dv(\phi,\psi)$. Then 
	\begin{align}\label{3.2.6-1}
	\Delta_\a\phi=\pa_Y\omega+i\a \CD,~\Delta_\a\psi=-i\a\omega+\pa_Y\CD,
	\end{align}
	and $\omega|_{Y=0}=0$ because of  the boundary conditions in \eqref{3.2.1}. Thus, we can rewrite $\eqref{3.2.1}_2$ and $\eqref{3.2.1}_3$ as
	\begin{align}
	i\a\cm^{-2}\xi&=\sqrt{\vep}\pa_Y\omega+\sqrt{\vep}(1+\lambda) i\a \CD-i\a(U_s-c)\phi-q_1-\sqrt{\vep}\pa_Y^2U_s\xi,\label{3.2.6-2}\\
	\cm^{-2}\pa_Y\xi&=-\sqrt{\vep}i\a \omega+\sqrt{\vep}(1+\lambda)\pa_Y\CD-i\a(U_s-c)\psi-q_2.\label{3.2.6-3}
	\end{align}
	By taking inner product of \eqref{3.2.6-2} and \eqref{3.2.6-3} with $-i\a\b{\xi}$ and $\pa_Y\b{\xi}$ respectively, we deduce that
	\begin{align}
	\cm^{-2}\|\pa_Y\xi,\a\xi\|_{L^2}^2=&\underbrace{\sqrt{\vep}\int_0^\infty -\pa_Y\omega i\a\b{\xi}-i\a\omega\pa_Y\b{\xi}\dd Y}_{J_5}+\underbrace{\sqrt{\vep}(1+\lambda)\int_0^\infty \a^2\CD\b{\xi}+\pa_Y\CD\pa_Y\b{\xi}\dd Y}_{J_6}\nonumber\\
	&+\underbrace{\int_0^\infty(q_1+\sqrt{\vep}\pa_Y^2U_s\xi)i\a\b{\xi}-q_2\pa_Y\b{\xi}\dd Y}_{J_7}+\underbrace{\int_0^\infty i\a(U_s-c)(i\a\b{\xi} \phi-\psi\pa_Y\b{\xi})\dd Y}_{J_8}.\label{3.2.6}
	\end{align}
	Integrating by parts and using boundary condition $\omega|_{Y=0}=0$ yield that
	\begin{align}
	 J_5=i\a\sqrt{\vep}\b{\xi}\omega|_{Y=0}=0.\label{3.2.6-5}
	\end{align}
	 For $J_6$, by using the continuity equation $\eqref{3.2.1}_1$, we have
	\begin{align}\label{3.2.6-4}
	\CD=-i\a(U_s-{c})\xi+q_0,~\pa_Y\CD=-i\a(U_s-c)\pa_Y\xi-i\a\pa_YU_s{\xi}+\pa_Yq_0,
	\end{align}
	which implies that
	\begin{align}
	J_6=&-\sqrt{\vep}(1+\lambda)\int_0^\infty i\a(U_s-c)(|\pa_Y\xi|^2+\a^2|\xi|^2)\dd Y-i\a\sqrt{\vep}(1+\lambda)\int_0^\infty \pa_YU_s\xi\pa_Y\b{\xi}\dd Y\nonumber\\
	&+\sqrt{\vep}(1+\lambda)\int_0^\infty \pa_Yq_0\pa_Y\b{\xi}+\a^2q_0\b{\xi}\dd Y.\nonumber
	\end{align}
	For last two terms on the right hand side, we obtain by Cauchy-Schwarz and Young's inequalities that 
	$$\left|i\a\sqrt{\vep}(1+\lambda)\int_0^\infty \pa_YU_s{\xi}\pa_Y\b{\xi}\dd Y\right|\leq C\sqrt{\vep}\|\pa_Y\xi\|_{L^2}\|\a\xi\|_{L^2}\leq C\sqrt{\vep}(\|\pa_Y\xi\|_{L^2}^2+\a^2\|\xi\|_{L^2}^2),
	$$
	and
	$$
	\begin{aligned}
	\sqrt{\vep}(1+\lambda)\left|\int_0^\infty \pa_Yq_0\pa_Y\b{\xi}+\a^2q_0\b{\xi}\dd Y\right|&\leq C\sqrt{\vep}\|(\pa_Y\xi,\a\xi)\|_{L^2}\|(\pa_Yq_0,\a q_0)\|_{L^2}\\
	&\leq \frac{\cm^{-2}}{8}\|(\pa_Y\xi,\a\xi)\|_{L^2}^2+C\cm^{2}\vep\|(\pa_Yq_0,\a q_0)\|_{L^2}^2.
	\end{aligned}
	$$
	Thus,  taking real part of $J_6$ gives
	\begin{align}
	\text{Re}J_6\leq & -\a\text{Im}c\sqrt{\vep}(1+\lambda)\|(\pa_Y\xi,\a\xi)\|_{L^2}^2+\left(C\sqrt{\vep}+\frac{\cm^{-2}}{8}\right)\|(\pa_Y\xi,\a\xi)\|_{L^2}^2+C\cm^{2}\vep\|(\pa_Yq_0,\a q_0)\|_{L^2}^2\nonumber\\
	\leq & \frac{\cm^{-2}}{4}\|(\pa_Y\xi,\a\xi)\|_{L^2}^2+C\cm^2\vep\|(\pa_Yq_0,\a q_0)\|_{L^2}^2,\label{3.2.6-6}
	\end{align}
	for $0<\vep\ll1$ being sufficiently small. Again by Young's inequality, we get 
	\begin{align}
	|J_7|+|J_8|&\leq C\left(\|(q_1,q_2)\|_{L^2}+\a\|(\phi,\psi)\|_{L^2}\right)\|(\pa_Y\xi,\a \xi)\|_{L^2}+\f{C\sqrt{\vep}}{\a}\|\a\xi\|_{L^2}^2\nonumber\\
	&\leq\frac{\cm^{-2}}{4}\|(\pa_Y\xi,\a\xi)\|_{L^2}^2+C\cm^2(\|(q_1,q_2)\|_{L^2}^2+\a^2\|(\phi,\psi)\|_{L^2}^2),\label{3.2.6-7}
	\end{align}
	where we have used $\frac{1}{n}=\f{\sqrt{\vep}}{\a}\ll1$.
	By \eqref{3.2.6-5}, \eqref{3.2.6-6} and \eqref{3.2.6-7}, the real part of \eqref{3.2.6} yields
	\begin{align}
	\cm^{-2}\|(\pa_Y\xi,\a\xi)\|_{L^2}&\leq C \|(q_1,q_2)\|_{L^2}+C\sqrt{\vep}\|(\pa_Yq_0,\a q_0)\|_{L^2}+C\a\|(\phi,\psi)\|_{L^2}.\label{3.2.7-1}
	\end{align}
	Moreover, by \eqref{3.2.6-4} and \eqref{3.2.7-1}, we obtain
	\begin{align}
	\|\dv(\phi,\psi)\|_{H^1}&\leq C\|(\pa_Y\xi,\a\xi)\|_{L^2}+C\|q_0\|_{H^1}\nonumber\\
	&\leq C\|(q_1,q_2)\|_{L^2}+C\|q_0\|_{H^1}+C\a\|(\phi,\psi)\|_{L^2}.\label{3.2.7}
	\end{align}
	Putting the bound \eqref{3.2.2} on $\|(\phi,\psi)\|_{L^2}$ into \eqref{3.2.7-1} and \eqref{3.2.7} yields the  estimate 
$$\begin{aligned}
\|\dv(\phi,\psi)\|_{H^1}+\cm^{-2}\|(\pa_Y\xi,\a\xi)\|_{L^2}&\leq C(1+\frac{1}{\text{Im}c})\|(\cm^{-1}q_0,q_1,q_2)\|_{L^2}+C(1+\sqrt{\vep})\|q_0\|_{H^1}\\
&\leq \frac{C}{\text{Im}c}\|(\cm^{-1}q_0,q_1,q_2)\|_{L^2}+\|q_0\|_{H^1}.
\end{aligned}
$$	
Hence,  \eqref{3.2.4} holds.
	
	Finally, we derive the estimate on $\|(\pa_Y\omega,\a\omega)\|_{L^2}$. By taking inner products of
	  $\eqref{3.2.6-2}$ and $\eqref{3.2.6-3}$ with $\pa_Y\b{\omega}$ and $i\a\b{\omega}$ respectively then using the fact that
	$$\int_0^\infty(i\a \CD\pa_Y\b{\omega}+i\a\b{\omega}\pa_Y \CD )\dd Y=\int_0^\infty (i\a\xi \pa_Y\b{\omega}+\pa_Y\xi i\a\b{\omega})\dd Y=0,
	$$ we obtain 
	$$
	\begin{aligned}
	\sqrt{\vep}\|(\pa_Y\omega,\a\omega)\|_{L^2}^2&=\int_0^\infty (q_1+\sqrt{\vep}\pa_Y^2U_s\xi)\pa_Y\b{\omega}+q_2i\a\b{\omega}\dd Y+\int_0^\infty i\a(U_s-c)\left(\phi\pa_Y\omega+\psi i\a\b{\omega} \right)\dd Y\\
	&\leq C\left( \|(q_1,q_2)\|_{L^2}+\sqrt{\vep}\|\xi\|_{L^2}+\a\|(\phi,\psi)\|_{L^2}\right)\|(\pa_Y\omega,\a\omega)\|_{L^2},
	\end{aligned}
	$$
	which implies
	\begin{align}\label{3.2.7-2}
	\|(\pa_Y\omega,\a\omega)\|_{L^2}&\leq \frac{C}{\sqrt{\vep}}\|(q_1,q_2)\|_{L^2}+Cn\|(\xi,\phi,\psi)\|_{L^2}\nonumber\\
	&\leq \frac{Cn}{\a\text{Im}c}(1+\text{Im}c)\|(\cm^{-1}q_0,q_1,q_2)\|_{L^2}\nonumber\\
	&\leq \frac{Cn}{\a\text{Im}c}\|(\cm^{-1}q_0,q_1,q_2)\|_{L^2}.
	\end{align}
 By combining this with the bound \eqref{3.2.4} on $\|\dv(\phi,\psi)\|_{H^1}$ and recalling \eqref{3.2.6-1}, we have
	$$
	\begin{aligned}
	\|(\Delta_\a\phi,\Delta_\a\psi)\|&\leq C\|(\pa_Y\omega,\a\omega)\|_{L^2}+C\|\dv(\phi,\psi)\|_{H^1}\nonumber\\
	&\leq \frac{Cn}{\a\text{Im}c}\left(1+\sqrt{\vep}\right)\|(\cm^{-1}q_0,q_1,q_2)\|_{L^2}+C\|q_0\|_{H^1}\nonumber\\
	&\leq \frac{Cn}{\a\text{Im}c}\|(\cm^{-1}q_0,q_1,q_2)\|_{L^2}+C\|q_0\|_{H^1},
	\end{aligned}
	$$
which is \eqref{3.2.4-1}. The uniqueness of solution follows from the a priori bounds \eqref{3.2.2}-\eqref{3.2.4-1}.
	
As for the construction of solution, we introduce a parameter $\eta\in [0,1]$ and study a auxiliary problem $L_{S,\eta}(\xi^\eta,\phi^\eta,\psi^{\eta})=(q_0,q_1,q_2)$ as follows:
	\begin{equation}
	\left\{
	\begin{aligned}
	&i\a\eta(U_s-c)\xi^\eta+\dv(\phi^\eta,\psi^\eta)=q_0,\\
	&\sqrt{\vep}\Delta_\a \phi^\eta+\eta\lambda i\a\sqrt{\vep}\dv(\phi^\eta,\psi^\eta)-i\a\eta(U_s-c)\phi^\eta-(i\a\cm^{-2}+\eta\sqrt{\vep}\pa_Y^2U_s)\xi^\eta=q_1,\\
	&\sqrt{\vep}\Delta_\a \psi^\eta+\eta\lambda \sqrt{\vep}\pa_Y\dv(\phi^\eta,\psi^\eta)-i\a\eta(U_s-c)\psi^\eta-\cm^{-2}\pa_Y\xi^\eta=q_2,\label{3.2.8}\\
	&\pa_Y\phi^\eta|_{Y=0}=\psi^\eta|_{Y=0}=0.
	\end{aligned}
	\right.
	\end{equation}
	When $\eta=0,$ \eqref{3.2.8} reduces to the classical Stokes system for incompressible flow:
	\begin{align}
	\dv(\phi^0,\psi^0)=q_0,~\sqrt{\vep}\Delta_\a \phi^0-i\a\cm^{-2}\xi^0=q_1,~
	\sqrt{\vep}\Delta_\a \psi^0-\cm^{-2}\pa_Y\xi^0=q_2,~\pa_Y\phi^0|_{Y=0}=\psi^0|_{Y=0}=0.\nonumber
	\end{align}
	It is standard to show the existence and uniqueness of  solution $(\xi^0,\phi^0,\psi^0)\in H^1(\mathbb{R}_+)\times H^2(\mathbb{R}_+)^2$  for any $(q_0,q_1,q_2)\in H^1(\mathbb{R}_+)\times L^2(\mathbb{R}_+)^2.$ Moreover, by repeating previous energy estimates to \eqref{3.2.8} and slightly modifying the proof of bounds \eqref{3.2.5-4}, \eqref{3.2.7-1}, \eqref{3.2.7} and \eqref{3.2.7-2},
	one can deduce the following estimates on $(\xi^\eta,\phi^\eta,\psi^\eta)$
	\begin{align}
	\|(\pa_Y\phi^\eta,\a\phi^\eta)\|_{L^2}+\|(\pa_Y\psi^\eta,\a\psi^\eta)\|_{L^2}&\leq C\vep^{-\f14}\|(\cm^{-1}\xi^\eta,\phi^\eta,\psi^\eta)\|_{L^2}^{\f12}
	\|(\cm^{-1}q_0,q_1,q_2)\|_{L^2}^{\f12},\nonumber\\
	\cm^{-2}\|(\pa_Y\xi^\eta,\a\xi^\eta)\|_{L^2}+\|\dv(\phi^\eta,\psi^\eta)\|_{H^1}&\leq C\|(q_1,q_2)\|_{L^2}+C\|q_0\|_{H^1}+C\a\|(\phi^\eta,\psi^\eta)\|_{L^2},\nonumber\\
	\|(\pa_Y\omega^\eta,\a\omega^\eta)\|_{L^2}&\leq \frac{C}{\sqrt{\vep}}\|(q_1,q_2)\|_{L^2}+Cn\|(\cm^{-1}\xi^\eta,\phi^\eta,\psi^\eta)\|_{L^2},\nonumber
	\end{align}
	where $\omega^\eta=\pa_Y\phi^\eta-i\a\psi^\eta$ and the constant $C>0$ does not depend on $\eta$. Putting above inequalities together yields the following uniform-in-$\eta$ estimate
	$$\|\xi^\eta\|_{H^1}+\|(\phi^\eta,\psi^\eta)\|_{H^2}\leq C(\vep,\a)\|(\cm^{-1}q_0,q_1,q_2)\|_{L^2}+C(\vep,\a)\|q_0\|_{H^1},
	$$
	where the constant $C(\vep,\a)$ may depend on $\vep$ and $\a$, but not on $\eta\in [0,1]$. Thus the existence of solution to \eqref{3.2.1} as well as its analytic dependence on $c$ can be established by the same bootstrap argument as in Lemma \ref{lm2.5}.  By uniqueness, the solution obtained  satisfies the bounds \eqref{3.2.2}-\eqref{3.2.4-1}. And this completes the proof of the proposition. 
\end{proof}

\subsection{Quasi-compressible-Stokes iteration} 
In this subsection, we will construct a solution $\vec{\Xi}=(\varrho,\cu,\cv)$ to the linearized system \eqref{3.0.1} via an  iteration scheme based on the solutions to quasi-compressible and Stokes approximations given in Propositions  \ref{prop3.2} and \ref{prop3.3}.

We first consider the case when source term $(f_u,f_v)\in L^2(\mathbb{R}_+)^2$. At zeroth step, we define $\vec{\Xi}_0=(\xi_0,\phi_0,\psi_0)$ as the solution to Stokes approximate system
\begin{align}\label{3.3.1}
L_S(\xi_0,\phi_0,\psi_0)=(0,f_u,f_v),
\end{align}
which  yields an error $$\vec{\CE}_0\eqdef\CL(\xi_0,\phi_0,\psi_0)-L_Q(\xi_0,\phi_0,\psi_0)=\left(0,-\psi_0\pa_YU_s,0\right).$$
Because of the regularizing effect of solution operator to Stokes approximation $L_S$, this error has higher regularity and  fast decay so that $\vec{\CE}_0\in H^2_w(\mathbb{R}_+)$. We can then eliminate it by considering $(\varrho_1,\cu_1,\cv_1)$ as the solution to quasi-compressible approximation
\begin{align}\label{3.3.1-1}
L_Q(\varrho_1,\cu_1,\cv_1)=-\vec{\CE}_0.
\end{align}
Then we have
\begin{align}
&\CL(\varrho_1,\cu_1,\cv_1)-L_Q(\varrho_1,\cu_1,\cv_1)=E_Q(\varrho_1,\cu_1,\cv_1),\label{3.3.1-2}
\end{align}
where the error operator $E_Q$ is defined in \eqref{t4}.
According to Proposition \ref{prop3.2}, the solution $(\varrho_1,\cu_1,\cv_1)$ is in $ H^2(\mathbb{R}_+)^3$.
Thus the error term $E_Q(\varrho_1,\cu_1,\cv_1)$ is in $L^2(\mathbb{R}_+)$.
This allows us to correct this error by using the solution $(\xi_1,\phi_1,\psi_1)$ to the Stokes approximate system again:
\begin{align}
L_S(\xi_1,\phi_1,\psi_1)=-E_Q(\varrho_1,\cu_1,\cv_1)\label{3.3.1-3}
\end{align}
Now we set $\vec{\Xi}_1=(\varrho_1,\cu_1,\cv_1)+(\xi_1,\phi_1,\psi_1)$ as the approximate solution as the first step, which together with $\vec{\Xi}_0$ generates an error term
$$\vec{\CE}_1\eqdef \CL(\vec{\Xi}_0+\vec{\Xi}_1)-(0,f_u,f_v)
=\left(0,-\psi_1\pa_YU_s,0\right).
$$
Now we can iterate the above process.  Given the approximate solution $\vec{\Xi}_j$ as well as the error $$\vec{\CE}_{j}=\left(0,-\psi_j\pa_YU_s,0\right)$$
in the $j$-th ($j\geq 1$) step,
we define the $j+1$-order approximate solution $\vec{\Xi}_{j+1}$ as
$$
\vec{\Xi}_{j+1}=(\varrho_{j+1},\cu_{j+1},\cv_{j+1})+(\xi_{j+1},\phi_{j+1},\psi_{j+1}),
$$
where $(\varrho_{j+1},\cu_{j+1},\cv_{j+1})$ is the solution to quasi-compressible system
\begin{align}
L_Q(\varrho_{j+1},\cu_{j+1},\cv_{j+1})=-\vec{\CE}_{j},\label{3.3.1-4}
\end{align}
and
$(\xi_{j+1},\phi_{j+1},\psi_{j+1})$ solves the Stokes approximate system
\begin{align}
&L_S(\xi_{j+1},\phi_{j+1},\psi_{j+1})=-E_Q(\varrho_{j+1},\cu_{j+1},\cv_{j+1}).\label{3.3.1-5}
\end{align}
Observe that for each positive integer $N\geq 0$, it holds 
$$\CL(\sum_{j=0}^N\vec{\Xi}_j)=(0,f_u,f_v)+\vec{\CE}_N,
$$
where the error term in $N$-th step is $\vec{\CE}_N=\left(0,-\psi_N\pa_YU_s,0\right)$. Therefore, at 
this point, formally 
the series 
$\vec{\Xi}=\sum_{j=0}^\infty \vec{\Xi}_j$ gives a solution to the original system \eqref{3.0.1}.\\

If in addition  $f_u,f_v\in H^1(\mathbb{R}_+)$ and $\|\Omega(f_u,f_v)\|_{L^2_w}<\infty$ where operator $\Omega$ is defined in \eqref{add3.1.5}, then we introduce $(\varrho_0,\cu_0,\cv_0)\in H^2(\mathbb{R}_+)^3$ as the solution to the quasi-compressible system
\begin{align}\label{3.3N.14}
L_Q(\varrho_0,\cu_0,\cv_0)=(0,f_u,f_v),
\end{align}
which yields an error term
\begin{align}\label{3.3N.15}
\vec{\CE}_{-1}&\eqdef\CL(\varrho_0,\cu_0,\cv_0)-L_Q(\varrho_0,\cu_0,\cv_0)\nonumber\\
&=\bigg(0,-\sqrt{\vep}\Delta_\a\big[(U_s-c)\varrho_0\big]+\lambda i\a\sqrt{\vep}\dv(\cu_0,\cv_0)-\sqrt{\vep}\pa_Y^2U_s\varrho_{0},\lambda\sqrt{\vep}\pa_Y\dv(\cu_0,\cv_0)\bigg)\nonumber\\
&=E_Q(\varrho_0,\cu_0,\cv_0).
\end{align}
The new error term $\CE_{-1}$ is in $L^2(\mathbb{R}_+)^3$. So we can take $\vec{\Upsilon}=(\tilde{\rho},\tilde{u},\tilde{v})$ as the solution to  original linear system \eqref{3.0.1} with inhomogeneous source term $-\CE_{-1}$, i.e. $\CL(\vec{\Upsilon})=-\vec{\CE}_{-1}$. Then it is clear that $\vec{\Xi}\eqdef(\varrho_0,\cu_0,\cv_0)+\vec{\Upsilon}$ defines a solution to \eqref{3.0.1}. 

The above iteration can be rigorously justified by proving the convergence of  iteration that is given in  Proposition \ref{prop3.0}.\\

{\bf Proof of Proposition \ref{prop3.0}:} Recall  the bounds on the parameters $|c|$ and $n$
in \eqref{para}. We can take $0<\vep\ll1$ suitably small such that the following bounds hold for any $c\in \overline{D_0}:$
\begin{align}\label{bc}
|c|<\min\{\gamma_4,\gamma_5\},~n\text{Im}c\gtrsim \vep^{-\f14}\geq\max\{ 2\gamma_4^{-1},2\gamma_5^{-1}\},~|c|^{-2}\text{Im}c\gtrsim \vep^{-\f18}\gtrsim 2\gamma_4^{-1},
\end{align}
where the constants $\g_4$ and $\g_5$ are given in Proposition \ref{prop3.2} and \ref{prop3.3} respectively. Thus,  we have $\overline{D_0}\subsetneq \Sigma_{Q}\cap\Sigma_{S}$, where $\Sigma_Q$ and $\Sigma_S$ are resolvent sets of $L_Q$ and $L_S$, which are defined in \eqref{Q} and \eqref{S} respectively.
From \eqref{3.3.1-4}, we know that $(\varrho_{j+1},\cu_{j+1},\cv_{j+1})$ is the solution to quasi-compressible approximation \eqref{3.1.1} with inhomogeneous source term $s_{1,j+1}=\psi_j\pa_YU_s$, $s_{2,j+1}=0.$
Then we have
$$\Omega(s_{1,j+1},s_{2,j+1})=\f{i}{\a}\pa_Y\left( A^{-1}\psi_j\pa_YU_s\right)=-\f{i}{\a}\cw^{-1}\psi_j+\f{i}{\a}A^{-1}\pa_YU_s\pa_Y\psi_j.
$$
To eliminate the singular factor $\a^{-1}$, we 
use the fact that $\pa_Y\psi_j=\dv(\phi_j,\psi_j)-i\a\phi_j$, 
the two bounds given in \eqref{A4}, \eqref{3.1.8} and Hardy inequality to obtain
\begin{align}
\|\Omega(s_{1,j+1},s_{2,j+1})\|_{L^2_w}&\lesssim \frac{1}{\a}\||\pa_Y^2U_s|^{-\f12}|\cw|^{-1} Y\|_{L^\infty}\|Y^{-1}\psi_j\|_{L^2}+\f1\a\||\pa_Y^2U_s|^{-\f12}\pa_YU_s\|_{L^\infty}\|\pa_Y\psi_j\|_{L^2}\nonumber\\
&\lesssim \f{1}{\a}\|\pa_Y\psi_j\|_{L^2}\lesssim \f1\a\|\dv(\phi_j,\psi_j)\|_{L^2}+\|\phi_j\|_{L^2}.\label{3.3N.4}
\end{align}
Similarly, we get
\begin{align}
\|s_{1,j+1}\|_{L^2}&\lesssim \|Y\pa_YU_s\|_{L^\infty}\|Y^{-1}\psi_{j}\|_{L^2}\lesssim\|\pa_Y\psi_j\|_{L^2}\lesssim \|\dv(\phi_j,\psi_j)\|_{L^2}+\a\|\phi_j\|_{L^2}.\label{3.3N.2}\\
\|\dv(s_{1,j+1},s_{2,j+1})\|_{L^2}&\lesssim \a\|s_{1,j+1}\|_{L^2}\lesssim \a\|\dv(\phi_j,\psi_j)\|_{L^2}+\a^2\|\phi_j\|_{L^2}.\label{3.3N.3}
\end{align}
Thus, by applying bounds given in  \eqref{3.1.29} and \eqref{3.1.30} in Proposition \ref{prop3.2} to  $(\varrho_{j+1},\cu_{j+1},\cv_{j+1})$ and using \eqref{3.3N.4}-\eqref{3.3N.3},
we obtain 
\begin{align}
&\|\cu_{j+1}\|_{H^1}+\|(\cm^{-2}\varrho_{j+1},\cv_{j+1})\|_{H^2}+\a^{-1}\|\dv(\cu_{j+1},\cv_{j+1})\|_{H^1}\nonumber\\
&\qquad\lesssim \frac{1}{\text{Im}c}\|\Omega(s_{1,j+1},s_{2,j+1})\|_{L^2_w}+\a^{-1}(\|s_{1,j+1}\|_{L^2}+\a\|s_{2,j+1}\|_{L^2})+\|\dv(s_{1,j+1},s_{2,j+1})\|_{L^2}
\nonumber\\
&\qquad\lesssim \left(1+\a^2+\frac{1}{\text{Im}c}\right)\left(\a^{-1}\|\dv(\phi_j,\psi_j)\|_{L^2}+\|\phi_j\|_{L^2}\right)\nonumber\\
&\qquad\lesssim\frac{1}{\text{Im}c}\left(\a^{-1}\|\dv(\phi_j,\psi_j)\|_{L^2}+\|\phi_j\|_{L^2}\right),\label{3.3N.5}
\end{align}
and
\begin{align}
\|\pa_Y^2u_{j+1}\|_{L^2}
&\lesssim \frac{n^{\f12}}{(\text{Im}c)^{\f12}}\|\Omega(s_{1,j+1},s_{2,j+1})\|_{L^2_w}+\a^{-1}(\|s_{1,j+1}\|_{L^2}+\a\|s_{2,j+1}\|_{L^2})+\|\dv(s_{1,j+1},s_{2,j+1})\|_{L^2}\nonumber\\ 
&\lesssim\left(1+\a^2+\frac{n^{\f12}}{(\text{Im}c)^{\f12}}\right)\left(\a^{-1}\|\dv(\phi_j,\psi_j)\|_{L^2}+\|\phi_j\|_{L^2}\right)\nonumber\\
&\lesssim \frac{n^{\f12}}{(\text{Im}c)^{\f12}}\left(\a^{-1}\|\dv(\phi_j,\psi_j)\|_{L^2}+\|\phi_j\|_{L^2}\right) .\label{3.3N.6}
\end{align}
Here we have also used $\a\in (0,1)$.

Next according to \eqref{3.3.1-5}, we solve $(\xi_{j+1},\phi_{j+1},\psi_{j+1})$ from the Stokes approximation \eqref{3.2.1} with inhomogeneous source term $q_{0,j+1}=0$,
$$q_{1,j+1}=\sqrt{\vep}\Delta_\a\big[(U_s-c)\varrho_{j+1}\big]-\lambda\sqrt{\vep}i\a\dv(\cu_{j+1},\cv_{j+1})+\sqrt{\vep}\pa_Y^2U_s\varrho_{j+1},~ q_{2,j+1}=-\lambda\sqrt{\vep}\pa_Y\dv(\cu_{j+1},\cv_{j+1}).$$
By \eqref{3.3N.5}, we have
\begin{align}\label{3.3N.7}
\|{q}_{1,j+1},q_{2,j+1}\|_{L^2}&\lesssim \sqrt{\vep}\left(\|\varrho_{j+1}\|_{H^2}+\|\dv(\cu_{j+1},\cv_{j+1})\|_{H^1} \right)\nonumber\\
&\lesssim \sqrt{\vep}\left(\|\cm^{-2}\varrho_{j+1}\|_{H^2}+\a^{-1}\|\dv(\cu_{j+1},\cv_{j+1})\|_{H^1} \right)\nonumber\\
&\lesssim \f{\sqrt{\vep}}{\text{Im}c}\left(\a^{-1}\|\dv(\phi_j,\psi_j)\|_{L^2}+\|\phi_j\|_{L^2}\right).
\end{align}
Then by applying \eqref{3.2.2}-\eqref{3.2.4-1} in Proposition \ref{prop3.3} to $(\xi_{j+1},\phi_{j+1},\psi_{j+1})$, using \eqref{3.3N.7} and $\a=n\sqrt{\vep}$, we can deduce that
\begin{align}
\|(\cm^{-1}\xi_{j+1},\phi_{j+1},\psi_{j+1})\|_{L^2}
&\lesssim \frac{1}{n(\text{Im}c)^2}\left(\a^{-1}\|\dv(\phi_j,\psi_j)\|_{L^2}+\|\phi_j\|_{L^2}\right),\label{3.3N.8}\\
\a^{-1}\|\dv(\phi_{j+1},\psi_{j+1})\|_{H^1}+\a^{-1}\|\cm^{-2}\pa_Y\xi_{j+1}\|_{L^2}&\lesssim \frac{1}{n(\text{Im}c)^2}\left(\a^{-1}\|\dv(\phi_j,\psi_j)\|_{L^2}+\|\phi_j\|_{L^2}\right), \label{3.3N.8-1}\\
\|(\pa_Y\phi_{j+1},\pa_Y\psi_{j+1})\|_{L^2}&\lesssim \frac{1}{n^{\f12}(\text{Im}c)^{\f32}}\left(\a^{-1}\|\dv(\phi_j,\psi_j)\|_{L^2}+\|\phi_j\|_{L^2}\right),\label{3.3N.9}\\
\|(\pa_Y^2\phi_{j+1},\pa_Y^2\psi_{j+1})\|_{L^2}&\lesssim \frac{1}{(\text{Im}c)^2}\left(\a^{-1}\|\dv(\phi_j,\psi_j)\|_{L^2}+\|\phi_j\|_{L^2}\right).\label{3.3N.10}
\end{align}

Set
$$E_{j}\eqdef \|(\cm^{-1}\xi_{j},\phi_{j},\psi_{j})\|_{L^2}+\a^{-1}\|\dv(\phi_{j},\psi_{j})\|_{H^1}+\a^{-1}\|\cm^{-2}\pa_Y\xi_{j}\|_{L^2},~ j=0,1,2,\cdots. $$
By the estimates  \eqref{3.3N.8} and \eqref{3.3N.8-1}, we have
\begin{align}\label{3.3N.10-1}
E_{j+1}\leq \frac{C}{n(\text{Im}c)^2}\left( \a^{-1}\|\dv(\phi_j,\psi_j)\|_{L^2}+\|\phi_j\|_{L^2}\right)\leq \frac{C}{n(\text{Im}c)^2}E_j,~j=0,1,2,\cdots.
\end{align}
Recall \eqref{para} for the bounds on $c$ and $n$ when $\a=K\vep^{\f18}$ and $c\in D_0$.
By taking $\vep_3\in (0,1)$ suitably small so that $\frac{C}{n(\text{Im}c)^2}\leq C\vep^{\f18}<\f12$ for any $\vep\in (0,\vep_3)$, we can deduce from \eqref{3.3N.10-1} that
\begin{align}\label{3.3N.11}
\sum_{j=0}^\infty E_j\leq\sum_{j=0}^\infty\left(\f12\right)^jE_0\leq CE_0. 
\end{align}
Furthermore, by using the  bounds obtained in \eqref{3.3N.5}, \eqref{3.3N.6}, \eqref{3.3N.9}-\eqref{3.3N.11} and $\frac{1}{n^{\f12}(\text{Im}c)^{\f32}}\lesssim 1,\forall c\in D_0$,
we get
\begin{align}
&\sum_{j=1}^\infty\|(\pa_Y\phi_{j},\pa_Y\psi_{j})\|_{L^2}\lesssim \frac{1}{n^{\f12}(\text{Im}c)^{\f32}}\left(\sum_{j=0}^\infty E_j\right)\lesssim E_0,\label{3.3N.12}\\
&\sum_{j=1}^\infty\|(\pa^2_Y\phi_{j}, \pa_Y^2\psi_{j})\|_{L^2}\lesssim \frac{1}{(\text{Im}c)^2}\left(\sum_{j=0}^\infty E_j\right)\lesssim \frac{1}{(\text{Im}c)^2}E_0,\label{3.3N.12-1}\\
&\sum_{j=1}^\infty\|\cu_j\|_{H^1}+\sum_{j=1}^\infty\|(\cm^{-2}\varrho_j,\cv_j)\|_{H^2}+\a^{-1}\|\dv(\cu_j,\cv_j)\|_{H^1}\lesssim \frac{1}{\text{Im}c}\left(\sum_{j=0}^\infty E_j\right)\lesssim \frac{1}{\text{Im}c}E_0,\label{3.3N.12-2}\\
&\sum_{j=1}^\infty\|\pa_Y^2\cu_j\|_{L^2}\lesssim \frac{n^{\f12}}{(\text{Im}c)^{\f12}}\left(\sum_{j=0}^\infty E_j\right)\lesssim \frac{n^{\f12}}{(\text{Im}c)^{\f12}}E_0.\label{3.3N.12-3}
\end{align}
In view of \eqref{3.3N.11}-\eqref{3.3N.12-3}, we have justified the convergence of $\vec{\Xi}=(\rho,u,v)=\sum_{j=0}^\infty \vec{\Xi}_j$ in $H^1(\mathbb{R}_+)\times H^2(\mathbb{R}_+)^2$. This gives the existence of solution. Moreover, Recall \eqref{3.3.1}. By applying \eqref{3.2.2}-\eqref{3.2.4-1} to $\vec{\Xi}_0$ with $q_0=0$, $q_1=f_u$ and $q_2=f_v$, we derive the following estimates
\begin{align}
E_0&\lesssim \frac{1}{\a\text{Im}c}\|(f_u,f_v)\|_{L^2},\label{3.3N.13-1}\\
\|(\pa_Y\phi_0,\pa_Y\psi_0)\|_{L^2}&\lesssim \frac{n^{\f12}}{\a(\text{Im}c)^{\f12}}\|(f_u,f_v)\|_{L^2},\label{3.3N.13-2}\\
\|(\pa_Y^2\phi_0,\pa_Y^2\psi_0)\|_{L^2}&\lesssim \frac{n}{\a\text{Im}c}\|(f_u,f_v)\|_{L^2}.\label{3.3N.13}
\end{align}
By summarizing the  estimates \eqref{3.3N.11}-\eqref{3.3N.13}, we have
\begin{align}
\|(\cm^{-1}\rho,u,v)\|_{L^2}&\lesssim \sum_{j=0}^\infty\|(\cm^{-1}\xi_j,\phi_j,\psi_j)\|_{L^2} +\sum_{j=1}^\infty \|(\cm^{-1}\varrho_j,\cu_j,\cv_j)\|_{L^2}\nonumber\\
&\lesssim \left(1+\frac{1}{\text{Im}c}\right)E_0\lesssim \frac{1}{\a(\text{Im}c)^2}\|(f_u,f_v)\|_{L^2},\label{3.3N.14-1}\\
\|\cm^{-2}\pa_Y\rho\|_{L^2}+\|\dv(u,v)\|_{H^1}&\lesssim \a\sum_{j=0}^\infty E_j+\sum_{j=1}^\infty \|\cm^{-2}\pa_Y\varrho_j\|_{L^2}+\sum_{j=1}^\infty\|\dv(\cu_j,\cv_j)\|_{H^1}\nonumber\\
&\lesssim \left(\a+\frac{1}{\text{Im}c}+\frac{\a}{\text{Im}c}\right)E_0\lesssim \frac{1}{\a(\text{Im}c)^2}\|(f_u,f_v)\|_{L^2}, \label{3.3N.14-2}
\end{align}
\begin{align}
\|(\pa_Yu,\pa_Yv)\|_{L^2}&\lesssim \|\pa_Y\phi_0,\pa_Y\psi_0\|_{L^2}+ \sum_{j=1}^\infty \|\pa_Y\phi_j,\pa_Y\psi_j\|_{L^2}+\sum_{j=1}^\infty\|(\pa_Y\cu_j,\pa_Y\cv_j)\|_{L^2}\nonumber\\
&\lesssim \frac{n^{\f12}}{\a(\text{Im}c)^{\f12}}\|(f_u,f_v)\|_{L^2}+\left(1+\frac{1}{\text{Im}c}\right)E_0\nonumber\\
&\lesssim  \frac{n^{\f12}}{\a(\text{Im}c)^{\f12}}\left(1+\frac{1}{n^{\f12}(\text{Im}c)^{\f32}} \right)\|(f_u,f_v)\|_{L^2}
\lesssim \frac{n^{\f12}}{\a(\text{Im}c)^{\f12}}\|(f_u,f_v)\|_{L^2},\label{3.3N.14-3}
\end{align}
and
\begin{align}
\|(\pa_Y^2u,\pa_Y^2v)\|_{L^2}&\lesssim \|(\pa_Y^2\phi_0,\pa_Y^2\psi_0)\|_{L^2}+\sum_{j=1}^\infty\|(\pa_Y^2\phi_j,\pa_Y^2\psi_j)\|_{L^2}+\sum_{j=1}^\infty\|(\pa_Y^2\cu_j,\pa_Y^2\cv_j)\|_{L^2} \nonumber\\
&\lesssim \frac{n}{\a\text{Im}c}\|(f_u,f_v)\|_{L^2}+\left(\frac{1}{(\text{Im}c)^2}+\frac{n^{\f12}}{(\text{Im}c)^{\f12}}\right)E_0 \nonumber\\
&\lesssim 
\frac{n}{\a\text{Im}c}\left(1+\frac{1}{n(\text{Im}c)^2}+\frac{1}{n^{\f12}(\text{Im}c)^{\f12}}\right)\|(f_u,f_v)\|_{L^2}\lesssim \frac{n}{\a\text{Im}c}\|(f_u,f_v)\|_{L^2}.\label{3.3N.14-4}
\end{align}
Putting \eqref{3.3N.14-1}-\eqref{3.3N.14-4} together yields the  estimates \eqref{3.0.2}-\eqref{3.0.4}. The analytic dependence on $c$ of the solution $(\rho,u,v)$ follows from the uniformly convergence. Therefore, the  proof of the first part of Proposition \ref{prop3.0} is completed.

Now we assume that $f_u,f_v\in H^1(\mathbb{R}_+)$ and $\|\Omega(f_u,f_v)\|_{L^2_w}<\infty$. As discussed in the formal presentation of the iteration scheme,   we can decompose the solution $(\rho,u,v)$ into $(\rho,u,v)=(\varrho_0,\cu_0,\cv_0)+\vec{\Upsilon}=(\varrho_0,\cu_0,\cv_0)+(\tilde{\rho},\tilde{u},\tilde{v})$, where $(\varrho_0,\cu_0,\cv_0)$ is  the solution to \eqref{3.3N.14} that generates  an error $\vec{\CE}_{-1}$ defined in \eqref{3.3N.15},
and $\vec{\Upsilon}$ solves $\CL(\vec{\Upsilon})=-\vec{\CE}_{-1}$. By \eqref{3.1.29} and \eqref{3.1.30} in Proposition \ref{prop3.2}, we have
\begin{align}\label{3.3N.17}
&\|\cu_0\|_{H^1}+\|(\cm^{-2}\varrho_0,\cv_0)\|_{H^2}+\a^{-1}\|\dv(\cu_0,\cv_0)\|_{H^1}\nonumber\\
&\qquad\quad\lesssim \frac{1}{\text{Im}c}\|\Omega(f_u,f_v)\|_{L^2_w}+\f1{\a}\|f_u\|+\|f_v\|_{L^2}+\|\dv(f_u,f_v)\|_{L^2},
\end{align}
and
\begin{align}\label{3.3N.18} \|\pa_Y^2\cu_0\|_{L^2}&\lesssim \frac{n^{\f12}}{(\text{Im}c)^{\f12}}\|\Omega(f_u,f_v)\|_{L^2_w}+\f1{\a}\|f_u\|+\|f_v\|_{L^2}+\|\dv(f_u,f_v)\|_{L^2}.
\end{align}
Then we can estimate the $L^2$-bound of the error $\vec{\CE}_{-1}$  by 
\begin{align}\label{3.3N.18-1}
\|\vec{\CE}_{-1}\|_{L^2}\lesssim \sqrt{\vep}\left(\|\varrho_0\|_{H^2}+\|\dv(\cu_0,\cv_0)\|_{H^1}\right).
\end{align}
By \eqref{3.3N.18-1}, applying \eqref{3.0.2}-\eqref{3.0.4} to $\vec{\Upsilon}=(\tilde{\rho},\tilde{u},\tilde{v})$ leads to
\begin{align}
\|(\cm^{-1}\tilde{\rho},\tilde{u},\tilde{v})\|_{L^2}&\lesssim \frac{1}{\a(\text{Im}c)^2}\|\vec{\CE}_{-1}\|_{L^2}\lesssim \frac{1}{n(\text{Im}c)^2}\left(\|\varrho_0\|_{H^2}+\|\dv(\cu_0,\cv_0)\|_{H^1}\right),\label{3.3N.19-2}\\
\|\cm^{-2}\pa_Y\tilde{\rho}\|_{L^2}+\|\dv(\tilde{u},\tilde{v})\|_{L^2}&\lesssim \frac{1}{\a(\text{Im}c)^{2}}\|\vec{\CE}_{-1}\|_{L^2}\lesssim\f{1}{n(\text{Im}c)^{2}}\left(\|\varrho_0\|_{H^2}+\|\dv(\cu_0,\cv_0)\|_{H^1}\right), \label{3.3N.19-3}\\
\|(\pa_Y\tilde{u},\pa_Y\tilde{v})\|_{L^2}&\lesssim \frac{n^{\f12}}{\a(\text{Im}c)^{\f12}}\|\vec{\CE}_{-1}\|_{L^2}\lesssim \frac{1}{n^{\f12}(\text{Im}c)^{\f12}}\left(\|\varrho_0\|_{H^2}+\|\dv(\cu_0,\cv_0)\|_{H^1}\right),\label{3.3N.19-1}\\
\|(\pa_Y^2\tilde{u},\pa_Y^2\tilde{v})\|_{L^2}&\lesssim \frac{n}{\a\text{Im}c}\|\vec{\CE}_{-1}\|_{L^2}\lesssim\f{1}{\text{Im}c}\left(\|\varrho_0\|_{H^2}+\|\dv(\cu_0,\cv_0)\|_{H^1}\right).\label{3.3N.19}
\end{align}
By summarizing the estimates  \eqref{3.3N.17}, \eqref{3.3N.18}, \eqref{3.3N.19-2}-\eqref{3.3N.19} and using the fact that $n(\text{Im}c)^2\gtrsim 1$ and $n^{\f12}(\text{Im}c)^{\f32}\gtrsim 1$ for $c\in D_0$,  we derive the following estimates
$$
\begin{aligned}
\|(\cm^{-1}\rho,u,v)\|_{H^1}&\lesssim \frac{1}{\text{Im}c}\|\Omega(f_u,f_v)\|_{L^2_w}+\f1{\a}\|f_u\|+\|f_v\|_{L^2}+\|\dv(f_u,f_v)\|_{L^2},\\
\|\pa_Y^2u,\pa_Y^2v)\|_{L^2}&
\lesssim \frac{n^{\f12}}{(\text{Im}c)^{\f12}}\left(1+\frac{1}{n^{\f12}(\text{Im}c)^{\f32}} \right)\|\Omega(f_u,f_v)\|_{L^2_w}+\frac{1}{\text{Im}c}\left(\f1{\a}\|f_u\|+\|f_v\|_{L^2}+\|\dv(f_u,f_v)\|_{L^2}\right)\\
&\lesssim \frac{n^{\f12}}{(\text{Im}c)^{\f12}}\|\Omega(f_u,f_v)\|_{L^2_w}+\frac{1}{\text{Im}c}\left(\f1{\a}\|f_u\|+\|f_v\|_{L^2}+\|\dv(f_u,f_v)\|_{L^2}\right).
\end{aligned}
$$
Thus,  the improved estimates \eqref{3.0.5} and \eqref{3.0.6} are proved. And this completes the
 proof of the proposition. \qed
\section{Proof of Theorem \ref{thm1.1}}
Finally, in this section, we prove Theorem \ref{thm1.1}. 
We construct the solution to linearized system \eqref{1.4} with no-slip boundary condition \eqref{1.4-1} in the following form
\begin{align}\label{4.0}
\vec{\Xi}(Y;c)=\vec{\Xi}_{\text{app}}(Y;c)-\vec{\Xi}_{\text{sm}}(Y;c)-\vec{\Xi}_{\text{re}}(Y;c),
\end{align}
Here $\vec{\Xi}_{\text{app}}$ is the approximate solution obtained in \eqref{2.3.1} which satisfies \eqref{2.4.1}, $\vec{\Xi}_{\text{sm}}=(\rho_{\text{sm}},u_{\text{sm}},v_{\text{sm}})$ and $\vec{\Xi}_{\text{re}}=(\rho_{\text{re}},u_{\text{re}},v_{\text{re}})$ solve the remainder system
\begin{align}
\CL(\vec{\Xi}_{\text{sm}})=(0,E_{u,\text{sm}},E_{v,\text{sm}}),~ v_{\text{sm}}|_{Y=0}=0,\nonumber
\end{align}
and 
\begin{align}
\CL(\vec{\Xi}_{\text{re}})=(0,0,E_{v,\text{re}}),~ v_{\text{re}}|_{Y=0}=0,\nonumber
\end{align}
respectively. By Proposition \ref{prop2.2} and \ref{prop3.0}, both $\vec{\Xi}_{\text{sm}}$ and $\vec{\Xi}_{\text{re}}$ are well-defined. Moreover,
it is straightforward to check that $\vec{\Xi}=(\rho,u,v)$ satisfies
$$\CL(\vec{\Xi})=\vec{0},~ v|_{Y=0}=0.
$$

To recover the  no-slip boundary condition on the tangential component, we introduce the mapping
\begin{align}
\mathcal{F}: D_0\rightarrow \mathbb{C},~ \mathcal{F}(c)\eqdef u(0;c)=\mathcal{F}_{\text{app}}(c)-u_{\text{sm}}(0;c)-u_{\text{re}}(0;c).\nonumber
\end{align}

On one hand, from Proposition \ref{prop2.1},  $\mathcal{F}_{\text{app}}(c)$ is analytic and has a unique zero point in $D_0$.  On the other hand, according to Remark \ref{rmk3.1} (a), both $u_{\text{sm}}(0;c)$ and $u_{\text{de}}(0;c)$ are analytic in $D_0$. Then by applying  estimates \eqref{3.0.2}, \eqref{3.0.3} to $\vec{\Xi}_{\text{sm}}$ with $(f_u,f_v)=(E_{u,\text{sm}},E_{v,\text{sm}})$, using the bound in \eqref{2.4.5} and the Sobolev inequality,
we deduce that 
\begin{align}
|u_{\text{sm}}(0;c)|&\leq \|u_{\text{sm}}(\cdot~;c)\|_{L^2}^{\f12}\|\pa_Yu_{\text{sm}}(\cdot~;c)\|_{L^2}^{\f12}\nonumber\\
&\leq \frac{Cn^{\f14}}{\a(\text{Im}c)^{\f54}}\|(E_{u,\text{sm}},E_{v,\text{sm}})\|_{L^2}
\leq \frac{Cn^{\f14}\vep^{\f7{16}}}{\a(\text{Im}c)^{\f54}}\leq C\vep^{\f1{16}},~ \forall c\in D_0.\label{4.1}
\end{align}
Here we have used \eqref{para} in the last inequality. For $u_{\text{re}}(0;c)$, we use the bounds given in  \eqref{3.0.5} for $\|u_{\text{re}}\|_{H^1}$ with  $(f_u,f_v)=(0,E_{v,\text{re}})$ and \eqref{2.4.4} to have
\begin{align}\label{4.2}
|u_{\text{re}}(0;c)|&\leq C\|u_{\text{re}}(\cdot~;c)\|_{H^1}\leq \frac{C}{\text{Im}c}\|E_{v,\text{re}}\|_{L^2_w}+\|E_{v,\text{re}}\|_{H^1}\nonumber\\
&\leq C\left(1+\frac{1}{\text{Im}c}\right)\vep^{\f3{16}}\leq C\vep^{\f{1}{16}}.
\end{align}
Thus, by recalling   the lower bound of $|\mathcal{F}_{\text{app}}(c)|$ on the circle $\pa D_0$ 
 in \eqref{2.3.5}, and by using the bounds in \eqref{4.1} and \eqref{4.2}, it holds that
$$\forall c\in \pa D_0,~|\mathcal{F}(c)-\mathcal{F}_{\text{app}}(c)|\leq |u_{\text{re}}(0;c)|+|u_{\text{sm}}(0;c)|\leq C\vep^{\f1{16}}\leq \f14{K^{-\theta}}\leq \frac{1}{2}|\mathcal{F}_{\text{app}}(c)|,
$$
by taking $\vep\in (0,1)$ suitably small. Therefore, by Rouch\'e's Theorem, $\CF(c)$ and $\CF_{\text{app}}(c)$ have the same number of zero points in $D_0$. This justifies the existence of 
a unique $c\in D_0$ such that $\vec{\Xi}(Y;c)$ defined in \eqref{4.0} solves the linear equation \eqref{1.4} with the no-slip boundary condition \eqref{1.4-1}. The proof of Theorem \ref{thm1.1} is completed. 
\qed\\

\noindent{\bf Acknowledgment:} The research of T. Yang was supported by the General Research Fund of Hong Kong, CityU 11303521.

\section{Appendix}
Recall \eqref{3.1.5} the definition of operator $\Lambda$.
\begin{lemma}[The operator $\Lambda^{-1}$]\label{lm4.1} Let $\cm\in (0,1)$ and $\a\in (0,1)$. For small $|c|\ll1$, if $\|h\|_{L^2}<\infty$, there exists a unique solution $\psi\in H^2(\mathbb{R}_+)\cap H^1_0(\mathbb{R}_+)$ to $\Lambda(\psi)=h$ in $\mathbb{R}_+$, $\psi|_{Y=0}=0$ which satisfies the following estimates
	\begin{align}\label{ap1}
	\|\pa_Y^2\psi\|_{L^2}+\|(\pa_Y\psi,\a\psi)\|_{L^2}\lesssim \min\{\a^{-1}\|h\|_{L^2},\|(1+Y)h\|_{L^2}\}.
	\end{align}
	Moreover, the solution mapping $\Lambda^{-1}(\cdot~;c): L^2(\mathbb{R}_+)\rightarrow H^2(\mathbb{R}_+)\cap H^1_0(\mathbb{R}_+)$ is analytic in $c$.
\end{lemma}
\begin{proof}
	We first establish the a priori estimate \eqref{ap1}. Denote $\langle\cdot,\cdot\rangle$ the standard inner product in $L^2(\mathbb{R}_+)$.
	Taking inner product with $-\b{\psi}$ in the equation $\Lambda(\psi)=h$ gives that
	\begin{align}\label{ap2}
	\int_0^\infty A^{-1}|\pa_Y\psi|^2+\a^2|\psi|^2\dd Y=\langle h,-\b{\psi}\rangle.
	\end{align}
	Note that
	\begin{align}\label{ap5} A^{-1}=(1-\cm^2U_s^2)^{-1}+O(1)|c|.
	\end{align}Then for suitably small $|c|$, the real part of \eqref{ap2}  yields
	\begin{align}
	\|\pa_Y\psi\|_{L^2}^2+\a^2\|\psi\|_{L^2}^2\leq C|\text{Re}\langle h,\b{\psi}\rangle|\leq C \|h\|_{L^2}\|\psi\|_{L^2}\leq C \a^{-2}\|h\|_{L^2}^2,\nonumber
	\end{align}	
	which gives $H^1$-estimate of $\psi$. The estimate of $\pa_Y^2\psi$ can be obtained by using
	the equation to have
	\begin{align}
	\|\pa_Y^2\psi\|_{L^2}&\lesssim\nonumber \|A^{-1}\pa_YA\|_{L^\infty}\|\pa_Y\psi\|_{L^2}+\a^2\|A\psi\|_{L^2}+\|Ah\|_{L^2}\\
	&\lesssim \|(\pa_Y\psi,\a\psi)\|_{L^2}+\|h\|_{L^2}\lesssim \a^{-1}\|h\|_{L^2}.\label{ap4}
	\end{align}
	
	If in addition,  $(1+Y)h\in L^2(\mathbb{R}_+)$, we can obtain $|\langle h,\b{\psi} \rangle|\leq \|Y^{-1}\psi\|_{L^2}\|Yh\|_{L^2}\leq C\|\pa_Y\psi\|_{L^2}\|Yh\|_{L^2}$ by using the Hardy inequality for the term $Y^{-1}\psi$. Thus real part of \eqref{ap2} yields the $H^1$-estimate \begin{align}
	\|(\pa_Y\psi,\a\psi)\|_{L^2}\leq C\|Yh\|_{L^2}.\label{ap3}
	\end{align}
	The estimate on $\pa_Y^2\psi$ then follows from the equation and  $H^1$-estimate \eqref{ap3} 
	similar to the estimation in  \eqref{ap4}. Hence the estimate \eqref{ap1} holds. The uniqueness of solution follows from the a priori estimate \eqref{ap1}. 
	
	Set
	$$\Lambda_0: H^2(\mathbb{R}_+)\cap H^1_0(\mathbb{R}_+)\rightarrow L^2(\mathbb{R}_+),~\Lambda_0(\psi)=\pa_Y\left[(1-\cm^2U_s^2)^{-1}\pa_Y\psi \right]-\a^2\psi.$$
	For $\cm\in (0,1)$, the  operator $\Lambda_0$ is  clearly uniformly elliptic and  invertible in $L^2(\mathbb{R}_+)$. Then by  \eqref{ap5}, the existence and analytic dependence on $c$ of $\Lambda^{-1}$ follow from the standard 
	argument. Hence, we omit the details for brevity. The proof of the lemma is completed.
\end{proof}

\end{document}